\documentclass[12pt]{article}
\usepackage[utf8]{inputenc}
\usepackage[T1]{fontenc}
\usepackage[english]{babel} 
\usepackage{csquotes}
\usepackage{charter}
\usepackage{amsmath}
\usepackage{amssymb}
\usepackage{amsthm}
\usepackage[nottoc]{tocbibind}
\usepackage{graphicx}
\usepackage{stmaryrd}
\usepackage{accents}
\usepackage[top=2.5cm, bottom=2cm, left = 2.5cm, right= 2.5cm]{geometry}
\usepackage{amsfonts}
\usepackage{dsfont}
\usepackage{authblk}
\usepackage{xcolor}
\usepackage[linesnumbered,ruled,vlined]{algorithm2e}
\usepackage{ifthen}
\usepackage{hyperref}
\usepackage[all]{hypcap}
\usepackage{aliascnt}
\usepackage[block = none,style=authoryear,natbib=true,maxbibnames=9,maxcitenames=2,uniquelist=false,backend = biber]{biblatex}

\usepackage[linecolor=blue!60!,backgroundcolor=blue!10!,textwidth=2.5cm,textsize=scriptsize]{todonotes}

\hypersetup{pdftitle={Confidence envelopes for the false discoveries with heterogeneous data},            pdfauthor={Romain Périer, Gilles Blanchard, Sebastian Döhler, Guillermo Durand, Etienne Roquain},          unicode=true,           colorlinks=true,       linkcolor = purple,
  citecolor=blue,      }

\SetKwInput{KwInput}{Input}
\SetKwInput{KwOutput}{Output}
\SetKwInput{KwReturn}{Return}
\newcommand{\func}[5][]{
\ifthenelse{\equal {#1}{}}
{
\left\{\begin{array}{rcl}
 #2 & \longrightarrow & #3 \\
 #4 & \longmapsto & #5 \end{array}\right.
}
{#1:\left\{\begin{array}{rcl}
 #2 & \longrightarrow & #3 \\
 #4 & \longmapsto & #5 \end{array}\right.}
 }

\DeclareCiteCommand{\cite}
  {\usebibmacro{prenote}}
  {\usebibmacro{citeindex}\printtext[bibhyperref]{\usebibmacro{cite}}}
  {\multicitedelim}
  {\usebibmacro{postnote}}

\DeclareCiteCommand*{\cite}
  {\usebibmacro{prenote}}
  {\usebibmacro{citeindex}\printtext[bibhyperref]{\usebibmacro{citeyear}}}
  {\multicitedelim}
  {\usebibmacro{postnote}}

\DeclareCiteCommand{\parencite}[\mkbibparens]
  {\usebibmacro{prenote}}
  {\usebibmacro{citeindex}\printtext[bibhyperref]{\usebibmacro{cite}}}
  {\multicitedelim}
  {\usebibmacro{postnote}}

\DeclareCiteCommand*{\parencite}[\mkbibparens]
  {\usebibmacro{prenote}}
  {\usebibmacro{citeindex}\printtext[bibhyperref]{\usebibmacro{citeyear}}}
  {\multicitedelim}
  {\usebibmacro{postnote}}

\DeclareCiteCommand{\footcite}[\mkbibfootnote]
  {\usebibmacro{prenote}}
  {\usebibmacro{citeindex}\printtext[bibhyperref]{ \usebibmacro{cite}}}
  {\multicitedelim}
  {\usebibmacro{postnote}}

\DeclareCiteCommand{\footcitetext}[\mkbibfootnotetext]
  {\usebibmacro{prenote}}
  {\usebibmacro{citeindex}\printtext[bibhyperref]{\usebibmacro{cite}}}
  {\multicitedelim}
  {\usebibmacro{postnote}}

\DeclareCiteCommand{\textcite}
 {\boolfalse{cbx:parens}}
  {\usebibmacro{citeindex}\printtext[bibhyperref]{\printnames{labelname}\printtext{ (\usebibmacro{citeyear}\printtext{)}}}}
 {\ifbool{cbx:parens}
  {\bibcloseparen\global\boolfalse{cbx:parens}}
  {}\multicitedelim}
{\usebibmacro{textcite:postnote}}

\let\originalleft\left
\let\originalright\right
\renewcommand{\left}{\mathopen{}\mathclose\bgroup\originalleft}
\renewcommand{\right}{\aftergroup\egroup\originalright}

\RequirePackage{xstring}

\newcommand{\paren}[2][a]{\IfEqCase{#1}{{a}{\left(#2\right)}{0}{(#2)}{1}{\big(#2\big)}{2}{\Big(#2\Big)}{3}{\bigg(#2\bigg)}{4}{\Bigg(#2\Bigg)}}[\PackageError{paren}{Undefined option to paren: #1}{}]}
\newcommand{\norm}[2][a]{\IfEqCase{#1}{{a}{\left\lVert#2\right\rVert}{0}{\lVert#2\rVert}{1}{\big\lVert#2\big\rVert}{2}{\Big\lVert#2\Big\rVert}{3}{\bigg\lVert#2\bigg\rVert}{4}{\Bigg\lVert#2\Bigg\rVert}}[\PackageError{norm}{Undefined option to norm: #1}{}]}
\newcommand{\brac}[2][a]{\IfEqCase{#1}{{a}{\left[#2\right]}{0}{[#2]}{1}{\big[#2\big]}{2}{\Big[#2\Big]}{3}{\bigg[#2\bigg]}{4}{\Bigg[#2\Bigg]}}[\PackageError{brac}{Undefined option to brac: #1}{}]}
\newcommand{\inner}[2][a]{\IfEqCase{#1}{{a}{\left\langle#2\right\rangle}{0}{\langle#2\rangle}{1}{\big\langle#2\big\rangle}{2}{\Big\langle#2\Big\rangle}{3}{\bigg\langle#2\bigg\rangle}{4}{\Bigg\langle#2\Bigg\rangle}}[\PackageError{inner}{Undefined option to inner: #1}{}]}

\newcommand{\floor}[2][a]{\IfEqCase{#1}{{a}{\left\lfloor#2\right\rfloor}{0}{\lfloor#2\rfloor}{1}{\big\lfloor#2\big\rfloor}{2}{\Big\lfloor#2\Big\rfloor}{3}{\bigg\lfloor#2\bigg\rfloor}{4}{\Bigg\lfloor#2\Bigg\rfloor}}[\PackageError{floor}{Undefined option to floor: #1}{}]}

\newcommand{\ceil}[2][a]{\IfEqCase{#1}{{a}{\left\lceil#2\right\rceil}{0}{\lceil#2\rceil}{1}{\big\lceil#2\big\rceil}{2}{\Big\lceil#2\Big\rceil}{3}{\bigg\lceil#2\bigg\rceil}{4}{\Bigg\lceil#2\Bigg\rceil}}[\PackageError{ceil}{Undefined option to ceil: #1}{}]}

\newcommand{\abs}[2][a]{
\IfEqCase{#1}{{a}{\left\vert#2\right\rvert}{0}{\vert#2\rvert}{1}{\big\vert#2\big\rvert}{2}{\Big\vert#2\Big\rvert}{3}{\bigg\vert#2\bigg\rvert}{4}{\Bigg\vert#2\Bigg\rvert}}[\PackageError{abs}{Undefined option to abs: #1}{}]}

\newcommand{\set}[2][a]{
\IfEqCase{#1}{{a}{\left\{#2\right\}}{0}{\{#2\}}{1}{\big\{#2\big\}}{2}{\Big\{#2\Big\}}{3}{\bigg\{#2\bigg\}}{4}{\Bigg\{#2\Bigg\}}}[\PackageError{set}{Undefined option to set: #1}{}]}

\newcommand{\intr}[1]{\llbracket #1 \rrbracket}
\newcommand{\wh}[1]{{\widehat{#1}}}

 \newcommand{\Pro}[2][a]{\mathbb{P}\paren[#1]{#2}} 
\newcommand{\Proo}[3][a]{\mathbb{P}_{#3}\paren[#1]{#2}}
\newcommand{\Esp}[2][a]{\mathbb{E}\brac[#1]{#2}}

\newcommand{\ind}[1]{\mathds{1}_{#1}}
\newcommand{\indev}[2][a]{\mathds{1}{\set[#1]{#2}}}
\newcommand{\cH}{\mathcal{H}}
\newcommand{\cB}{\mathcal{B}}
\newcommand{\cK}{\mathcal{K}}
\newcommand{\cA}{\mathcal{A}}
\newcommand{\cP}{\mathcal{P}}
\newcommand{\cN}{\mathcal{N}}

\newcommand{\cU}{\mathcal{U}}

\newcommand{\cdf}{F}
\newcommand{\Nm}{\left\llbracket1,m\right\rrbracket}
\newcommand{\NM}{\left\llbracket0, m\right\rrbracket}
\newcommand{\NS}{\left\llbracket0,|S|\right\rrbracket}
\newcommand{\intrN}{\left\llbracket1,N\right\rrbracket}

\newcommand{\Nmz}{\left\llbracket1,m_0\right\rrbracket}
\newcommand{\Nk}{\left\llbracket1,k\right\rrbracket}
\newcommand{\NN}{\left\llbracket1,N\right\rrbracket}
\newcommand{\Rfam}{\mathfrak{R}}
\newcommand{\Afam}{\mathfrak{A}}

\newcommand{\nat}{\mathbb{N}}

\newcommand{\RR}{\mathbb R}
\newcommand{\statfam}{\mathfrak{F}}
\newcommand{\Vhat}{\widehat{V}}
\newcommand{\Dhat}{\widehat{D}}

\newcommand{\mzsc}{\hat{m}_0^{\SCo}}

\DeclareMathOperator*{\argmin}{arg\,min}
\DeclareMathOperator{\FDR}{FDR}
\DeclareMathOperator{\FDP}{FDP}
\DeclareMathOperator{\FDX}{FDX}
\DeclareMathOperator{\FWER}{FWER}
\DeclareMathOperator{\JER}{JER}
\DeclareMathOperator{\GHS}{GHS}
\DeclareMathOperator{\SCo}{SC1}
\DeclareMathOperator{\SCt}{SC2}
\DeclareMathOperator{\MBR}{MBR}
\DeclareMathOperator{\IP}{IP}

\def\NewTheorem#1{\newaliascnt{#1}{Theoremm}
  \newtheorem{#1}[#1]{#1}
  \aliascntresetthe{#1}
  \expandafter\def\csname #1autorefname\endcsname{#1}
}
\NewTheorem{Theorem}
\NewTheorem{Definition}
\NewTheorem{Proposition-Definition}
\NewTheorem{Lemma}
\NewTheorem{Example}
\NewTheorem{Corollary}
\NewTheorem{Proposition}
\NewTheorem{Remark}

\expandafter\def\csname Modelautorefname\endcsname{Model}
\newtheorem{Assumption}{Assumption}
\expandafter\def\csname Assumptionautorefname\endcsname{Assumption}

\expandafter\def\csname Procedureautorefname\endcsname{Procedure}

\expandafter\def\csname Notationautorefname\endcsname{Notation}
\expandafter\def\csname Subsectionautorefname\endcsname{Subsection}
\addto\extrasenglish{

}
\newcommand{\secref}[1]
{\hyperref[#1]{Section~\ref{#1}}}
\newcommand{\appref}[1]
{\hyperref[#1]{Appendix~\ref{#1}}}
\newcommand{\equaref}[1]
{\hyperref[#1]{Equation~\eqref{#1}}}
\newcommand{\algoref}[1]
{\hyperref[#1]{Algorithm~\ref{#1}}}

\newcommand{\defn}[2][]{
\ifthenelse{\equal {#1}{}}
    {\begin{Definition}
        #2
    \end{Definition}}
        {\begin{Definition}[#1]
        #2
    \end{Definition}}
}

\newcommand{\defnl}[3][]{
\ifthenelse{\equal {#1}{}}
    {\begin{Definition}\label{#3}
        #2
    \end{Definition}}
        {\begin{Definition}[#1]\label{#3}
        #2
    \end{Definition}}
}

\newcommand{\thm}[2][]{
\ifthenelse{\equal {#1}{}}
    {\begin{Theorem}
        #2
    \end{Theorem}}
        {\begin{Theorem}[#1]
        #2
    \end{Theorem}}
}
\newcommand{\thml}[3][]{
\ifthenelse{\equal {#1}{}}
    {\begin{Theorem}\label{#3}
        #2
    \end{Theorem}}
        {\begin{Theorem}[#1]\label{#3}
        #2
    \end{Theorem}}
}

\newcommand{\prop}[2][]{
\ifthenelse{\equal {#1}{}}
    {\begin{Proposition}
        #2
    \end{Proposition}}
        {\begin{Proposition}[#1]
        #2
    \end{Proposition}}
}

\newcommand{\propl}[3][]{
\ifthenelse{\equal {#1}{}}
    {\begin{Proposition}\label{#3}
        #2
    \end{Proposition}}
        {\begin{Proposition}[#1]\label{#3}
        #2
    \end{Proposition}}
}

\newcommand{\rem}[2][]{
\ifthenelse{\equal {#1}{}}
    {\begin{Remark}
        #2
    \end{Remark}}
        {\begin{Remark}[#1]
        #2
    \end{Remark}}
}

\newcommand{\reml}[3][]{
\ifthenelse{\equal {#1}{}}
    {\begin{Remark}\label{#3}
        #2
    \end{Remark}}
        {\begin{Remark}[#1]\label{#3}
        #2
    \end{Remark}}
}

\newcommand{\lem}[2][]{
\ifthenelse{\equal {#1}{}}
	{
	\begin{Lemma}
	#2
	\end{Lemma}
	}
	{
	\begin{Lemma}[#1]
	#2
	\end{Lemma}
	}
}

\newcommand{\leml}[3][]{
\ifthenelse{\equal {#1}{}}
	{
	\begin{Lemma}\label{#3}
	#2
	\end{Lemma}
	}
	{
	\begin{Lemma}[#1]\label{#3}
	#2
	\end{Lemma}
	}
}

\newcommand{\cor}[2][]{
\ifthenelse{\equal {#1}{}}
	{
	\begin{Corollary}
	#2
	\end{Corollary}
	}
	{
	\begin{Corollary}[#1]
	#2
	\end{Corollary}
	}
}

\newcommand{\corl}[3][]{
\ifthenelse{\equal {#1}{}}
	{
	\begin{Corollary}\label{#3}
	#2
	\end{Corollary}
	}
	{
	\begin{Corollary}[#1]\label{#3}
	#2
	\end{Corollary}
	}
}

\newcommand{\ex}[2][]{
\ifthenelse{\equal {#1}{}}
    {\begin{Example}
        #2
    \end{Example}}
        {\begin{Example}[#1]
        #2
    \end{Example}}
}
\newcommand{\exl}[3][]{
\ifthenelse{\equal {#1}{}}
    {\begin{Example}\label{#3}
        #2
    \end{Example}}
        {\begin{Example}[#1]\label{#3}
        #2
    \end{Example}}
}

\newcommand{\ass}[2][]{
\ifthenelse{\equal {#1}{}}
    {\begin{Assumption}
        #2
    \end{Assumption}}
        {\begin{Assumption}[#1]
        #2
    \end{Assumption}}
}
\newcommand{\assl}[3][]{
\ifthenelse{\equal {#1}{}}
    {\begin{Assumption}\label{#3}
        #2
    \end{Assumption}}
        {\begin{Assumption}[#1]\label{#3}
        #2
    \end{Assumption}}
}

\title{Confidence envelopes for the false discoveries with heterogeneous data}

\author[1,2]{Romain Périer}
\author[1,3]{Gilles Blanchard}
\author[4]{Sebastian Döhler}
\author[1,2]{Guillermo Durand}
\author[5]{Etienne Roquain}

\affil[1]{Laboratoire de Mathématiques d'Orsay (LMO), Université Paris-Saclay, CNRS}
\affil[2]{INRIA Celeste}
\affil[3]{INRIA DataShape}
\affil[4]{Hochschule Darmstadt, University of Applied Sciences}
\affil[5]{Laboratoire de Probabilités, Statistique et Modélisation (LPSM), Sorbonne Université, CNRS}

\date{\today}
 
\AtEveryBibitem{
\clearfield{doi}
  \clearfield{url}
  \clearfield{issn}
  \clearfield{month}
  \clearfield{editor}
  \clearfield{ISBN}
}

\begin{document}

\maketitle

\begin{abstract}

In the context of selective inference, confidence envelopes for the false discoveries allow the user to select any subset of null hypotheses while having a statistical guarantee on the number of false discoveries in the selected set. Many constructions of such envelopes have been proposed recently, using local test families \citep{Genovese2006,goeman2011multiple}, paths \citep{Katsevichramdas2020} or interpolation \citep{blanchard2020post}. All those methods have in common that they have been well-studied for the homogeneous case where all $p$-values under the null have a uniform distribution over $[0,1]$. However, in many applications the data are heterogeneous and discrete, hence
the $p$-values have heterogeneous, discrete distributions, and the previous constructions may  incur a loss of power, in the sense that they over-estimate the number of false discoveries.
In this paper, we  bridge the previous constructions under the homogeneous case with new tools. We also apply these tools to
propose
several 
confidence envelopes based on tools tailored for heterogeneous data, like the Bretagnolle inequality, or a new variant of the Simes inequality.
We compare these new envelopes to their homogeneous counterparts on simulated
data.
\end{abstract}

\tableofcontents

\vspace{1cm}

\section{Introduction}

\subsection{Overview of goals and motivations}

We describe the aims of the paper, approach and contributions.
The objects used will be more formally defined in \secref{sec_notation}
but should be familiar to a reader already acquainted with the 
multiple testing literature.

Let $(H_i)_{i \in \Nm}$ be a family of null statistical hypotheses to test; each null hypothesis corresponds to a subset
of probability distributions driving the observed data. 
For a data distribution $\mu$, we say that ``$\mu$ satisfies the
null hypothesis $H_i$'' if $\mu \in H_i$, and we denote $\cH_0(\mu)$, immediately
abbreviated to $\cH_0$, the set of true nulls, i.e., 
indices $i$ of null hypotheses
$H_i$ satisfied by $\mu$. Based on the data, the user outputs a subset of selected
hypotheses $\wh{S}\subseteq \Nm$. A common goal in multiple testing
is to return an upper bound $\Vhat(\wh{S})$ on the number of false discoveries (FD)
$\abs[0]{\wh{S} \cap \cH_0}$. In the present paper, we focus
on bounds holding with high probability, whatever the true
$\mu$ (belonging to some overarching statistical model $\statfam$).

Observe that $\wh{S}$ is a random set whose dependence on the data can be complex and not fully known, especially if the user departs from
a precisely prescribed algorithm to engage in various forms of data snooping. For this reason, an approach considered as more statistically sound than assuming a specific form for $\wh{S}$
is to establish a family of upper bounds $(\Vhat(S))_{S \subseteq \Nm}$
holding (with guaranteed large probability) simultaneously for all potential selection sets $S \subseteq \Nm$.
Such a family of bounds will be called a {\em FD confidence envelope} in
the rest of the present paper.
(It is also sometimes called a \textit{post hoc} multiple testing bound since it entails
a guarantee in particular for any data-dependent $\wh{S}$.)

Several approaches for constructing FD confidence envelopes have been
proposed in the literature. A particularly important device
to do so, also playing a central role in the present work, is the \textit{inversion procedure} introduced by~\citet{Genovese2006}. For any subset $A \subseteq \Nm$, we denote $H_A$ the set of distributions simultaneously satisfying $H_i, i \in A$; also called intersection hypothesis, since $H_A = \bigcap_{i\in A} H_i$. Assume that for any subset $A$, there
exists a {\em local test} $\varphi_A \in \set{0,1}$ (the value 1 indicating rejection) such that if $\mu$ satisfies $H_A$, then
the probability that the local test falsely rejects (local type I error)
is bounded by some prescribed constant $\delta$.
Then, $V_\varphi(S) = \max_{A: \varphi_A = 0} \abs{A \cap S}$ is a FD confidence envelope with probability $1-\delta$. (An equivalent bound based on the notion of
closed testing associated to the family of local tests was proposed by~\citealp{goeman2011multiple}.)

This result shifts the problem to two questions: how to construct appropriate
local tests and, for given local tests and rejection set $S$, how to compute efficiently the associated inversion bound $V_\varphi(S)$.
Note that computing $\varphi_A$ for all subsets $A\subseteq \Nm$ is exponential in $m$ and typically infeasible. Sometimes
one may be contented with an efficiently computable upper bound on $V_\varphi(S)$, also called \textit{shortcut}.

As is common, we assume that the data has already been processed in such a way that for each hypothesis
$H_i$, an observable $p$-value $p_i$ has been constructed. Recall
that the fundamental property of a $p$-value associated
to a null hypothesis $H$ is that, for any data distribution
satisfying $H$, the random variable $p$ has a
{\em super-uniform} distribution (i.e., stochastically larger
or equal to a uniform random variable on $[0,1]$). 
If we furthermore assume that $p$-values $p_i$ for 
$i \in \cH_0$ are
independent (also a common assumption), it is possible to construct
a local test $\varphi_A$ derived from an envelope on the empirical cumulative distribution function (ecdf) of
i.i.d. uniform random variables. 
Namely, assume that $\cB_n(t)$ is a $(1-\delta)$-upper confidence envelope on the ecdf $\wh{F}_n(t) = \frac{1}{n} \sum_{i} \indev{U_i \leq t}$, if 
$(U_i)_{i \in \intr{n}}$ are i.i.d. uniform $[0,1]$, that is, it holds $\wh{F}_n(t) \leq \cB_n(t)$ for all $t \in [0,1]$, with probability $(1-\delta)$.
Then the test $\varphi_A$ which rejects if 
$\wh{F}_A(t) = \frac{1}{|A|} \sum_{i \in A} \indev{p_i \leq t} > \cB_{|A|}(t)$ for any $t$ 
is a local
test for $H_A$. This principle of (i.i.d.) uniformity testing
based on detecting violations of confidence envelopes on the ecdf is long-standing in statistics, from the Kolmogorov-Smirnov test (having asymptotic validity)
 to the more recent higher criticism (HC) tests
in \citet{Genovese2006}. Various
developments on this idea, based on more refined
envelopes, can be found in \citet{MBR2024}.

The approach described above, constructing local tests from 
upper envelopes for
ecdf of uniform random variables, also applies to super-uniform null $p$-values because of a simple coupling device described in \secref{sec_homo_concentration}. 

However, in many situations, such as when $p$-values
are discrete, the exact cdf $F_i$ of $p$-value $p_i$
under the null is exactly known; in such a scenario,
using ``default'' confidence envelopes holding for uniform variables may result in a loss of power of the associated local test. This is the starting motivation of this work: we first concentrate on deriving
sharp confidence envelopes for ecdf of 
such variables (independent, but with different
known distributions $F_i$, a setting we call {\em heterogeneous}), and then turn to
constructing FD bounds that can
be derived from them. A point of particular attention
is to address the need for computable bounds (or shortcuts thereof), an issue which becomes thornier in the heterogeneous
case. Namely, while homogeneous local ecdf confidence envelopes $\cB_{\abs{A}}(t)$ only depend on the cardinality $|A|$
of the considered subset of hypotheses, their heterogeneous
counterpart typically depend on the whole set $A$. 
Driven by this need of computable bounds, we will also explore the alternative approach of \citet{blanchard2020post} to construct FD bounds based on the control of a Joint Error Rate (JER) on a specific reference family of subsets of hypotheses. This approach has, until now, come in 3 different flavors
depending on the reference family, namely: $k$-FWER, top-$k$, and with deterministic regions. They will all be re-introduced in \secref{sec_homo_JER} and used in our contributions later.

\subsection{Confidence envelopes in the context of selective inference}

As mentioned before, data snooping (also called $p$-hacking, or significance chasing) plays a critical role in modern day applications and is considered one of the main causes of the replication crisis (see \citealp[``Principle 4'']{MR3511040}, and \citealp{MR2216666}). Indeed, classical statistical guarantees generally do not hold when the dataset has undergone a first selection step based on the observation of the data itself, but the classical method that has been used (whichever it is) does not take into account this random selection step. For instance, think of the Benjamini-Hochberg (BH) procedure applied to $10^6$ tests, but where only the $10^3$ smallest $p$-values have been selected by the data analyst as input to the procedure. Contrary to the plain BH procedure, there is no guarantee on the False Discovery Rate control of such a two-step procedure. While the flaw is obvious in this example, various, less blatant flavors of data-dependent selection often make their way into statistical practice
while seeming innocuous to the practitioner, for example selection based on a different statistic which is not completely independent of the test statistic.

Selective inference is the branch of statistics that aims to answer this concern by providing methods and procedures with proven statistical guarantees when a first selection step occurred. Selective inference is generally tackled in two possible fashions. One consists in assuming that the selection procedure is known (for example, coming from the LASSO estimator) and to design valid inference by conditioning on the selection event, see for example \citet{doi:10.1073/pnas.1507583112,lee2016exact,fithian2017optimalinferencemodelselection}.

The other fashion adopts a ``simultaneous'' point of view, where we seek to achieve guarantees for all possible selection subsets. This has been investigated for inference in, for instance, \citet{scheffe1953, berk2013, bachoc2019}. For multiple testing purposes, this fashion takes the form of the construction of the aforementioned confidence envelopes $\paren[1]{\Vhat(S)}_{S \subseteq \Nm}$, first introduced by \citet{Genovese2006}. Since then, many works have focused on such constructions, with, notably, a prolific series of work focused on closed testing \citep{goeman2011multiple, MR3305943, MR3992394, goeman2021only, MR4731977}, and another series based on the JER approach \citep{blanchard2020post, durand2020post, 10.1093/bioinformatics/btac693, blain22notip, NEURIPS2023_f6712d51, MBR2024}. Note that the inversion procedure and the closed testing approach, both based on local tests, are recalled in \secref{sec_homo_inversion}, while the JER approach is recalled in \secref{sec_homo_JER}.

Confidence envelopes have notably been applied to gene expression analysis \citep{goeman2011multiple, hemerik2018, 10.1093/bioinformatics/btac693} while software coded with the \texttt{R} language \citep{R-base} is already available for such analyses \citep{IIDEA} in a user-friendly \texttt{shiny} application \citep{shiny}. Confidence envelopes have also been applied a lot in functional Magnetic Resonance Imaging (fMRI) studies \citep{blain22notip, NEURIPS2023_f6712d51, MR4731977}.

\subsection{Contributions and paper organization}

In \secref{sec_notation}, we give the formalism, notation and assumptions used in the remainder of the paper.

In \secref{sec_consolidation}, we revisit past works on confidence bounds construction while consolidating this theory with several contributions, including:
\begin{itemize}
    \item refined versions of known concentration inequalities,
\item extension of past works in the top-$k$ setting to super-uniformity (see Assumption \ref{ass_super_unif}),
    \item a new dynamic programming algorithm for confidence bound computation in the top-$k$ setting.
\end{itemize}

In \secref{sec_general_shortcuts}, we first provide new generic shortcuts for the inversion procedure, which are relevant
in particular when local tests detect violations of
uniform confidence envelopes.
These shortcuts are applied
in the homogeneous case to construct (over-)estimates of 
$m_0$. Such estimates can be plugged into
the JER methods, leading to so-called adaptive versions,
often leading to more powerful procedures that adapt to
to the true (but unknown) number of null hypotheses. 
Finally, we establish a novel formal equivalence 
in certain settings between 
the inversion procedure bound and this adaptive JER approach.

In \secref{sec_hetero_concentration}, we provide two 
ecdf confidence envelopes
valid under heterogeneity. The first one is a slight refinement of a known one, the second one is completely new. 
Using those as local tests, we apply to each of them the generic shortcuts from \secref{sec_general_shortcuts}; this allows us in turn to get estimates of $m_0$ under heterogeneity.

In \secref{sec_hetero_JER}, we demonstrate how to use the two previous
ecdf confidence envelopes
from \secref{sec_hetero_concentration} to construct new confidence bounds that use the JER formalism and that take the heterogeneity into consideration. In particular, we give constructions for the three JER settings: $k$-FWER, top-$k$, and deterministic regions. 
We also develop adaptive versions of our basic confidence bounds,
using the estimators of $m_0$ constructed in Section~\ref{sec_hetero_concentration}.

In \secref{sec_num}, we provide numerical experiments that demonstrate the power gain of our new constructions, with respect to their analogous past methods that do not take the heterogeneity into account. We do so with 
numerical simulations.

In \secref{sec_ccl}, we give some concluding remarks and perspectives.

Some complementary results are given in the appendix. In \autoref{appendix_VZ} we give a (previously known) third concentration inequality valid under heterogeneity  and the related confidence bounds, but they behave rather poorly in practice so they were excluded from the main body. In \autoref{appendix_gaps} we discuss and close some technical gaps in previous literature. In \autoref{appendix_ties}, we give complements about the top-$k$ setting, including links with another JER approach. In \autoref{appendix_random_pvalues}, we give a randomization technique to construct uniform random variables with coupling properties, which covers the classical ``randomized $p$-value'' construction. Finally, in \autoref{sec:KRcomplements}, we discuss slight refinements about past confidence envelopes in the top-$k$ setting, and \appref{sec_proof} is dedicated to the proofs.

\section{Multiple testing and confidence bound formalism}\label{sec_notation}

\subsection{Mathematical notation}

In the remainder, $m\geq 1$ and $N\geq 1$ are two positive integers. Specifically, $m$ will be used to denote the number of hypotheses tested, while $N$ is used when we need some generic integer. For any two $a,b\in\mathbb{N}$ with $a\leq b$, $\llbracket a, b\rrbracket$ is the integer interval $\set[a]{a, a+1,\dotsc,b-1, b}$.

The binary operator $\wedge$ denotes the minimum of two terms, and $\vee$ denotes the maximum of two terms.

For a vector $(x_1,\dotsc,x_N)\in\RR^N$, the notation $x_{(1)},\dotsc,x_{(N)}$ denotes the order statistics of $(x_1,\dotsc,x_N)$, such that $x_{(1)}\leq \dotsb \leq x_{(N)}$, and the notation $x_{[1]},\dotsc,x_{[N]}$ denotes the reverse-order statistics of $(x_1,\dotsc,x_N)$, such that $x_{[1]}\geq \dotsb \geq x_{[N]}$. In case of ties, any permutation that orders can be used. Finally, for $A\subseteq\llbracket 1, N\rrbracket$, the notation $x_{(1:A)},\dotsc,x_{(|A|:A)}$ denotes the order statistics of the elements $(x_i)_{i\in A}$ and, likewise, the notation $x_{[1:A]},\dotsc,x_{[|A|:A]}$ denotes the reverse-order statistics of the elements $(x_i)_{i\in A}$.

\subsection{Classical multiple testing formalism}

We let $\mathfrak{X}=(\mathcal{X},\Xi)$ a measurable space and $\statfam$ a model on $\mathfrak{X}$, that is a set of probability measures on $\mathfrak{X}$. We observe data belonging to $\mathcal{X}$ distributed according to a distribution that belongs to $\statfam$. Formally, this data is represented by a random variable often denoted $X$, defined on an abstract probability space $(\Omega, \Sigma, \mathbb P)$, such that the pushforward measure $X_{\sharp\mathbb P}$ is an element of $\statfam$. For a given $\mu\in\statfam$, we say that ``the law of $X$ is $\mu$'' or that ``$X$ is distributed according to $\mu$'', and we denote $X\sim\mu$, if $X_{\sharp\mathbb P}=\mu$.

For a given $\mu\in\statfam$, a statistic $T: \mathfrak{X}\to \paren[1]{\RR^N, \mathcal B\paren[0]{\RR^N}}$, and a Borelian set $A\in \mathcal B\paren{\RR^N}$, the notations $\Proo{T\in A}{\mu}$ and $\Proo{T(X)\in A}{\mu}$ are short for ``$\Pro[1]{T(X)\in A}$ for some $X$ with $X\sim\mu$''. The notation $\mathbb{E}_{\mu}$ is used in the same way for expectations.

We consider $m\geq1$ null hypotheses $H_{1}, \dotsc, H_{m}$ to test, which formally are submodels of $\statfam$, that is, $H_i\subseteq \statfam$ for all $i\in\Nm$. We denote by $\cH_0=\cH_0(\mu)$ (the dependence in $\mu$ will be dropped when there is no ambiguity) the set of indices of null hypotheses that are true, that is $\cH_0(\mu)=\set{i\in\Nm : \mu\in H_i}$. In other words, $H_i$ is true if and only if $i\in\cH_0$. The cardinal of $\cH_0$ is always denoted $m_0=m_0(\mu)=\abs{\cH_0(\mu)}$. We denote the complementary of $\cH_0$ in $\Nm$ by $\cH_1=\cH_1(\mu)$. Thus $\cH_1$ indexes the true signal that we would be able to ideally detect.

To test each $H_i, i\in\Nm$, we assume that have at hand a statistic which takes its values in $[0,1]$, named a $p$-value, and denoted $p_i$. When there is no ambiguity, we will identify the statistic $p_i$ and the random variable $p_i(X)$. We make the following two assumptions on the $p$-values. The first one is the classical requirement for a statistic to be considered a $p$-value (and induce a test with type-I error control). The second one is an independence assumption required for the application of our concentration inequalities.

\assl[Super-uniformity under the null]{For each $\mu\in\statfam$, for each $i\in\cH_0(\mu)$, for $X\sim\mu$, the distribution of $p_i(X)$ is super-uniform, which is often denoted by $p_i(X)\succeq\cU\paren[1]{[0,1]}$. This means that the cdf of $p_i(X)$ is upper bounded by the cdf of $\cU\paren[1]{[0,1]}$. Mathematically,
\begin{equation}
    \forall \mu\in\statfam, \forall i\in\cH_0(\mu), \forall x\in\RR,\; \Proo{p_i\leq x}{\mu} \leq 0 \wedge (x \vee 1).
    \label{eq_super_unif}
\end{equation}
}{ass_super_unif}

\assl[Independence under the null]{
For each $\mu\in\statfam$, the $p$ values $p_i(X)$, $i\in\cH_0(\mu)$ are mutually independent.
}{ass_indep}

\rem{
\autoref{ass_indep} is rather classical in multiple testing, albeit not totally realistic. In the literature, it is sometimes required that the family of $p$-values $\paren[1]{p_i(X)}_{i\in\cH_0}$ is also independent from the family $\paren{p_i(X)}_{i\in\cH_1}$, for example to get FDR bounds in \citet{dohler2018new}. In this paper, however, classical multiple testing results and inequalities are only ever applied to $\paren[1]{p_i(X)}_{i\in\cH_0}$, making this additional requirement unnecessary.
}

The following is a simple toy example making use of the introduced notation.
\exl[Gaussian one-sided under independence]
{In this example we assume that $X=(X_1,\dotsc,X_m)$ is a vector of independent Gaussians and the null hypotheses refer to the nullity of their means in contrast to their positivity (we assume that they share the same variance for simplicity). That is, formally, $(\mathcal{X},\Xi)=\paren[1]{\mathbb R^m, \mathcal B\paren{\mathbb R^m  }}$, $\statfam=\set[1]{ \mathcal N(\boldsymbol{\nu}, \sigma^2 \mathrm{Id}) : \forall j \in\Nm, \nu_j\geq 0, \sigma >0  }$, and for each $i\in\Nm$, $H_i= \set[1]{ \mathcal N(\boldsymbol{\nu}, \sigma^2 \mathrm{Id}) \in \statfam :\nu_i=0 }$. Then we can construct $p$-values by letting $p_i(X)=1-\Phi(X_i)$, where $\Phi$ denotes the cdf of $\mathcal N(0,1)$. Such $p$-values satisfy Assumptions \ref{ass_super_unif} and \ref{ass_indep}. They are even exactly uniform under the null.
}{ex_gauss_one_sided}

For every subset of hypotheses $S\subseteq\Nm$, let $V(S)=|S\cap\cH_0|$. If we think of $S$ as a selection set of hypotheses deemed significant, $V(S)$ is then the number of false discoveries in $S$. The function $V(\cdot)$ is the main object of interest in multiple testing, with past works proposing rejection procedures $R=R(X)$ with a given statistical control on $V(R)$. Popular error criteria include the Family-Wise Error Rate (FWER), defined as 
\[
\FWER(R)=\Pro{V(R)>0},
\]
the False Discovery Rate (FDR, \cite{benjamini1995controlling}), defined as the mean of the False Discovery Proportion (FDP):
\[
\FDP(R)=\frac{V(R)}{\abs{R}\vee 1}, \;\; \FDR(R) = \Esp{\FDP(R)},
\]
and the False Discovery Exceedance (FDX, \cite{Genovese2006}) for some level $\gamma\in(0,1)$:
\[
\FDX(R) = \Pro[1]{\FDP(R)>\gamma}.
\]

However, in this work and in the confidence bounds literature, one does not want to prescribe a single rejection set $R$ with proven control on $V(R)$ (and only $V(R)$), but to provide control on the function $V(\cdot)$ as a whole.

\subsection{Confidence bounds}\label{subsec_confidence_envelope}

With the formalism introduced in last Section, a simultaneous confidence upper bound is a functional 
\[
\Vhat :\left\{\begin{array}{rcl}
 \mathcal X\times (0,1) &  \longrightarrow &  \paren[2]{\cP\paren[1]{\Nm} \to \llbracket0,m\rrbracket} \\
  (x,\alpha)&\longmapsto   &  \Vhat_{(x,\alpha)}(\cdot)  \\
\end{array},\right.
\]
such that the following inequality is true:
\begin{equation}
\forall \mu\in\statfam, \forall \alpha \in (0,1),\; \Proo{\forall S \subseteq \Nm, V(S)\leq \Vhat_{(X,\alpha)}(S)}{\mu}\geq 1-\alpha.
\label{eq_confidence_formal}
\end{equation}
In the remainder or this work, the dependence in $(X,\alpha)$ will be dropped when there is no ambiguity and $\Vhat_{(X,\alpha)}$ will simply be written $\Vhat$.

Equation \eqref{eq_confidence_formal} means that, with high probability, we are able to bound the number of false discoveries in any selection set $S$. Note that the ``$\forall S \subseteq \Nm$'' being inside the probability is what makes the confidence bound ``simultaneous''.

From a simultaneous confidence upper bound on the false discoveries, we immediately get a simultaneous confidence lower bound on the trues discoveries, by letting
\begin{equation}
    \Dhat(S)=\abs{S}-\Vhat(S).
    \label{eq_confidence_formal_TP}
\end{equation}
Sometimes, confidence bounds are reported into this form, as in \citet{goeman2011multiple,goeman2021only}, or in \secref{sec_num} here when we report the numerical results of our experiments.

Also note that, by a simple division by $\abs{S}\vee 1$, \equaref{eq_confidence_formal} can be rewritten as an equation giving a simultaneaous upper bound on the FDP:
\begin{equation}
\forall \mu\in\statfam, \forall \alpha \in (0,1),\; \Proo{\forall S \subseteq \Nm, \FDP(S)\leq \frac{\Vhat_{(X,\alpha)}(S)}{\abs{S}\vee 1}}{\mu}\geq 1-\alpha.
\label{eq_confidence_formal_FDP}
\end{equation}

This, in turn, allows to use simultaneous upper bounds to construct procedures that control the FDX. Note that controlling the FDX is generally more desirable than controlling the FDR, given the nature of the latter as an expected value. See for example \citet[Figure 4]{MR3418717} for a credible example where the FDR is controlled but the FDP has a highly undesirable behavior (either 0 because no discoveries are made, or higher than the target level). The construction for FDX control is the following: one can compute the largest $S$ such that $\frac{\Vhat(S)}{\abs{S}\vee 1}$ is less than or equal to a nominal level $\gamma\in(0,1)$, and \eqref{eq_confidence_formal_FDP} ensures that, with high probability, the $\FDP$ of $S$ is upper bounded by $\gamma$, so that the $\FDX$ is bounded by $\alpha$.

\subsection{Heterogeneity}

In this work, we are interested in the case where the distributions of $p$-values under the null are heterogeneous, departing from the classical assumption that they are all distributed as $\cU([0,1])$ under the null (as in \cite{MBR2024}, for instance), that we can call the \textit{homogeneous uniform} case.

Also note that, with respect to the super-uniformity of Assumption \ref{ass_super_unif}, the homogeneous uniform case appears as a ``worst case'', in that many  multiple testing procedures, designed to provide statistical control of the type-I error for that case, become more conservative under super-uniformity. This is easy to see with classical rejection procedures that often take the form or $p$-value sublevel sets $R(X)=\set{i\in\Nm: p_i\leq \hat t(X)}$. Given that super-uniform $p$-values are stochastically larger than uniforms, the threshold $\hat t(X)$, although ensuring statistical validity, becomes badly calibrated and ``too small'', leading to fewer rejections.

In order to be able to leverage heterogeneity and super-uniformity and create confidence bounds that are more powerful than those designed to guarantee control in the homogeneous uniform case, we need some knowledge about said heterogeneity, that is, to have access to the distribution of any $p_i$ under the null, or at least to have a stochastic lower bound 
on that distribution. This motivates the following assumption:
\assl[Stochastic heterogeneous lower bound]{
For each $i\in\Nm$, there exists a known nondecreasing function $\cdf_i$ such that, if $\mu\in H_i$, the cdf of $p_i$ is upper bounded (as a function) by $\cdf_i$. Mathematically,
\begin{equation*}
    \forall \mu\in\statfam, \forall i\in\cH_0(\mu), \forall x\in\RR,\; \Pro[1]{p_i(X)\leq x}\leq \cdf_i(x).
    \label{eq_hetero}
\end{equation*}
}{ass_hetero}
If \autoref{ass_super_unif} holds, then 
\autoref{ass_hetero} is always satisfied by 
$F_i(x)=x\mapsto 0\vee (x\wedge 1)$,
but with
\autoref{ass_hetero} 
we hope to take advantage of 
tighter stochastic lower bounds if they are known, and we allow those to be heterogeneous. 
\autoref{ass_hetero} also allows to generalize beyond the super-uniform case
(i.e. if \autoref{ass_super_unif} does not hold), as is the case with weighted $p$-values (see \citealp{dohler2018new}). 

\autoref{ass_hetero} is satisfied in particular if, for all $\mu\in\statfam$, $i\in \cH_0(\mu)$, the cdf of $p_i$ does not depend on $\mu$ and is known. This is typically the case with discrete tests, where heterogeneity arises naturally, as we can see in \autoref{ex_hetero} below. More generally, \autoref{ass_hetero} is also always satisfied by $x\mapsto \sup_{\mu\in H_i}\Proo{p_i\leq x}{\mu}$. The latter choice of $F_i$ is the one made in \citet{dohler2018new}.

\defnl[Support]{For $i\in\Nm$, we define the support of $F_i$, by analogy with the support of a measure, by
\begin{align*}
    \mathcal A_i=\set[a]{x\in\RR : \forall\varepsilon>0,\; F_i(x+\varepsilon)-F_i(x-\varepsilon)>0}.
\end{align*}
}{def_support}

\exl[Independent binomials]{In this example we assume that $X=(X_1,\dotsc,X_m)$ is a vector of independent binomials of respective numbers of trials $n_1,\dotsc,n_m$ (known) and of unknown probabilities of success $\pi_1,\dotsc,\pi_m$. We test $H_i$: ``$\pi_i=\frac12$'' against ``$\pi_i\neq\frac12$''. That is, formally, $(\mathcal{X},\Xi)=\paren[1]{\mathbb R^m, \mathcal B\paren{\mathbb R^m  }}$, $\statfam=\set[1]{ \bigotimes_{j=1}^m\mathcal B\paren[a]{n_j,\pi_j} : \forall j \in\Nm, \pi_j\in[0,1]  }$, and for each $i\in\Nm$, $H_i= \set[1]{ \bigotimes_{j=1}^m\mathcal B\paren[a]{n_j,\pi_j}  \in \statfam :\pi_i=\frac12 }$. Then we can construct two-sided $p$-values by letting $p_i(\cdot):x\mapsto\Pro[1]{f_i(N_i)\leq f_i(x_i)}$, where: $N_i\sim\mathcal B\paren[a]{n_i,\frac12}$ and $f_i$ denotes the probability mass function of $\mathcal B\paren[a]{n_i,\frac12}$. Such $p$-values satisfy Assumptions \ref{ass_super_unif}, \ref{ass_indep}, and \ref{ass_hetero}, and their distributions under the null are heterogeneous as soon as not all $n_i$ are equal. In particular, we know the exact cdf of $p_i$ under the null and we can let $F_i$ be this cdf. Indeed, notice that $p_i(x)=\sum_{k=0}^{n_i}f_i(k)\ind{f_i(k)\leq f_i(x_i)}$ so, under the null, the support of $p_i$ is $\cA_i=\set[a]{\sum_{k=0}^{n_i}f_i(k)\ind{f_i(k)\leq f_i(K)}:K\in\llbracket0,n_i\rrbracket}$, and its cdf is the step function which has only discontinuities at the elements of $\cA_i$ and which is equal to the identity on $\cA_i$. For $n_i\in\set{5,15,30}$, the graphs of $F_i$, and the graph of the cdf of $\cU\paren[1]{[0,1]}$, are represented in Figure \ref{fig_discrete_cdf}.
}{ex_hetero}

\begin{figure}[ht]
\begin{center}
\includegraphics[scale=0.5]{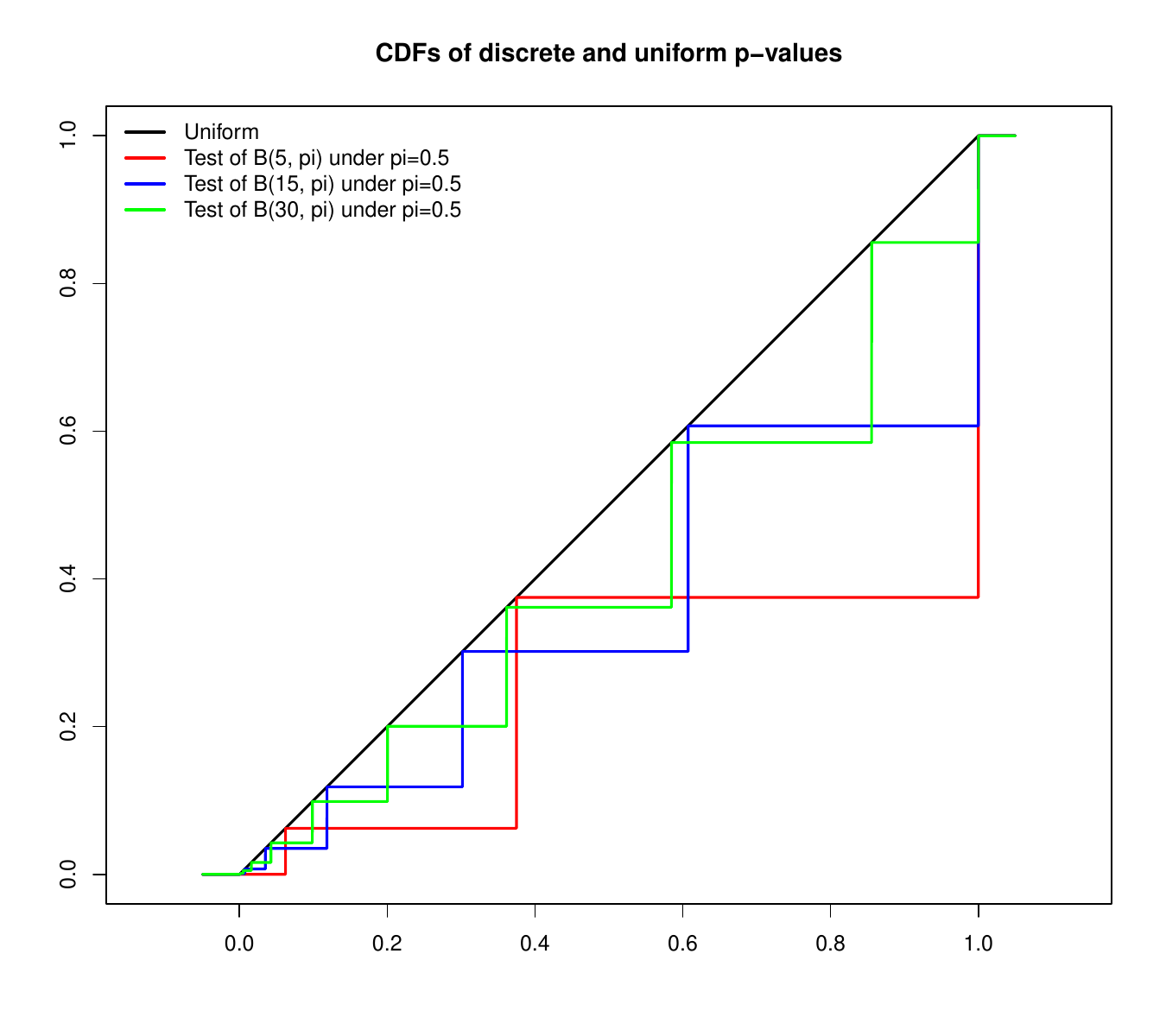}
\end{center}
\caption{Graphs of the cdf of $\cU([0,1])$ (black) and the cdf of the $p$-value under the null testing if a binomial with 5, 15 or 30 trials is of parameter $\frac12$ or not (other colors).}
\label{fig_discrete_cdf}
\end{figure}

Heterogeneity has already been investigated as a way to correct the conservativeness induced by discrete tests, for the FWER control \citep{tarone1990modified}, for the FDR control \citep{MR2134603, heller2014falsediscoveryratecontrolling, HABIGER20151, dohler2018new} and for the FDX control \citep{dohler2020controlling}. To our knowledge, this manuscript is the first work tackling the exploitation of heterogeneity in the context of confidence envelopes.

\section{Review and extension of existing theory}\label{sec_consolidation}

In this section, we revisit past works on confidence envelopes and tools used to build them, while introducing some new contributions to consolidate those works. Namely, in \secref{sec_homo_concentration}, we refine some known concentration inequalities and extend them to the super-uniform case thanks to a coupling device. 

Then, in \secref{sec_homo_inversion}, we recall the inversion method of \citet{Genovese2006} (which is equivalent to the closed testing approach of \citealp{goeman2011multiple}, as proven by \citealp[Appendix S-1.3]{blanchard2018supplement}), and a practical way to compute the derived confidence envelope in favorable cases.

Finally, in \secref{sec_homo_JER}, we recall the JER approach of \citet{blanchard2020post} and its three usual settings: $k$-FWER, top-$k$, and with deterministic regions. As new contributions, we explain how past works on the top-$k$ setting are valid under \autoref{ass_super_unif} with  adequate modifications and how some bounds can be computed efficiently in the top-$k$ setting with a new algorithm. 

Other contributions are deferred to Appendices \ref{appendix_gaps} and \ref{appendix_ties}: the bridging of the theoretical gaps in some past proofs, the irrelevance of the slight difference between our definition of the top-$k$ setting and the one of \citet{MBR2024}, and the links between the $k$-FWER approach and the top-$k$ approach.

\subsection{Homogeneous concentration inequalities}\label{sec_homo_concentration}

We start with recalling classical concentration inequalities that have been used to construct confidence bounds. While they are usually stated to be true for uniform random variables (as they are in \cite{MBR2024}), which imply that they are suitable to be used in the homogeneous uniform case (hence the name of this section), here we generalize them to super-uniform variables, which implies that they are also suitable under Assumption \ref{ass_super_unif}.

The underlying argument follows from coupling between super-uniform and uniform variables, we isolate this
(standard) device first in the following result:
\leml[Super-uniform/uniform coupling]{
Let $(U_i)_{i \in \intrN}$ be $N$ independent random variables that are super-uniform. Then there exist $\paren[1]{\widetilde{U}_i}_{i \in \intrN}$ independent random variables that are $\cU[0,1]$ (possibly defined on another probability space)
and $(V_i)_{i \in \intrN}$ random variables defined on the same space as $\paren[1]{\widetilde{U}_i}_{i \in \intrN}$, such that $(V_i)_{i \in \intrN}$  has the same distribution
as $(U_i)_{i \in \intrN}$ and $V_i \geq \widetilde{U_i}$ for all $i \in \intrN$.
}{lem_unifcoupling}

The proof of this device is given in \appref{proof_lem_unifcoupling}.

The first inequality we depict here is the well-known one-sided DKW inequality \citep{dvoretzky1956asymptotic}, with the optimal constant of \citet{massart1990tight} and the validity domain extended by \citet{reeve2024short} (to all non-negative scalars $\lambda$).

\propl[DKW inequality - Reeve's update]{
Let $U_1, \dotsc, U_N$ be $N$ independent random variables that are super-uniform.

Then, for all $\lambda \geq 0$, 
\begin{equation}\label{eq_dkwm}
\Pro{\sqrt{N}\sup\limits_{t\in [0,1]}\paren{\frac1N\sum_{i=1}^N\indev{U_i\leq t} - t}> \lambda}\leq \exp\left(-2\lambda^2\right).
\end{equation}
}{DKWin}
This concentration inequality has been used in the past in the JER setting of deterministic regions \citep{durand2020post} and in the top-$k$ setting \citep{MBR2024}, see \secref{sec_homo_JER}, although there was a gap in a proof in \citet{durand2020post}, see \appref{appendix_gaps}.

The second inequality is derived from Wellner's inequality and provided by \citet{MBR2024}. 
\propl[Uniform Wellner's inequality]{
Let $U_1, \dotsc, U_N$ be $N$ independent random variables that are super-uniform, and let $\kappa = \frac{\pi^2}{6}$. For all $\delta \in (0,1]$, 
\[
\Pro[4]{\forall t\in (0,1], \frac1N\sum_{i=1}^N\indev{ U_i\leq t}\leq t h^{-1}\paren{\frac{2\log(\kappa/\delta) + 4\log\paren[1]{1+ \log_2\paren{1/t}}}{Nt}}} \geq 1-\delta,
\]
where $h$ is the function defined by $h(x) = x(\log x-1)+1$ for all $x>1$.
}{Wellnerin}
This concentration inequality has been used in the top-$k$ setting, see \secref{sec_homo_JER}. The proof of the two last inequalities to the super-uniform case is deferred to \appref{proof_DKWin}.
The third inequality, Simes inequality \citep{simes1986improved}, has been extensively used in the literature \citep{goeman2011multiple, blanchard2020post}.

\propl[Simes inequality]{
Let $U_1, \dotsc, U_N$ be $N$ independent random variables that are super-uniform.

Then for all $\delta \in [0,1]$, $\Proo{\exists k\in \NN : U_{(k)}\leq \frac{k\delta}{N}}{}\leq \delta.$ 
}{Simesin}
We recall its simple proof using the classical Benjamini-Hochberg (BH, \cite{benjamini1995controlling}) procedure, because a similar argument will be used to prove the new heterogeneous analog to this inequality in \secref{sec_hetero_simes}.
\begin{proof}
We apply the BH procedure to $U_1, \dotsc, U_N$, seen as $N$ $p$-values under the null, at level $\delta$. It controls the FDR at level $\delta$ (even with super-uniformity instead of uniformity under the null), and given that all $U_i$ are under the null, its FDR is equal to its FWER, which in turn is exactly $\Proo{\exists k\in \NN : U_{(k)}\leq \frac{\delta k}{N}}{}$.
\end{proof}
\rem{
Given that the FDR control of the Benjamini-Hochberg procedure is valid under a more general dependence setting than independence, namely, the wPRDS condition \citep{benjamini2001control}, such is also the case for the Simes inequality. For more information about the wPRDS condition, see \citet{MR2448601} and \citet{giraud2021introduction}.
}

In \citet{MBR2024}, Simes inequality is given under another form, closer to the formulation of the inequalities of Propositions \ref{DKWin} and \ref{Wellnerin}.
\propl[Simes inequality - stochastic process version]{
Let $U_1, \dotsc, U_N$ be $N$ independent random variables that are super-uniform. For all $\delta \in (0,1]$, 
\[
\Pro[4]{\forall t >0, \frac1N\sum_{i=1}^N\indev{ U_i\leq t}<\frac{t}{\delta}} \geq 1-\delta.
\]
}{Simesin_var}
For completeness, we give the proof of the last proposition in \appref{proof_Simesin_var}.
\rem{
Notice that in our statement, a ``$<$'' sign is used, contrary to the ``$\leq$'' used in \citet{MBR2024}. It didn't matter in \citet{MBR2024} because the authors worked with uniform random variables, but in the general context of Assumption \ref{ass_super_unif}, our version can lead to less conservative bounds (see \secref{sec_homo_JER}).
}

Finally, we stress that these inequalities are not designed to use known information about heterogeneity of the data and are instead valid under the ``worst case'' that is the homogeneous uniform case. Heterogeneous analogues will be introduced in \secref{sec_hetero_concentration}.

\subsection{Confidence envelopes by the inversion procedure}\label{sec_homo_inversion}

In this section, we recall the inversion procedure introduced by \citet{Genovese2006}, which uses the notion of local test.
We first introduce, for $t\in \RR$ and $A \subset \Nm$, a notation for the number of elements in $A \cap R(t)$, where $R(t) = \set{i \in \Nm : p_i\leq t}$:
\begin{equation}\label{def_ia}
 i_A(t) = \sum\limits_{i\in A}\indev{ p_i\leq t}.
\end{equation}
When $A = \Nm$, we simplify the notation to $i(t)$.

\defnl[Local test]{Let $\alpha \in [0,1]$ and $A\subseteq \Nm$. Then we call $H_A = \bigcap\limits_{i\in A}H_i$ the intersection hypothesis for the subset $A$, and we say that a statistic $\varphi_A$ is a local test for $A$ at level $\alpha\in[0,1]$ if, for all $\mu \in H_A, \Proo{\varphi_A = 1}{\mu}\leq \alpha$.\\
A family $\paren{\varphi_A}_{A\subseteq \Nm}$ of local tests has level $\alpha\in[0,1]$ if for all $A\subseteq \Nm$, $\varphi_A$ is a local test for $A$ at level $ \alpha$.
}{def_local_test}

We can easily define local tests using concentration inequalities like the ones studied in previous section, as in the following example.

\ex{ Let $\alpha \in (0,1]$ and assume \autoref{ass_super_unif} and \autoref{ass_indep}. Then for $A\subseteq \Nm$, a local test of level $\alpha$ derived from \autoref{Simesin_var} is expressed by \[\varphi_A = \sup\limits_{t\in (0,1]}\indev{\frac{1}{|A|}i_A(t)\geq \frac{t}{\alpha}}.\] 
}

Denote $\cN_\varphi(X) = \set[1]{A \subseteq \Nm \text{ : } \varphi_A(X) = 0} $ the set of potential candidates for $\cH_0(\mu)$ according to our local tests (notice that it is a $\paren{1-\alpha}$-confidence set for $\cH_0$).
\prop[Inversion procedure, \cite{Genovese2006}]{Let $\alpha \in [0,1]$. Assume that $\paren{\varphi_A}_{A\subseteq \Nm}$ is a family of local tests at level $\alpha$. Then, the functional $\hat{V}^{\IP}_{\varphi}$ defined by $\hat{V}^{\IP}_{\varphi} (S) = \max\limits_{A \in \cN_\varphi}|A\cap S|$ is a confidence bound at level $\alpha$.}

We say that a confidence envelope is 
an \textit{exact shortcut} (for the local test family $\varphi$) if it is 
equal to $\hat{V}^{\IP}_{\varphi}$, and we 
say that this confidence envelope is a \textit{shortcut} if it
is an upper bound of $\hat{V}^{\IP}_{\varphi}$.

Note that the computation of the inversion procedure envelope is NP-hard in general,
hence a challenging goal is to express $\hat{V}^{\IP}_{\varphi}$ as an exact shortcut computable
 in polynomial time.
 Structural assumptions on the local tests can lead to such exact shortcuts. One of the most usual 
 class of local tests is the following,
 where the rejection thresholds of a local test depend only on the cardinality
 of the matching subset of hypotheses.

\defnl[Homogeneous Simes-like local tests]{ Let $ \Phi = \paren{\varphi_A}_{A \subseteq \Nm}$ be a family of local tests. We say that $\Phi$ is a homogeneous Simes-like local test family if, for all $A\subseteq \Nm$,
\begin{equation}\label{locform}
\varphi_A = \max_{i \in \llbracket 1, |A|\rrbracket}\indev{p_{(i:A)}\leq \ell_{i:|A|}},
\end{equation}
with $\paren{\ell_{i:n}}_{\substack{
0\leq n\leq m\\
1 \leq i\leq m
}}$ such that $\ell_{i:n_1}\geq \ell_{i:n_2}$ for all $n_1 \leq n_2$ and $i\in \Nm$.}{defsimeslike}
Note that the homogeneous term will be dropped when there is no ambiguity on the working case. We call \textit{vanilla Simes local test family} the previous local test family for the specific choice $\ell_{i:n} = \frac{i\alpha}{n}$.

\propl[\cite{goeman2021only}, Lemma 6]{Assume that, for all $A\subseteq \Nm$, $\varphi_A$ is of the form \eqref{locform}, then 
\begin{equation}\label{eq_SCgoemann}
\forall S \subseteq \Nm, \hat{V}^{\IP}_{\varphi}(S) =\min\limits_{1\leq u\leq |S|}\set{\abs{\set{i\in S : p_i> \ell_{u:\hat{m}_0^{\GHS}}}}+u-1},
\end{equation}
where 

\begin{equation} \label{m0goemann} \hat{m}_0^{\GHS} = \max \set{n\in \llbracket0,m\rrbracket : \forall i \in         \llbracket 1, n\rrbracket, p_{\brac{i}}> \ell_{(n-i+1):n}}.
\end{equation}
The same result holds true if the $\leq$ inequality is replaced
by $<$ in \eqref{locform}, and $>$ by $\geq$ in \eqref{eq_SCgoemann} and \eqref{m0goemann}. 
}{SCgoemann}

The proof of this proposition in \citet{goeman2021only} relies on the closed testing framework first used by \citet{goeman2011multiple} as an equivalent point of view to the inversion procedure confidence envelope. In \secref{sec_example_shortcut}, we present a new proof relying instead on JER tools (see \secref{sec_homo_JER}). The above proposition offers computation in quadratic time when the statistical framework leads to homogeneous Simes-like local tests. Note that, in the above, $\hat m_0^{\GHS}$ acts as an over-estimator of $m_0$.

\rem{\autoref{SCgoemann} is well adapted to the homogeneous case, but the local tests derived from  heterogeneous inequalities that will be introduced in \secref{sec_hetero_concentration}
will fail to satisfy the conditions of Proposition~\ref{SCgoemann}. Namely, while they
satisfy a form similar to the one of homogeneous Simes-like local tests, they are typically expressed by \[
\varphi_A = \max_{i \in \llbracket 1, |A|\rrbracket}\indev{p_{(i:A)}\leq \ell_{i:A}},
\]
where, for all $i\in \Nm$, $\paren{\ell_{i:A}}_{A\subseteq\Nm}$ is nonincreasing  with respect to the inclusion partial order, and, for all $A\subseteq \Nm$, $\paren{\ell_{i:A}}_{i\in \Nm}$ is nondecreasing.\\
Thus these local tests will, in general, fully depend on $A$ rather than only $|A|$, 
which is the reason why the shortcut from \autoref{SCgoemann} cannot be applied when the data heterogeneity is considered to construct local tests.}

\subsection{Reference families, JER and interpolation}\label{sec_homo_JER}

\subsubsection{JER methods}

We present here the alternative approach to constructing confidence bounds originally introduced in \citet{blanchard2020post}, based on the two concepts of reference families and Joint Error Rate (JER) that we start by defining. A reference family is a finite family $\Rfam_\alpha(X)=(R_k,\zeta_k)_{k\in \cK}$ with $R_k\subseteq\Nm$, $\zeta_k\in\mathbb N$, where everything (that is, $\cK$ and all the $R_k$ and $\zeta_k$) depends on $(X,\alpha)$, but the dependency is not explicitly written. Likewise, when there is no ambiguity, the reference family will simply be denoted $\Rfam_\alpha$ or even $\Rfam$. The JER of a reference family is an error criterion that takes the following form:
\begin{equation}
\JER(\Rfam) = \Pro{\exists k\in\cK, |R_k\cap\cH_0| > \zeta_k } = \Pro{\exists k\in\cK, V(R_k) > \zeta_k }.
\label{eq_jer}
\end{equation}
Think of each $\zeta_k$ as an over-estimator of the number of false discoveries in $R_k$. Then the JER is the probability that this upper bound fails for at least one $k$. If the JER is small, then, with high probability, the $\zeta_k$'s are simultaneous upper bounds of the false discoveries in all the regions $R_k$'s.

We are then interested in constructing reference families that control the JER at any target level, the inequation that we want to see satisfied is then:
\begin{equation*}
\forall \mu\in\statfam, \forall \alpha \in (0,1),\; \JER_{\mu}(\Rfam_\alpha) \leq \alpha,
\end{equation*}
or, in other words, at any given $\alpha \in (0,1)$,
\begin{equation}
\forall \mu\in\statfam,\; \Proo{\forall k\in\cK, V(R_k)\leq \zeta_k}{\mu} \geq 1-\alpha.
\label{eq_jer_control}
\end{equation}
Note that Equation \eqref{eq_jer_control} is really similar to Equation \eqref{eq_confidence_formal} except that the uniform guarantee, instead of being over all $S\subseteq \Nm$, is only over all the $R_k\subseteq \Nm, k\in \cK$, with $\cK$ having cardinality potentially much smaller than $2^m$. A simultaneous confidence bound over all subsets of hypotheses is then derived from a JER-controlling reference family by interpolation. Let 
\begin{equation}
\Afam(\Rfam)= \set{A\subseteq \Nm:  \forall k\in\cK, |R_k\cap A| \leq \zeta_k }.
\label{eq_a}
\end{equation}
The JER control informs us that, with high probability, $\cH_0\in\Afam (\Rfam)$. We leverage this information with the following confidence bound construction:
\begin{equation}
\hat V^{\JER}_\Rfam(S) = \max_{A\in\Afam (\Rfam)}|S\cap A|,
\label{eq_vstar}
\end{equation}
which optimally uses the information provided by the JER control of the reference family, as proven by Proposition 2.1 of \citet{blanchard2020post}. Note that there exists an equivalence between constructions of confidence envelopes based on reference families and on the inversion procedure
based on local tests. The links between the two approaches are discussed in \citet[Appendix S-1.3]{blanchard2018supplement}. 

Because of the combinatorial
maximum operation $\max_{A\in\Afam(\Rfam)}$, 
the computation of $\hat V^{\JER}_\Rfam(S)$ is intractable in general (see Proposition 2.2 of \citealp{blanchard2020post}), but for specific structures of reference families introduced in the next paragraph, a polynomial computation can be achieved.

First, if the set $\set{R_k : k\in \cK}$ is totally ordered for the inclusion relation, that is, the $R_k$'s are nested, then, for all $S\subseteq\Nm$, 
\begin{equation}\label{eq_vbar}
\hat V^{\JER}_\Rfam(S)=\min_{k\in\cK}\paren[1]{\zeta_k+\abs{S\setminus R_k}}\wedge|S|,
\end{equation}
by Proposition 2.5 of \citet{blanchard2020post}. Second, if the reference family satisfies the forest structure constraint introduced by \citet{durand2020post} and defined by: 
\[
\forall k,k'\in\cK,\;R_k\cap R_{k'}\in\set{R_k, R_{k'}, \varnothing},
\]
then fast algorithms also exist to compute $\hat V^{\JER}_\Rfam(S)$ for a single $S\subseteq\Nm$ \citep[Algorithm 1]{durand2020post} or, even better, to compute $\paren[1]{\hat V^{\JER}_\Rfam(S_t)}_{1\leq t\leq m}$ for a path of nested hypotheses subsets $S_1\subsetneq \dotsb \subsetneq S_m$ \citep[Algorithm 4]{durand2025fastalgorithmcomputecurve}. Note that the forest structure condition encompasses the nested condition.

We can now recall the three settings that have been studied for the construction of reference families that control the JER. They will be studied again in the context of heterogeneity in \secref{sec_hetero_JER}.
\subsubsection{\texorpdfstring{$k$}{k}-FWER setting}
The first setting is the ``$k$-FWER'' setting, in which we set $\cK=\llbracket1,K\rrbracket$ for some $K\leq m$, $\zeta_k=k-1$ and $R_k=\set{i\in\Nm: p_i\leq \tau_k}$, with a threshold family $\paren{\tau_k}_{1\leq k\leq K}$ that is nondecreasing and such that the JER control holds. By monotonicity of $\paren{\tau_k}_{1\leq k\leq K}$, such reference family satisfies the nested condition so \equaref{eq_vbar} holds and an algorithm given by \citet{10.1093/bioinformatics/btac693} can be applied for fast computation of $\hat V^{\JER}_\Rfam(S)$ for any given $S$. The name of the setting comes from the fact that the JER control in that case is equivalent to the fact that, simultaneously on $k$, each $R_k$ controls the $k$-FWER. An example of templates that yields JER control under our Assumptions is the Simes template introduced by \citet[Section S-2]{blanchard2018supplement}, defined by $\tau_k=\frac{\alpha k}{m}$:
\thml[Simes template yields JER control]{
Under Assumptions \ref{ass_super_unif} and \ref{ass_indep}, for all $\mu\in\statfam$, for all $\alpha\in(0,1)$, the reference family $\Rfam=\paren[a]{R_k, k-1}_{1\leq k \leq m}$, with $R_k=\set[a]{i\in\Nm: p_i\leq \frac{\alpha k}{m}}$, controls the JER at level $\alpha$.
}{thm_simes_kFWER}

We defer the proof of the last theorem to \appref{proof_thm_simes_kFWER}

\reml{The same proof scheme proves that Hommel's inequality \citep{blanchard2018supplement} also induces a $k$-FWER reference family that controls the JER, given by \[R_k=\set[a]{i\in\Nm: p_i\leq \frac{\alpha k}{m} \sum\limits_{i=1}^m\frac1i},\] without needing \autoref{ass_indep}.}{rem_hommel}

\subsubsection{Top-\texorpdfstring{$k$}{k} setting}

We move on to second setting, namely the ``top-$k$'' setting that appeared in \citet{Katsevichramdas2020}. Let $\sigma$ a (random) permutation that orders the $p$-values, i.e., $p_{\sigma(k)}=p_{(k)}$ for all $k\in\Nm$. Because we do not assume that the $p$-values have continuous distributions, there may be ties and $\sigma$ may not be unique. The top-$k$ setting consists in fixing 
\begin{equation}\label{eq_def_Rk_top_k}
    R_k=\set[1]{\sigma(1),\dotsc,\sigma(k)}
\end{equation}
for all $k\in \Nm$, and to find appropriate $\zeta_k$'s such that the JER control holds. 
\reml{
In past works assuming that the $p$-values have continuous distributions, $R_k$ was rather defined as $R_k^\dag=\set{i\in\Nm: p_i\leq p_{(k)}}$, and it is a.s. equal to our definition under such assumption. However, our definition ensures that $|R_k|=k$ even in the case of ties, which is practical for computational purposes, while we can only say that $\abs[a]{R_k^\dag}\geq k$ under \autoref{ass_super_unif}. Thankfully, the choice of $R_k$ or $R_k^\dag$ is inconsequential with respect to the confidence envelope, as is proven in \appref{sub_appendix_ties}.
}{rem_topk_ties}
\reml{
We can always assume that the sequence $\paren{\zeta_k}_{1 \leq k \leq m}$ is nondecreasing, up to replacing $\zeta_k$ by $\min(\zeta_k, \zeta_{k+1}, \dotsc, \zeta_m)$. Indeed, if the event inside the probability of Equation \eqref{eq_jer_control} is realized, for $k^*\geq k$ such that $\zeta_{k^*}=\min(\zeta_k, \zeta_{k+1}, \dotsc, \zeta_m)$, we have $V(R_k)\leq V\paren{R_{k^*}}\leq \zeta_{k^*}$.
}{rem_topk_nondecreasing}

Many constructions are proposed by \citet{MBR2024}, under the assumptions that the $p$-values are uniform under the nulls. These constructions can be extended to the case of \autoref{ass_super_unif} thanks to the following proposition.

\propl[Inequalities for the top-$k$ setting, extended to super-uniformity]{
Assume that Assumptions \ref{ass_super_unif} and \ref{ass_indep} hold.

Then, for all $\delta \in (0,1]$, and for all $\mu\in\statfam$,
\begin{equation}\label{eq_dkwm_topk}
\Proo{\forall t\in[0,1], i_{\cH_0(\mu)}(t) \leq mt +\sqrt{m}\sqrt{\frac12\log\frac1\delta}}{\mu}\geq 1-\delta,
\end{equation}
\begin{equation}\label{eq_wellner_topk}
\Proo{\forall t\in(0,1], i_{\cH_0(\mu)}(t) \leq mt h^{-1}\paren{\frac{2\log(\kappa/\delta) + 4\log\paren[1]{1+ \log_2\paren{1/t}}}{mt}}}{\mu}\geq 1-\delta,
\end{equation}
\begin{equation}\label{eq_simes_topk}
\Proo{\forall t>0, i_{\cH_0(\mu)}(t) < \frac{mt}{\delta} }{\mu}\geq 1-\delta.
\end{equation}
}{prop_concen_topk_hmogeneous}

The proof of the previous proposition is given in \appref{proof_prop_concen_topk_hmogeneous}.
This proposition allows to deduce the following JER-controlling families, as defined in the following theorem.

\thml[Reference families for the top-$k$ setting, extended to super-uniformity]{
Assume that Assumptions \ref{ass_super_unif} and \ref{ass_indep} hold. Let $\alpha\in(0,1)$, $\mu\in\statfam$, and $X\sim\mu$. The following three reference families control the JER at level $\alpha$.
\begin{enumerate}
    \item the family $\paren{R_k,\zeta_k}_{1 \leq k \leq m}$ with, for all $k\in\Nm$, $R_k$ as in \eqref{eq_def_Rk_top_k} and \[\zeta_k=\indev[a]{p_{(k)}>0}\floor[a]{mp_{(k)}+\sqrt{\frac m2\log\frac1\alpha}}\wedge k,\]
    \item the family $\paren{R_k,\zeta_k}_{1 \leq k \leq m}$ with, for all $k\in\Nm$, $R_k$ as in \eqref{eq_def_Rk_top_k} and \[\zeta_k=\indev[a]{p_{(k)}>0}\floor[a]{mp_{(k)} h^{-1}\paren{\frac{2\log(\kappa/\alpha) + 4\log\paren[1]{1+ \log_2\paren{1/p_{(k)}}}}{mp_{(k)}}}}\wedge k,\]
    \item\label{item_topk_simes} the family $\paren{R_k,\zeta_k}_{1 \leq k \leq m}$ with, for all $k\in\Nm$, $R_k$ as in \eqref{eq_def_Rk_top_k} and \[\zeta_k=\indev[a]{p_{(k)}>0}\floor[a]{\frac{mp_{(k)}}{\alpha}^-}\wedge k,\]
\end{enumerate}
where $\floor[a]{\cdot}$ denotes the floor function and $\floor[a]{\cdot^-}$ is its left-limit.
}{thm_families_topk_homogeneous}
\reml{
Those bounds are similar to the Equations (9), (10) and (15) of \citet{MBR2024}, except that they are bounds for the false discoveries and not the FDP, that they are valid under super-uniformity under the null instead of just uniformity, that they use the floor functions to better exploit the fact that the number of false discoveries is an integer, and that they use a $\indev[a]{p_{(k)}>0}$ term that was not useful in \citet{MBR2024} because the $p$-values were assumed to have continuous distributions, whereas here, in the context of \autoref{ass_super_unif}, it becomes necessary.
}{rem_diff_with_mbr}
\reml{
Note that each definition of $\zeta_k$ in \autoref{thm_families_topk_homogeneous} is of the form $f\paren[a]{p_{(k)}}\wedge k$ for some measurable function $f$ that only takes integer values. Also note that we can use the same argument than in \autoref{rem_topk_nondecreasing}: up to replacing $f(\cdot)$ by $t\mapsto \min_{t\leq t'\leq 1}f(t')$ (the minimum is well-defined because $f$ only takes integer values), we can furthermore assume that $f$ is nondecreasing.
}{rem_topk_nondecreasing_2}
\begin{proof}
First note that, almost surely, $V(R_k)\leq \sum_{i\in \cH_0}\indev[a]{p_i\leq p_{(k)}}$ because, if $i\in R_k$, then $p_i\leq p_{\sigma(k)}=p_{(k)}$. Then note that, by super-uniformity, $p_i>0$ a.s. if $i\in\cH_0$, hence $V(R_k)=0$ if $p_{(k)}=0$.

Then, like in \citet{MBR2024}, we just apply inequalities \eqref{eq_dkwm_topk}, \eqref{eq_wellner_topk} and \eqref{eq_simes_topk} by setting $t=p_{(k)}$ if $p_{(k)}>0$. The floor function can be applied because $V(R_k)$ is an integer. The minimum with $k$ appears from $V(R_k)\leq |R_k|=k$. Finally, in the case of Equation \eqref{eq_simes_topk}, the left-limit of the floor function can be applied instead of the floor function thanks to the strict inequality in \eqref{eq_simes_topk}.
\end{proof}

\reml{For computational purposes, note that, for all $t\in\RR$, $\floor[a]{t^-}=\ceil[a]{t}-1$, where $\ceil[a]{\cdot}$ denotes the ceiling function. See \autoref{counterex_simes_equal_simes} for an example where using the left-limit or not does matter.}{rem_right_floor}

A top-$k$ reference family satisfies the nested condition, so \equaref{eq_vbar} can also be applied for fast computation of $\hat V^{\JER}_\Rfam(S)$ for a given $S$. 
In particular, it can be applied to the reference top-$k$ sets themselves, giving rise to a possibly improved bound $\widetilde{\zeta}_k = \hat V^{\JER}_\Rfam(R_k) \leq \zeta_k, k\in\Nm$
(see \citealp{blanchard2020post} for a discussion).
To compute the whole family $(\widetilde{\zeta}_k)_{1 \leq k \leq m}$
more efficiently than by repeatedly applying \eqref{eq_vbar}, 
we introduce the following new \algoref{algo_topk_path}, which can spare the (possibly costly) computation of some of the original bounds $\zeta_k$.

\begin{algorithm}
\KwData{$p$-values $\paren{p_k}_{1 \leq k \leq m}$, measurable nondecreasing function $f$ taking integer values }
\KwResult{$\paren[1]{\hat V^{\JER}_\Rfam(R_k)}_{1 \leq k \leq m}$ on top-$k$ path}
$k \longleftarrow 1$\;
$\tilde{\zeta}_0 \longleftarrow 0$\;
\While{$k\leq m$}{
$\zeta_k \longleftarrow f\paren{p_{(k)}} \wedge k$\;
\If{$\zeta_k > \tilde{\zeta}_{k-1}$}{
$j\longleftarrow \paren{\zeta_k - \tilde{\zeta}_{k-1}}\wedge \paren{m-k+1}$\;
\For{$i = 1$ to $j$}{
$\tilde{\zeta}_{k+i-1} \longleftarrow \tilde{\zeta}_{k-1}+i$\;}
$k \longleftarrow k+j$\;}
\Else{$\tilde{\zeta}_{k}\longleftarrow \tilde{\zeta}_{k-1}$\;
$k\longleftarrow k+1$\;
}
}

\Return $\paren{\tilde{\zeta}_{k}}_{1 \leq k \leq m}$.
\caption{Computation of $\paren[1]{\hat V^{\JER}_\Rfam(R_k)}_{1 \leq k \leq m}$}
\label{algo_topk_path}
\end{algorithm}

\propl[]{
In the setting of a top-$k$ reference family (see~\eqref{eq_def_Rk_top_k}), assume that  $\zeta_k=f\paren[a]{p_{(k)}}\wedge k$ for some measurable, nondecreasing function $f$ that only takes integer values (see \autoref{rem_topk_nondecreasing_2}).
Then, \algoref{algo_topk_path} indeed returns the improved bound $\paren[1]{\hat V^{\JER}_\Rfam(R_k)}_{1 \leq k \leq m}=\paren[1]{\tilde{\zeta}_{k}}_{1 \leq k \leq m} $ for the reference family $\Rfam = \paren[1]{R_k,\zeta_k}_{1 \leq k \leq m}$. }{prop_algo_topk_path}

We defer the proof of the previous proposition to \appref{proof_prop_algo_topk_path}.

\reml{
\equaref{eq_vstar_algo_2} provides a simpler algorithm than \algoref{algo_topk_path}, that could be implemented with a simple \texttt{for} loop. However, this would imply to compute each $\zeta_k$, which can be costly, depending on the function $f$. On the contrary, \algoref{algo_topk_path} leverages the fact that not all $\zeta_k$ need to be computed: $\zeta_k$ is computed only at the steps $k$ visited by the \texttt{while} loop.
}{rem_simpler_algo}

We end this overview of the top-$k$ setting by remarking that there exists a direct correspondence between this setting and the $k$-FWER setting. Namely, from a reference family built according to one of these settings, we can build another one in the other setting, and they yield the exact same confidence envelopes. Furthermore, the two reference families using Simes inequality that we encountered in Theorems \ref{thm_simes_kFWER} and \ref{thm_families_topk_homogeneous} also yield the exact same confidence envelopes. This point is formalized precisely in \appref{sub_appendix_topk_kFWER}.

\subsubsection{Deterministic setting}

We finally move to the third setting relevant for applying the JER
point of view, called ``deterministic setting''.
In that setting, the regions $R_k$ are deterministic, fixed in advance (possibly reflecting some prior information), and we build appropriate $\zeta_k$'s such that the JER control holds. Any generic method that can upper bound the number of false discoveries in a deterministic set of hypotheses at any confidence level (including other confidence envelopes) can be applied to $R_k$ with the confidence level set to $\frac \alpha K$, where $K=\abs[a]{\cK}$, and JER control ensues by a simple union bound argument. We tend to choose our regions such that the forest structure conditions hold, so that \citet[Algorithm 4]{durand2025fastalgorithmcomputecurve} can be applied.

One such construction has been done in \citet{durand2020post}, based on the DKW inequality (see \autoref{DKWin}).
\propl[Proposition 1 of \citealp{durand2020post}]{
The family $\paren{R_k,\zeta_k}_{k\in\cK}$ with $\cK$ a deterministic set and, for all $k\in\cK$, $R_k$ deterministic, and 
\begin{equation}\label{eq:zeta_deterministic}
\zeta_k=\min_{t\in[0,1)}\floor[a]{\paren[a]{\frac{\lambda_{\alpha/K}}{2(1-t)}+\sqrt{\frac{\lambda_{\alpha/K}^2}{4(1-t)^2} +  \frac{|R_k|-i_{R_k}(t)  }{1-t}}}^2   }\wedge |R_k|,
\end{equation}
where $\lambda_{\alpha/K}=\sqrt{\frac12\log \frac K\alpha}$, controls the JER at level $\alpha$.
}{prop_dkwm_deter}

\section{General shortcuts for the inversion procedure}\label{sec_general_shortcuts}

This section introduces new shortcuts for local tests having a specific structure. Namely, we assume that our local tests, for all $A \subseteq \Nm$, are expressed as 
\begin{equation}\label{loctestenv}
\varphi_A = \max_{t\in [0,1]}\varphi_{A,t}. 
\end{equation}
While any local test can obviously be written under this form through
the trivial choice $\varphi_{A,t}=\varphi_A$ for all $t$,
the underlying guiding idea for considering this expression
is that, for particular representations taking this form, 
we are able to compute efficiently --- that is, in polynomial time ---  the minimum 
of $\varphi_{A,t}$ for a fixed $t$, over sets $A$ of a fixed cardinality.
Therefore, our aim in this section is to derive shortcuts relying first on taking such minima 
for fixed $t$.

In the remainder of this section, we assume that we 
have at hand a family of local tests expressed in the form~\eqref{loctestenv}.
\secref{sec_shortcuts} focuses on defining and establishing validity of the
new shortcuts.
In \secref{sec_example_shortcut} we
provide a methodology to derive confidence bounds, which turns out to be an exact shortcut for the inversion procedure
in the special case of homogeneous Simes-like local tests.

Local tests derived from ecdf envelopes, like the DKW and Wellner inequalities,
have a natural expression in the form~\eqref{loctestenv}.
That holds for Simes-like local tests as well,
as stated by the next lemma. Furthermore, the heterogeneous local test families that will be introduced in \secref{sec_hetero_concentration} are expressed as~\eqref{loctestenv} as well.

\leml[Alternative form of Simes-like tests]{Let A $\subseteq \Nm$, assume that $\paren{\varphi_A}_{A\subseteq \Nm} $ is a homogeneous Simes-like local test family with the sequence $\paren{\ell_{i:n}}_{1\leq i\leq m}$ nondecreasing for all $n\in \Nm$. Then, with the convention $\ell_{0:n} = -1$ for all $n\in \Nm$, for all $A\subseteq \Nm$:  \[\varphi_A = \max\limits_{t \in [0,1] } \indev{t\leq  \ell _{i_A(t):|A|}},\]
where $i_A$ is defined by \eqref{def_ia}.}{goemannloct}
We give the proof of the previous lemma in \appref{proof_goemannloct}.

\subsection{Shortcuts for the inversion procedure}\label{sec_shortcuts}
This section introduces new shortcuts for the inversion procedure.
These shortcuts rely on the fundamental statistic $\varphi_{S,n,t}$ that is defined for $S\subseteq \Nm$, $t\in [0,1]$ and $n \in \NS$ by 
\begin{equation} \label{expressionmin}
  \varphi_{S,n,t} = \min\limits_{\substack{A \subseteq \Nm\\ |A\cap S| = n}}\varphi_{A,t}.
\end{equation} 

For fixed $S$ and $t$, this statistic is computable in polynomial time (in $m$, independently of $n$) for all the local test families considered in the present article (including heterogeneous settings). Furthermore, as will be seen shortly (in \autoref{continuoustodiscrete}), taking the minimum or maximum over $t\in[0,1]$ can in fact be restricted to finite subsets
of $[0,1]$, thus preserving the polynomial complexity for such operations.

\propl{ The following confidence envelopes, defined for all $S\subseteq \Nm$ by
\begin{equation}\label{SC1}
\hat{V}^{\SCo}_{\varphi} (S) =\max\set{n\in \NS:\max\limits_{t \in [0,1]} \varphi_{S,n,t} = 0},
\end{equation}
and 
\begin{equation}\label{SC2}
\hat{V}^{\SCt}_{\varphi} (S) = \min\limits_{t \in [0,1]}\set[1]{\max\set{n \in \NS : \varphi_{S,n,t} = 0}},
\end{equation}
satisfy
\[\forall S \subseteq \Nm, \hat{V}^{\IP}_{\varphi} (S) \leq  \hat{V}^{\SCo}_{\varphi} (S) \leq \hat{V}^{\SCt}_{\varphi} (S).\]
}{SChetero} 

The proof of \autoref{SChetero} is provided in \appref{proof_SChetero}.

To fully benefit from the efficient computations of $\varphi_{S,n,t}$, one needs to be able to reduce the ``$\max\limits_{t \in [0,1]}$'' and the ``$\min\limits_{t \in [0,1]}$'' in \autoref{SChetero} into a ``$\max\limits_{t \in \cA}$'' and a ``$\min\limits_{t \in \cA}$'' with $\cA$ the set of realized $p$-values. The next proposition provides a monotonicity assumption on $\varphi_{A,\cdot}$, for $A\subseteq \Nm$, to this end.

\propl{Suppose, defining $p_{(0:A)} = p_0 = 0$ for all $A \subseteq \Nm$, that for all $A\subseteq \Nm$
$\varphi_{A,\cdot}$ is nonincreasing on $[p_{\paren{i:A}},p_{((i+1):A)})$ for all $i\in \llbracket 0, |A|-1\rrbracket$.

Then the three following equalities occur for all $A \subseteq \Nm$, $S \subseteq \Nm$ and $n \in \NM$:
\[\varphi_A = \max_{i\in A\cup\set{0}} \varphi_{A,p_i},\]
\[\max\limits_{t \in [0,1]} \varphi_{S,n,t} = \max\limits_{i \in \NM} \varphi_{S,n,p_i},\]
\[\hat{V}^{\SCt}_{\varphi} (S) = \min\limits_{i \in \NM}\set{\max\set{n \in \NS : \varphi_{S,n,p_i} = 0}}.\]
}{continuoustodiscrete}

\autoref{continuoustodiscrete} is proved in \appref{proof_continuoustodiscrete}.

The monotonicity assumption from the previous proposition will cover all the local tests provided by concentration inequalities considered in the present article. Thus, while we keep using
the notation $\max_{t \in [0,1]}$ for notational convenience
in the rest of the paper, in practice only evaluations
at $t$ equal to one of the $p$-values (or $t=0$) are required.

Finally, we give a sufficient monotonicity condition on $\varphi_{S,\cdot,t}$ for our two shortcuts for the inversion procedure to be equal: 
\propl{Assume that, for $S \subseteq \Nm$, for all $t\in [0,1]$, $\varphi_{S,\cdot,t}$ is nondecreasing, that is, for all $n \in \NS$ : \[\varphi_{S,n,t} = 1 \Rightarrow \forall k\geq n, \varphi_{S,k,t} =1.\] Then $\hat{V}^{\SCo}_{\varphi} (S) = \hat{V}^{\SCt}_{\varphi} (S) 
$.}{SC1eqSC2}
The proof of \autoref{SC1eqSC2} is deferred to \appref{proof_SC1eqSC2}.
Note that under the assumptions of \autoref{continuoustodiscrete}, we only need that $\varphi_{S,\cdot,p}$ is nondecreasing for all $p \in \set{0} \cup \set{p_i, i\in \Nm}$.

\subsection{Applications to homogeneous inequalities}\label{sec_example_shortcut}

We now study the 
shortcuts defined previously when the local tests are based
on homogeneous inequalities, namely
homogeneous Simes-like local tests, and DKW inequality based local tests. In both cases,
we first give an explicit expression for the shortcuts, then investigate
their use to obtain an over-estimator $\mzsc = \hat{V}^{\SCo}_{\varphi}
\paren[1]{\Nm}$ for $m_0=\abs{\cH_0}$, which is then combined with the JER point of view, in
order to obtain ``adaptive JER'' envelopes.

In both cases, we establish that this adaptive  JER envelope coincides with the inversion procedure. This shows as a by-product that the adaptive Simes procedure
and the adaptive DKW-based top-$k$ procedure introduced by \citet{MBR2024}
are both exact shortcuts for their respective associated local test families.

The exactness of the proposed
shortcuts in the homogeneous case provides strong motivation
for using a similar approach in the heterogeneous
setting in \secref{sec_hetero_JER}.

\subsubsection{Case of homogeneous Simes-like local tests}\label{sec:recovSCgoemann}

In this section, we assume that the local tests $\varphi_A$ are  homogeneous Simes-like local tests, as by \autoref{defsimeslike}. We define $\ell_{0:n}=-1$ for all $n\in \Nm$ to handle the case when $p$-values are equal to $0$. Recall that, from \autoref{goemannloct}, homogeneous Simes-like local tests can
be reformulated as $\varphi_A = \max\limits_{t\in [0,1]}\varphi_{A,t}$ with $\varphi_{A,t} = \indev{t \leq \ell_{i_A(t):|A|}}$.

Based on this representation, we start with a lemma giving an expression of $\varphi_{S,n,t}$, defined by \equaref{expressionmin}, which will allow us to compute our shortcuts.

\leml{Let $S\subseteq \Nm$ and $n\in \NS$. Then, for all $t\in [0,1]$, \[\varphi_{S,n,t} = \min_{0\leq k\leq i_{S^c}(t)}\indev{t\leq \ell_{\paren{\paren{n + i_S(t)-|S|}\vee 0 +k}:\paren{n+m-|S|-i_{S^c}(t) +k}}}.\] In particular, \[\varphi_{\Nm,n,t} = \indev{t\leq \ell_{\paren{n + i(t) - m}:n}}.\]}{Simeslikemin}

The proof of \autoref{Simeslikemin} is given in \appref{proof_Simeslikemin}.

The following example gives an explicit expression for $\varphi_{\Nm,n,t}$ in the specific case of vanilla Simes local tests,
as a by-product of \autoref{Simeslikemin}.
\exl{Consider the vanilla Simes local tests at level $\alpha \in (0,1]$, that is such that for all $n\in \Nm$ and $i\in \llbracket 1 , n \rrbracket$, $\ell_{i:n} = \frac{i}{m}\alpha$ in  the setting of \autoref{defsimeslike}. Then for all $t\in (0,1]$ and $n\in \Nm$,

\[\varphi_{\Nm,n,t}= \indev{t \leq \frac{n + i(t) - m}{n}\alpha}.\]
}
{minmaxsimes}

We now propose the following "adaptive JER" methodology based on this shortcut
applied to $S=\intr{1,m}$:
\begin{itemize}
\item Define $\mzsc = \hat{V}^{\SCo}_{\varphi}
\paren[1]{\Nm}$.
\item Introduce the reference family  $\Rfam = \paren{R_k,k-1}_{k\in \llbracket 1, \mzsc + 1\rrbracket}$ with 
\begin{equation}
\begin{cases} R_k & = \set{i\in \Nm : p_i \leq \ell_{k:\mzsc}}, \text{ for } k\in \llbracket 1, \mzsc \rrbracket;\\
R_{\mzsc + 1} & = \Nm.
\end{cases}\label{def:adaptreffamilysimes}
\end{equation}
\item Use the confidence envelope based on the JER bound
using this reference family.
\end{itemize}   

The main result of this section is that the adaptive JER confidence envelope coincides with the inversion procedure confidence envelope
using the initial local test family.

The following lemma establishes that the $\mzsc$ coincides with the one from \autoref{SCgoemann}, \equaref{m0goemann}.

\leml{ Defining $\mzsc = \hat{V}^{\SCo}_{\varphi} \paren[1]{\Nm}$, we have
\[\mzsc = \max \set{n\in \llbracket0,m\rrbracket : \forall t \in [0,1], t> \ell_{(n+i(t) -m):n}}.\]
Furthermore, if for all $n \in \Nm$ and $i\in \llbracket 1, n \rrbracket$, $\ell_{i+1:n+1}\geq \ell_{i:n}$, then \[\mzsc=
\hat{V}^{\SCt}_{\varphi} \paren{\Nm} =  \min\limits_{t \in [0,1]}\set[2]{\max\set{n \in \NM : t> \ell_{(n+i(t) -m):n}}}.\]}{scgoemanproof2}

\begin{proof}
The expression from \autoref{Simeslikemin} for $S = \Nm$ is nondecreasing with respect to $n$ when $\ell_{i+1:n'+1}\geq \ell_{i:n'}$ for all $i$ and $ n'$, thus applying \autoref{SC1eqSC2} leads to the conclusion.
\end{proof}

Note that the assumption ``$\ell_{i+1:n+1}\geq \ell_{i:n}$'' was already considered in \citet{goeman2021only} as a sufficient condition for using bisection to compute $\hat{m}_0^{\GHS}$.\\
\corl{In the same framework as \autoref{minmaxsimes}, 

\begin{equation}\label{closedformsimes}
\mzsc = \inf\limits_{t\in (0,\alpha)} \paren{\left\lceil\frac{m - i(t)}{1- t/\alpha} \right\rceil - 1}\wedge \paren[1]{m- i(0)}.
\end{equation}
}{SimesVSC1equalsVSC2}

The proof of the previous corollary is deferred in \appref{proof_SimesVSC1equalsVSC2}.

A similar bound was already considered for Simes top-$k$ adaptive bound in \citet{MBR2024}, but 
it did not consider cases with $p$-values equal to $0$ or to one of Simes thresholds. Denoting $\hat{m}_0^{\MBR}$ the bound from \citet{MBR2024}, we have here $\mzsc = \paren{\ceil{\hat{m}_0^{\MBR}} -1} \wedge \paren[1]{m-i(0)}$ in the framework of vanilla Simes local tests.

The following proposition states that our methodology 
produces an exact shortcut.

\propl{For a homogeneous Simes-like local test family,
consider the adaptive reference family given by~\eqref{def:adaptreffamilysimes}.
Then $\JER(\Rfam)\leq \alpha$ and $\hat{V}^{\IP}_{\varphi} = \hat{V}^{\JER}_\Rfam $.}{eqJERIP}

For completeness, we provide a proof of this proposition in \appref{proof_eqJERIP}

Combining \autoref{eqJERIP} and \autoref{SimesVSC1equalsVSC2}, the adaptive interpolated Simes bound from \citet{MBR2024} applied with our over-estimator defined by \eqref{closedformsimes} is an exact shortcut, computable in time $\mathcal{O}\paren[1]{m\log\paren{m}}$.
\rem{
Note that, for the family $\Rfam$ defined in \autoref{eqJERIP}, the expression of $\hat{V}^{\JER}_\Rfam$ does not match exactly the right-hand side of \equaref{eq_SCgoemann} from \autoref{SCgoemann}. The latter can be written $\hat{V}^{\JER}_{\Rfam_S}$ with the pseudo-reference family $\Rfam_S$ (which depends on $S$) defined by $\Rfam_S =\paren{R_k,k-1}_{k\in \llbracket 1, |S|\rrbracket}$, with $R_k = \set{i\in \Nm : p_i \leq \ell_{k:\mzsc}}$ for $k\in \llbracket 1, |S|\rrbracket$. Finally it can be checked that the proof of \autoref{eqJERIP} also implies that $\hat{V}^{\JER}_\Rfam=\hat{V}^{\JER}_{\Rfam_S}$, thus we obtain a new proof of \autoref{SCgoemann}, without using the closed testing formalism.

}

\subsubsection{Exactness of the adaptive top-\texorpdfstring{$k$}{k} procedure for the DKW inequality}

We first introduce an alternative expression
of \autoref{SCgoemann} suited to the top-$k$ setting, as it allows the user to avoid translating their top-$k$ designed local test into homogeneous Simes like designed local test (see \autoref{rem_topk_local_test_to_homo_simes}). It can be useful when  the functional $f_n$ is computationally hard to invert and it is convenient to prove theoretical results on top-$k$ bounds.
\propl{Assume that there exists a family of non negative valued function $\paren{f_{n}}$ such that, for all $A\subseteq \Nm$,  \begin{equation}\label{eq:topklocaltests} 
\varphi_A =\max_{t\in [0,1]}\indev{i_A(t) > f_{|A|} \paren{t}}.
\end{equation}
If, for $A\subseteq \Nm$, $f_{|A|}$ is nondecreasing, $\varphi_A =\max\limits_{j \in A}\indev{i_A(p_j) > f_{|A|} \paren{p_j}}.$\\
If, furthermore, the function $f_n$ is right continuous for all $n\in \Nm$ and  the function sequence $\paren{f_n}$ is nondecreasing, 
\begin{align}\label{eq_VIP_top_k_view}
\hat{V}^{\IP}_{\varphi}(S) &= \min\limits_{0\leq k\leq |S|}\paren[2]{\floor{f_{\mzsc}\paren{p_{(k : S)}}} + \abs{S\setminus R\paren{p_{(k:S)}}}} 
\end{align}
where  \begin{equation}\label{m0perier}
\mzsc  = \max \set[2]{n\in \llbracket0,m\rrbracket : \forall t \in [0,1], f_{n}\paren{t}\geq n+i(t)-m}.
\end{equation}
}{SCfuncversion}

The proof of this proposition is given in \appref{proof_SCfuncversion}.

In the remainder of this section, the local tests considered are the ones derived from \autoref{DKWin}, that is, we let $\alpha \in (0,1]$ and $\lambda_\alpha = \sqrt{\frac{\log(1/\alpha)}{2}}$, and for all $ A \subseteq \Nm$, a local test is defined by $\varphi_A = \max\limits_{t\in [0,1]}\varphi_{A,t}$, where, for all $t\in [0,1]$ and all $n\in \NM$,
\[
f_n(t) = nt+\sqrt{n}\lambda_\alpha,
\]
and
\begin{equation}\label{eq:dkwphit}
    \varphi_{A,t} = \indev{i_A(t) > f_{\abs{A}} \paren{t}}.
\end{equation}

\leml{Considering the DKW local tests, $\mzsc$ has the following closed form: \begin{equation}\label{eq:DKWhatm0nonlocal}
    \mzsc = \min\limits_{i\in [
   0,1)}\floor[a]{\paren[a]{\frac{\lambda_\alpha}{2\paren{1-t}}+\sqrt{\frac{\lambda_\alpha^2}{4\paren{1-t}^2} +  \frac{m - i\paren{t}}{1-t}}}^2   }\wedge m.
\end{equation}}{DKWestimator}

This expression already appeared in \citet{durand2020post} and was used as a plug-in estimator for the adaptive top-$k$ procedure derived from DKW inequality in \citet{MBR2024}.

\rem{Applying \autoref{DKWestimator} and  \equaref{eq:DKWhatm0nonlocal} to a deterministic subset $R_k\subseteq\Nm$ and at level $\alpha/K$, instead of $\Nm$ and $\alpha$, exactly recovers the formula of $\zeta_k$ given in \equaref{eq:zeta_deterministic}.}

The next proposition proves that 
the adaptive top-$k$ DKW method described in \citet{MBR2024}
 is an exact shortcut for the inversion procedure.
 Furthermore, this confidence bound is computable in time $\mathcal{O}\paren[1]{m\log\paren{m}}$.

\propl{Considering DKW local test family, $\hat V^{\IP}_\varphi =\hat V^{\JER}_\Rfam$ with $\Rfam = \paren{R_k,\zeta_k}$ where for all $k\in \Nm$, $\zeta_k = k \wedge \left\lfloor  \mzsc p_{(k)}+ \sqrt{\mzsc }\lambda_\alpha\right\rfloor$ with $\mzsc$ defined by \equaref{eq:DKWhatm0nonlocal}, and $R_k = \set[1]{\sigma(i), 1\leq i \leq k}$ with $\sigma$ a (stochastic) permutation ordering $p$-values by increasing order.
}{DKWtopkopt}

We prove \autoref{DKWestimator} and \autoref{DKWtopkopt} in \appref{proof_DKWtopkopt}.

\section{Heterogeneous probabilistic inequalities and associated local test confidence envelopes}\label{sec_hetero_concentration}
In this section, we prove new probabilistic inequalities adapted to \autoref{ass_indep} and \autoref{ass_hetero}. \secref{sec_bretagnolle} revisits the so-called Bretagnolle inequality depicted in \citet{shorack2009empirical}, an analogue of DKW inequality relying on a known data heterogeneity, and updates it thanks to a recent work \citep{reeve2024short}. Then, \secref{sec_hetero_simes} introduces a new $k$-FWER inequality, using a FDR control method from \citet{dohler2018new} and extracts an $m_0$-adaptive version of this inequality. From all those new inequalities, we construct local tests. Note that a third approach using an inequality due to \citet{van1978properties} is deferred to \autoref{appendix_VZ} because of its lack of practical power.

\subsection{Bretagnolle inequality and applications}\label{sec_bretagnolle}

The first inequality that we present is a slight refinement of an inequality found in \citet{shorack2009empirical} and that is attributed  to Bretagnolle. It can be thought as a heterogeneous analogue of DKW inequality (\autoref{DKWin}).

\propl[Bretagnolle inequality]{
Let $X_1, \dotsc, X_N$ be $N$ independent random variables, with respective cdf denoted by $J_1, \dotsc, J_N$. Assume that there exist some functions $K_1, \dotsc, K_N$ such that $J_i(x)\leq K_i(x)$ for all $i\in\llbracket1,N\rrbracket, x\in\RR$, and let $\bar{K}:x\mapsto\frac1N\sum_{i=1}^N K_i(x)$. Then, for all $\lambda \geq 0$,

\begin{equation}\label{eq_bretagnolle2}
\Pro{\sqrt{N}\sup\limits_{t\in\mathbb{R}}\paren{\frac1N\sum_{i=1}^N\indev{X_i\leq t}- \bar{K}(t)}> \lambda}\leq e \cdot \exp\left(-2\lambda^2\right).
\end{equation}
}{bretin}
We defer a proof based on the one from \citet{shorack2009empirical} updated by the results from \citet{reeve2024short} in \appref{proof_bretin}.

If the random variables $X_1, \dotsc, X_N$ are super-uniform and we use the cdf of the uniform distribution as $K_1, \dotsc, K_N$, then \eqref{eq_bretagnolle2} becomes \eqref{eq_dkwm}, that is DKW inequality, except that the right-hand side has a multiplicative factor of $e$, that we can think of as a penalty term due to the heterogeneity. In that sense, Bretagnolle inequality is indeed a heterogeneous analog of DKW inequality.

In the remainder of the section, we fix $\alpha \in (0,1]$, define $\tilde{\lambda}_\alpha =\sqrt{\frac{1 +\log\paren[a]{1/\alpha} }{2}}$, the random vector $\mathbb{G}$ and the non-random vector $\mathbf{F}$ under \autoref{ass_hetero}: \begin{align} 
   \forall t \in [0,1],& \hspace{2pt} \mathbb{G}(t) = \paren[1]{\indev[a]{p_i\leq t}- F_i(t)}_{1\leq i \leq m},\label{defiG}\\
   & \hspace{2pt} \mathbf{F}(t) =  \paren[1]{F_i(t)}_{1\leq i \leq m}.\label{defiF}
\end{align}
For $A\subseteq \Nm$, we will consider the local test deduced from Bretagnolle inequality (in the meaning of \autoref{def_local_test}), $\varphi_A = \max\limits_{t\in [0,1]} \varphi_{A, t}$, where for all $t \in [0,1]$,

\begin{align} 
\varphi_{A, t} =  \indev[a]{\sum\limits_{i\in A}\mathbb{G}_i(t)>\tilde{\lambda}_\alpha\sqrt{|A|}}.\label{eq_local_bretagnolle_1} 
\end{align}

\corl{
Assume \autoref{ass_indep} and \autoref{ass_hetero} hold. Then, for all $A\subset \Nm$, $\varphi_A$  is a local test at level $\alpha$.

}{cor_local_bretagnolle}

The proof of the previous corollary is given in \appref{proof_cor_local_bretagnolle}.

\reml{Note that $\tilde{\lambda}_\alpha = \lambda_{\tilde{\alpha}}$ with $\tilde{\alpha} = e^{-1}\alpha$. Thus when \autoref{ass_hetero} holds with for all $i\in \Nm$, $F_i= F$, the Bretagnolle inequality reduces to the DKW inequality at level $e^{-1}\alpha$.}{lambdabret}

Similarly to the local tests considered in the homogeneous setting, we get an expression of $\varphi_{S,n,t}$ \eqref{expressionmin} for the local test family aforementioned.

\leml{
For all $S\subseteq \Nm$, $n \in \NS$ and $i\in \Nm$,  
\[
\varphi_{S,n,t}=\min_{0\leq j\leq m-|S| }\indev{\sum\limits_{1\leq k\leq n}\paren[1]{\mathbb{G}(t)}_{(k:S)}+ \sum\limits_{1\leq k\leq j}\paren[1]{\mathbb{G}(t)}_{(k:S^c)}>\tilde{\lambda}_\alpha\sqrt{n+j}}.
\]
}{Bretmaxmin}
We defer the proof to \appref{proof_Bretmaxmin}.
This lemma allows us to compute 
a first confidence envelope thanks to \autoref{SChetero} in the following proposition. This envelope is a shortcut computable in polynomial time for
the inversion procedure applied to Bretagnolle local test. 

\propl{
Assume that \autoref{ass_indep} and \autoref{ass_hetero} hold. Then, 
\begin{equation}
\scalebox{0.8}{$\hat{V}^{\SCo}_\varphi(S) = \max\set{n \in \NS : \forall t \in [0,1],\exists j\leq m-|S|, \sum\limits_{1\leq k\leq n}\paren[1]{\mathbb{G}(t)}_{(k:S)}+ \sum\limits_{1\leq k\leq j}\paren[1]{\mathbb{G}(t)}_{(k:S^c)}\leq\tilde{\lambda}_\alpha\sqrt{n+j}}.$}
\end{equation}

In particular, when $S = \Nm$, with probability $1-\alpha$, 
\begin{equation}\label{bretm0}
\mzsc = \max\set{n \in \NM : \forall t \in [0,1],  \sum\limits_{1\leq j\leq n}\paren[1]{\mathbb{G}(t)}_{(j)}\leq\tilde{\lambda}_\alpha\sqrt{n}}.
\end{equation} 
}{BretSC1}

\equaref{bretm0} allows the user to approximate the upper bound on $m_0$ given by the inversion procedure applied to $\varphi_A$, which may not be computable in practice, by a statistic computable in polynomial time. Indeed applying \autoref{continuoustodiscrete} to compute \eqref{bretm0}, we first need for all $t\in \set[1]{p_i,i\in \Nm}$ to sort $\paren[1]{F_j\paren{t}}_{j\in R(t)}$ which is in time $\mathcal{O}\paren[1]{\log(m)m^2}$. Then
the following identity for $n > m-i(t) $:
\[\sum\limits_{1\leq j\leq n}\paren[1]{\mathbb{G}\paren{t}}_{(j)} = n - \paren[1]{ m-i(t) } - \sum\limits_{j \in R(t)^c}F_j\paren{t} - \sum\limits_{1\leq j\leq n - \paren{ m-i(t) }}\paren[1]{\mathbf{F}\paren{t}}_{\left[j : R(t)\right]}\] 
allows to compute the matrix $\paren{\sum\limits_{1\leq j\leq n}\paren[1]{\mathbb{G}\paren{p_i}}_{(j)} }_{(n,i)\in \Nm^2}$ in time  $\mathcal{O}\paren[1]{m^2}$, and finally we can find the maximum in \eqref{bretm0} in time $\mathcal{O}\paren[1]{m^2}$. Hence, $\mzsc$ is computable in time $\mathcal{O}\paren[1]{\log(m)m^2}$ for Bretagnolle local test family.

By the same type of reasoning, $\hat{V}^{\SCo}_\varphi(S)$ can be computed in time $\mathcal{O}\paren[1]{m^3}$.\\

Similar to previous adaptive procedures that plug in an over-estimator of $m_0$ (\textit{e.g.} \autoref{DKWtopkopt}), we are interested in using an over-estimator $\hat m_0$, such as the one defined by \eqref{bretm0}, in combination with Bretagnolle inequality. The following proposition states an adaptive version of the inequality by plug-in, suited to deriving such an adaptive procedure, which will be done in \secref{sec_adap_bretenv}.

\propl{Assume \autoref{ass_indep} and \autoref{ass_hetero}. Let $\hat{m}_0$ such that \[\set{\sup\limits_{ t \in [0,1]}\sum\limits_{k\in \cH_0} \mathbb{G}_k(t)\leq \sqrt{m_0} \tilde{\lambda}_\alpha} \subseteq\set{m_0\leq \hat{m}_0},\] then, with $\mathbf{F}$ defined by \eqref{defiF}, \[\Pro{\exists t \in [0,1],   i_{\cH_0}\paren{t} >  \sum_{1\le k \le \hat{m}_0 }\paren[1]{\mathbf{F}(t)}_{[k]} + \sqrt{\hat{m}_0}\tilde{\lambda}_\alpha} \leq \alpha.\]}{adaptivebretin}

The proof is given in \appref{proof_adaptivebretin}.

\subsection{Heterogeneous Simes-like inequality}\label{sec_hetero_simes}

In this section, we turn to a different type of concentration inequality, which we derive from
an existing step-up $\FDR$ control procedure designed for the heterogeneous setting \citep{dohler2018new}.
This is similar to
the way one can recover Simes inequality from the Benjamini-Hochberg procedure in the homogeneous setting.

To this end we define, for a subset of hypotheses $U\subseteq \Nm$ and a threshold family $\tau = \paren{\tau_k}_{1\leq k \leq m} $ (\textit{i.e.} a nondecreasing family of $m$ elements of $\RR_+$), the step-up procedure on $U$ with respect to $\tau$; $SU\paren[0]{\tau,U} = \set[1]{i\in U : p_i \leq \tau_{\hat{k}}}$, where $\hat{k} = \max\set{k \in\llbracket0,|U|\rrbracket: p_{(k:U)}\leq \tau_k}$, with, by convention, $p_{(0:U)}=0$ and $\tau_0=\tau_1$. We will consider, for $X \in \RR^d$ $\paren{d\in \nat}$, that $X_{[d+1]}=+\infty$. This last convention is merely considered to 
get aesthetic expressions in \autoref{DDRlocaltest}, similarly to the convention $\ell_{0:n} = -1$ in \secref{sec:recovSCgoemann}.

If \autoref{ass_hetero} holds, we denote, for all $(s,t) \in \paren{\RR_+}^2$  such that $F_i(s)<1$, \[\mathbf{H}(s,t) =  \paren{\frac{\cdf_i(t)}{1-\cdf_i(s)}}_{1\le i \le m}. \]

We introduce a condition involving the vector $\mathbf{H}$ on a threshold family.

\defn{We say that the threshold family $\tau$ satisfies \eqref{C1} at level $\alpha\in [0,1]$ with respect to $U\subseteq \Nm$ if $\cdf_i\paren{\tau_{|U|}}<1$ for all $i\in U$ and
\begin{equation} \label{C1} \tag{C1}
    \underset{1\leq k\leq |U|}{\mathrm{max}}\set{ \underset{\substack{A\subseteq U \\ |A| = |U| -k +1}}{\mathrm{max}}\set{
\frac{1}{k} \sum_{i \in A}\mathbf{H}_i(\tau_{|U|},\tau_k)
    }} \leq \alpha.
\end{equation}
}

The condition \eqref{C1} on the threshold family $\tau$ is relevant to get a control on the $\FDR$ at level $\alpha$ in addition to the the result of
\citet{dohler2018new} that we recall:

\thml[\citealp{dohler2018new}, Theorem 1]{Assume \autoref{ass_indep} and \autoref{ass_hetero}. 

Let $U\subseteq \Nm$ and $\alpha \in [0,1]$. Furthermore, assume that the two families $\paren{p_i}_{i\in U\cap \cH_0}$ and $\paren{p_i}_{i\in U\cap \cH_1}$ are independent.

Let $\tau$ be a threshold family satisfying \eqref{C1} at level $\alpha$ with respect to $U$. Then,
\[\FDR\paren[1]{SU\paren[0]{\tau, U}}\leq \alpha.\]
}{control_hetero_fdr}

The contribution of $\tau_{|U|}$ in \eqref{C1} is different from the contributions of the other thresholds. This asymmetry naturally results in letting $\tau_{|U|}$ parametrize a sequence of thresholds, like it is designed in the following example.

\exl{
Let $\alpha \in (0,1]$ and
\begin{equation}\label{eq_tau_lim}
\lambda \in \set[2]{t\in \cA: \forall i\in\Nm, F_i(t)<1, \text{ and } \paren[1]{\mathbf{H}(t,t)}_{[1]} \leq m \alpha },
\end{equation}
and, for all $k\in \llbracket1,m -1\rrbracket$, let \begin{equation} \label{eq_tau_AHSU}
\tau_k = \mathrm{max}\set[2]{t\in \cA, t\leq \lambda : \paren[1]{\mathbf{H}(\lambda,t)}_{[1]} + \dots + \paren[1]{\mathbf{H}(\lambda,t)}_{[m-k+1]}\leq k\alpha}.
\end{equation}
Then $\tau = \paren{\tau_1,\dots,\tau_{m-1},\lambda}$ satisfies \eqref{C1} at level $\alpha$ with respect to $\Nm$. \\
Thus $\FDR\paren[2]{SU\paren[1]{\tau, \Nm}}\leq \alpha$.
}{thresholds_ddr}

\rem{In comparison to 
\citet{dohler2018new}, we introduce an additional flexibility 
in the threshold family via the parameter 
$\lambda$.
We can recover the threshold
sequence of the [AHSU]
procedure of \citet{dohler2018new}
with the particular choice
$\lambda=\max\set[a]{t\in\cA:\sum_{i=1}^m 
\mathbf{H}_i(t, t) \leq m\alpha}$ in \autoref{thresholds_ddr}. 
The role of the parameter $\lambda$ is analogous to the parameter of the adaptive one-stage step-up procedure of \citet{MR2579914}.
}

An other fundamental example is the following.

\exl{
Let $\alpha \in (0,1]$, $U\subseteq \Nm$ and 
\begin{equation}\label{eq_lambda_U}
\lambda \in \set[2]{t\in \cA: \forall i\in U, F_i(t)<1, \text{ and } \paren[1]{\mathbf{H}(t,t)}_{[1:U]} \leq |U| \alpha },
\end{equation}
and, for all $k\in \llbracket1,|U| -1\rrbracket$, let 
\begin{equation} \label{eq_tau_U}
\tilde{\tau}_k = \mathrm{max}\set[2]{t\in \cA, t\leq \lambda : \paren[1]{\mathbf{H}(\lambda,t)}_{\brac{1:U}} + \dots + \paren[1]{\mathbf{H}(\lambda,t)}_{\brac{|U|-k+1:U} }\leq k\alpha}.
\end{equation}
Then $\tilde{\tau} = \paren{\tau_1,\dots,\tau_{|U|-1},\lambda, \bar{\cA}, \dots, \bar{\cA}}$ where $\bar{\cA} = \sup \cA$, satisfies \eqref{C1} at level $\alpha$ with respect to $U$.

}{thresholds_ddr2}
This last example will be particularly useful with $U = \cH_0$, in combination with the following lemma which states a practical monotonicity property of the condition \eqref{C1}.

\leml{Let $\tau$ be a threshold family from $\cA$, $V \subseteq \Nm$ and $\alpha\in (0,1]$. Assume that $\tau$ satisfies \eqref{C1} at level $\alpha$ with respect to $V$. \\Then for all $U\subseteq V$, for $\tilde{\tau}$ defined by \eqref{eq_tau_U} with $\lambda \in \set{\tau_{|U|},\tau_{|V|}}$, $\tau_k\leq \tilde{\tau}_k$  holds for all $k\in
\Nm$.\\
In particular, $\tau$ satisfies \eqref{C1} at level $\alpha$ with respect to all $U\subseteq V$.}{examplesbest}

We prove \autoref{examplesbest} in \appref{proof_examplesbest}.

In particular, this last lemma states with $U = V$ that \autoref{thresholds_ddr2} gives an optimal threshold family satisfying \eqref{C1} with respect to $V$, up to the choice of $\tau_{|V|}$. Furthermore, it states that a threshold family satisfying \eqref{C1} at level $\alpha$ with respect to $\Nm$ also satisfies \eqref{C1} at level $\alpha$ with respect to $\cH_0$, especially the thresholds defined in \autoref{thresholds_ddr}.

From the FDR control of a general step-up procedure with threshold family $\tau$ satisfying \eqref{C1}, we deduce an analogue of Simes inequality (\autoref{Simesin}), using $\tau$ instead of the family $\paren[a]{\frac{\alpha k}{m}}_{1 \leq k \leq m}$.

\thml[Heterogeneous Simes-like inequality]{ Let $U\subseteq \Nm$ such that $\mu \in H_U$.
Let $\tau$ be a threshold family satisfying \eqref{C1} at level $\alpha \in [0,1]$ with respect to $U$, and assume  \autoref{ass_indep} and \autoref{ass_hetero}.
Then, the following inequality holds: $$\mathbb{P}\paren{ \exists i \in \llbracket 1,|U|\rrbracket \text{, } p_{(i : U)} \leq \tau_i} \leq \alpha.$$
}{thm_hetero_simes}

\begin{proof}

\begin{equation*}
\begin{split}
\mathbb{P}\paren{ \exists i \in \llbracket1,|U|\rrbracket  \text{, } p_{(i : U)} \leq \tau_i} & = \FWER\paren[2]{SU\paren[1]{(\tau_k)_{1\leq k \leq m_0}, U }}\\
&= \FDR\paren[2]{SU\paren[1]{(\tau_k)_{1\leq k \leq m_0}, U }} \text{ since each null hypothesis in $U$ is true,}\\
&\leq \alpha,
\end{split}
\end{equation*}
where the  inequality comes from \autoref{control_hetero_fdr} applied to $U$. 

\end{proof}

Remark that, in opposition to \autoref{control_hetero_fdr}, this last theorem does not require any assumptions on $\cH_1$, therefore the assumptions on the null hypotheses are the same as the ones mandatory in \secref{sec_bretagnolle}.
\autoref{thm_hetero_simes}, in addition to \autoref{examplesbest}, entails the following corollary:

\corl{ Let $\tau$ be a threshold family satisfying \eqref{C1} at level $\alpha \in [0,1]$ with respect to $\Nm$, and assume  \autoref{ass_indep} and \autoref{ass_hetero}.
Then, the following inequality holds: $$\mathbb{P}\paren{ \exists i \in \llbracket 1,m_0\rrbracket \text{, } p_{\paren{i : \cH_0}} \leq \tau_i} \leq \alpha.$$}{corNmtoH0}

The first fundamental application of \autoref{thm_hetero_simes} is to
define local tests in the heterogeneous setting.

\corl[Simes-like heterogeneous local test]{Assume  \autoref{ass_indep} and \autoref{ass_hetero}. Let $\alpha\in (0,1]$ and \begin{equation}\label{lambdaDDR}
\lambda \in \set[2]{t\in \cA: \forall i\in\Nm, F_i(t)<1, \text{ and } \paren[1]{\mathbf{H}(t,t)}_{[1]} \leq \alpha }.
\end{equation}
Let $A\subseteq \Nm$ and define $k = |A|$. Then,
\begin{align}
\varphi_A &= \max\limits_{i \in \llbracket 1,k\rrbracket} \indev{\displaystyle\sum_{\ell = 1}^{ k-i+1} \paren[2]{\mathbf{H}\paren[1]{\lambda,p_{(i:A)}}}_{[\ell:A]} \leq i\alpha}\indev{p_{(i:A)}\leq \lambda} \label{localDDRorderstat}\\
&= \max_{t\in[0,\lambda]} \indev{\displaystyle\sum ^{k-i_A(t)+1}_{\ell=1}\paren[1]{\mathbf{H}\paren[0]{\lambda,t}}_{[\ell:A]} \leq i_A(t)\alpha}, \label{localDDRt}
\end{align}
where $i_A(t)$ is defined by \eqref{def_ia}, is a local test of level $\alpha$. }{DDRlocaltest}

The proof of the validity of these local tests is given in \appref{proof_DDRlocaltest}.

\rem{Here, our local tests are defined with a ``$\max\limits_{t\in[0,\lambda]}$'' instead of a ``$\max\limits_{t\in[0,1]}$''. All the propositions from \secref{sec_shortcuts} can easily be adapted to this case.}

This local test family, thanks to the mild assumption on $\lambda$ \eqref{lambdaDDR}, derives a computable expression of $\varphi_{S,n,t}$ \eqref{expressionmin} stated in the next lemma.

\leml{Assume  \autoref{ass_indep} and \autoref{ass_hetero}. Let $\alpha \in [0,1]$, $S\subseteq \Nm$, $n\in \NS$, $\lambda$ defined by \eqref{lambdaDDR}, $t\in [0,\lambda]$ and consider the local tests from \autoref{DDRlocaltest}. Define· $k_S(t) = \paren{ n - \abs{R(t)^c \cap S}}\vee 0$ and $S(t) = R(t)\cap S$, then, \[\varphi_{S,n,t} = \indev{\displaystyle\sum ^{m - i(t)   +1}_{k=1}\paren[2]{\mathbf{H}_{R(t)^c}\paren[a]{\lambda,t},\paren[1]{\mathbf{H}\paren[a]{\lambda,t}}_{[1:S(t)]}, \dots,\paren[1]{\mathbf{H}\paren[a]{\lambda,t}}_{\brac{k_S(t):S(t) }} }_{\brac{k}} \leq k_S(t)\alpha}.\]}{DDRminphi}

This previous lemma is proven in \appref{proof_DDRminphi}. It allows us to easily express the confidence envelopes achieved by the combination of the local tests described in \autoref{DDRlocaltest} and \autoref{SChetero}.

\propl{Assume  \autoref{ass_indep} and \autoref{ass_hetero}. Let $\alpha \in [0,1]$, $S\subseteq \Nm$, $n\in \NS$, $\lambda$ defined by \eqref{lambdaDDR} and consider the local tests from \autoref{DDRlocaltest}. Define for all $t\in [0,\lambda]$, $k_S(t) = \paren{ n - \abs{R(t)^c \cap S}}\vee 0$ and $S(t) = R(t)\cap S$, then:
\begin{equation}
\scalebox{0.73}{$ \hat{V}^{\SCo}_\varphi(S) =\max\set{n \in \NS : \forall t\in [0,\lambda], \displaystyle\sum ^{m - i(t)   +1}_{k=1}\paren[2]{\mathbf{H}_{R(t)^c}\paren[a]{\lambda,t},\paren[1]{\mathbf{H}\paren[a]{\lambda,t}}_{[1:S(t)]}, \dots,\paren[1]{\mathbf{H}\paren[a]{\lambda,t}}_{\brac{k_S(t):S(t) }} }_{\brac{k}} >k_S(t)\alpha}
$.}
\end{equation}
In particular, when $S = \Nm$, with probability $1-\alpha$, 
\begin{equation}\label{DDRm0}
\scalebox{0.77}{$\mzsc = \max\set{n \in \NM : \forall t\in [0,\lambda], \displaystyle\sum ^{m - i(t)   +1}_{k=1}\paren[2]{\mathbf{H}_{R(t)^c}\paren[a]{\lambda,t},\paren[1]{\mathbf{H}\paren[a]{\lambda,t}}_{[1:R(t)]}, \dots,\paren[1]{\mathbf{H}\paren[a]{\lambda,t}}_{\brac{k(t):R(t) }} }_{\brac{k}} > k(t)\alpha}$.}
\end{equation} 
}{DDRSC1}

Note that the computation of $\mzsc$ is in time $\mathcal{O}\paren[1]{\log(m)m^2}$ similarly to the framework of Bretagnolle's local test family.

The second fundamental application of \autoref{thm_hetero_simes} 
is to design adaptive, data-dependent threshold families. The idea is that we can replace theoretical thresholds 
satisfying condition \eqref{C1} depending on the (unknown)
$\cH_0$ by data-dependent thresholds that are smaller than
the theoretical ones on a suitable event.

\leml[Data-dependent threshold family]{Let $\alpha \in [0,1]$ and $\tilde{\tau}$ be a threshold family satisfying \eqref{C1}  at level $\alpha$ with respect to $\cH_0$. Suppose  \autoref{ass_indep} and \autoref{ass_hetero} hold. Then, for any (possibly data-dependent) threshold family $\tau$ such that  $\tau_k \leq \tilde{\tau}_k$ for all $k\in \Nmz$ whenever the event $\set{\forall i \in \cH_0,  p_{(i : \cH_0)} > \tilde{\tau}_i }$ is realized, the following inequality holds: \[\mathbb{P}\paren{ \exists i \in \Nmz \text{, } p_{(i : \cH_0)} \leq \tau_i} \leq \alpha.\]}{adap_hetero_simes}

\begin{proof}Define the events $N=\set{\forall k \in \Nmz, \tau_k \leq \tilde{\tau}_k}$, $\tilde{B} = \set{\forall i \in \cH_0,  p_{(i : \cH_0)} > \tilde{\tau}_i }$ and
$B = \set{\forall i \in \cH_0,  p_{(i : \cH_0)} > {\tau}_i }$. The assumption on $\tau$ translates to $\tilde{B} \subseteq N$,
hence $N \cap \tilde{B} = \tilde B$.
Furthermore $N \cap \tilde{B} \subseteq B$. 
Therefore $\tilde{B} \subseteq B$, and $\mathbb{P}\paren[1]{\tilde B}\geq 1-\alpha$ by \autoref{thm_hetero_simes} with $U = \cH_0$, thus $\mathbb{P}\paren{ B }\geq 1-\alpha$.
\end{proof}
\autoref{adap_hetero_simes} will be useful to define a $m_0$-adaptive form of the first threshold family from \autoref{thresholds_ddr} in the following example:
\propl{Let $\alpha \in (0,1]$, $\tilde{\tau}$ be defined by \autoref{thresholds_ddr2} with $U = \cH_0$ and $\hat{m}_0$ be a statistic such that \[\set{\forall i \in \Nmz \text{, } p_{(i : \cH_0)} > \tilde{\tau}_i}\subseteq \set{m_0 \leq \hat{m}_0}.\] 
Furthermore, assume  \autoref{ass_indep} and \autoref{ass_hetero}. Define $\tau_{\hat{m}_0} = \tilde{\tau}_{m_0} $ and \begin{equation} \label{eqadaptau}
\scalebox{0.9}{$\forall k \in \llbracket  1,\hat{m}_0-1 \rrbracket, \tau_k = \mathrm{max}\set[2]{t\in \cA, t\leq \tilde{\tau}_{\hat{m}_0} : \paren[1]{\mathbf{H}(\tau_{\hat{m}_0},t)}_{[1]} + \dots + \paren[1]{\mathbf{H}(\tau_{\hat{m}_0},t)}_{[\hat{m}_0 -k+1] }\leq k\alpha}.$}
\end{equation} Then $\tau = \paren{\tau_1,\dots, \tau_{\hat{m}_0},\bar{\cA},\dots,\bar{\cA}}$ satisfies the inequality \[\mathbb{P}\paren{ \exists i \in \Nmz \text{, } p_{(i : \cH_0)} \leq \tau_i} \leq \alpha.\] }{adap_thresholds}
\begin{proof}
We know that under the event  $\set{\forall i \in \Nmz,  p_{(i : \cH_0)} > \tilde{\tau}_i }$, the event  $\set{m_0 \leq \hat{m}_0}$ occurs, thus, for all $U \supseteq \cH_0$ such that $|U| = \hat{m}_0$, $\tau$ satisfies condition \eqref{C1} at level $\alpha$ with respect to $U$ by \equaref{eqadaptau}. Therefore, by \autoref{examplesbest}, $ \tau_k \leq \tilde{\tau}_k$, hence \autoref{adap_hetero_simes} indeed applies.
\end{proof}

\autoref{adap_thresholds} allows us to derive adaptive procedures in \secref{adapsimes} using two different upper bounds on $m_0$.

\rem{In the statement of \autoref{adap_thresholds}, $\tilde{\tau}$ can be considered as an ``oracle'' threshold family for the parameter $\lambda = \tilde{\tau}_{m_0}$, because it relies on the knowledge of $\cH_0$. However, defining $\lambda$ by \eqref{eq_tau_AHSU} with $k = m_0$ or by \eqref{lambdaDDR} implies in particular that $\lambda$ verifies \eqref{eq_lambda_U} with $U = \cH_0$, thus the knowledge of $\cH_0$ is not mandatory to define adaptive thresholds.}

\section{Heterogeneous confidence envelopes using reference families}\label{sec_hetero_JER}
This section formally introduces our new heterogeneous confidence envelopes derived from the heterogeneous inequalities depicted in \secref{sec_hetero_concentration}, following the methodology introduced in \secref{sec_example_shortcut}.

Note that these inequalities have already been used to build confidence envelopes in \secref{sec_hetero_concentration} thanks to the local test formalism --- derived from inexact shortcuts for the inversion procedure. Nevertheless, these confidence envelopes have a complexity in $\mathcal{O}\paren{m^3}$ in general, where the one we will exhibit here are in time $\mathcal{O}\paren[1]{m^2\log\paren{m}}$.

Finally, recall that for a nested reference family $\Rfam = (R_k,\zeta_k)_{1 \leq k \leq m}$, $ \hat{V}^{\JER}_\Rfam $ is defined by: 
\[ \forall S \subseteq \Nm, \hat{V}^{\JER}_\Rfam(S) = \min_{k\in \Nm} \paren[a]{\zeta_k + \abs{S \setminus R_k}}\wedge|S|.\]

\subsection{Top-\texorpdfstring{$k$}{k} setting}
This section will introduce our new bounds relying on the Bretagnolle inequality from \secref{sec_bretagnolle}. We will first introduce the non-adaptive bound which is a consequence of our new adaptive bound on \secref{sec_adap_bretenv}.
Recall that, in the top-$k$ setting, $R_k = \set[1]{\sigma(i), 1\leq i \leq k}$ with $\sigma$ a (stochastic) permutation ordering $p$-values by increasing order.

\subsubsection{Non adaptive bound}\label{sec_nadap_bretenv}
The first bound that we introduce here is the non adaptive one.
\corl{
Let $\alpha\in ]0,1[$. Assume \autoref{ass_indep} and \autoref{ass_hetero}. Define $\tilde{\lambda}_\alpha=\sqrt{\frac{1+\mathrm{log}(1/\alpha)}{2}}$ and $\Rfam = (R_k,\zeta_k)_{1 \leq k \leq m}$ where for all $ k\in \Nm,$
\[ \zeta_k=k \wedge \floor[2]{m \bar{F}\paren[1]{p_{(k)}}+ \sqrt{m}\tilde{\lambda}_\alpha}.\] 
Then, $ \hat{V}^{\JER}_\Rfam $ is a confidence bound of level $\alpha$.
}{nonadaptivetopk}
\begin{proof}
It is a direct application of \autoref{adaptivebretin} with $\hat{m}_0  = m$ knowing that $\abs{\cH_0\cap R_k}\leq i_{\cH_0}\paren{p_{(k)}}$ for all $k\in \Nm$.
\end{proof}

Note that the previous confidence bound is computed in time $\mathcal{O}\paren[1]{m\log(m)}$.

Compared to the DKW calibration, the Bretagnolle calibration does not use an upper bound $F$ on the $F_i$ but the mean $\bar{F}$ of the individual $F_i$, which is always smaller than $F$, and $\tilde{\lambda}_\alpha=\sqrt{\frac{1+\mathrm{log}(1/\alpha)}{2}}$ which is always greater than the constant $\sqrt{\frac{\mathrm{log}(1/\alpha)}{2}}$ used in the case of the DKW calibration. Note that this slight change on $\tilde{\lambda}_\alpha$ has consequences as it is discussed in \secref{sec_num}.\\
This envelope has the advantage, like all other non adaptive envelope, to upper bound $\abs{R(t) \cap \cH_0}$ deterministically, and thus gives the experimenter the opportunity to decide which bound to use according to the domain of $p$-values on which they want the better envelope (without looking at the $p$-values).

\subsubsection{Adaptive bound}\label{sec_adap_bretenv}
The following bounds is the adaptive version of the Bretagnolle calibration. 
\thml{
Let $\alpha\in (0,1]$. Assume \autoref{ass_indep} and \autoref{ass_hetero}. Define $\mzsc$ by \eqref{bretm0}, $\tilde{\lambda}_\alpha=\sqrt{\frac{1+\mathrm{log}(1/\alpha)}{2}}$ and $\Rfam = (R_k,\zeta_k)_{1 \leq k \leq m}$ where for all $ k\in \Nm,$
\begin{equation}\label{zetabretadaptive}
  \zeta_k=k \wedge \mzsc \wedge \left\lfloor \sum _{1 \le i \le \mzsc}\paren[2]{\mathbf{F}\paren[1]{p_{(k)}}}_{[i]}+ \sqrt{\mzsc }\tilde{\lambda}_\alpha\right\rfloor.
\end{equation}
Then, $ \hat{V}^{\JER}_\Rfam $ is a confidence bound of level $\alpha$.
}{adaptivetopk}

\begin{proof}
By \autoref{adaptivebretin}, $\Pro{\exists t \in [0,1],   i_{\cH_0}\paren{t} >  \sum_{1\le k \le \mzsc}\paren[1]{\mathbf{F}(t)}_{[k]} + \sqrt{\mzsc}\tilde{\lambda}_\alpha} \leq \alpha $ because $\set{\sup\limits_{ t \in [0,1]}\sum\limits_{k\in \cH_0} \mathbb{G}_k(t)\leq \sqrt{m_0} \tilde{\lambda}_\alpha} = \set{\varphi_{\cH_0} = 0} \subseteq \set{m_0\leq \mzsc}$ with $\varphi_{\cH_0}$ defined by \eqref{eq_local_bretagnolle_1}.
The conclusion follows from the inequality $\abs{\cH_0\cap R_k}\leq i_{\cH_0}\paren{p_{(k)}}$ for all $k\in \Nm$.
\end{proof}
Note that the previous confidence bound is computed in time $\mathcal{O}\paren[1]{m^2\log(m)}$, because the family $\paren{\zeta_k}_{k\in \Nm}$ is computed in time $\mathcal{O}\paren[1]{m^2\log(m)}$. Then, the effective computation of $\hat{V}^{\JER}_\Rfam(S)$ is in time $\mathcal{O}\paren{\abs{S}}$, which is convenient for exploratory purposes.

This adaptive top-$k$ procedure reduces to the one derived from the DKW inequality described in \autoref{DKWtopkopt} when $F_k = F_1$ for all $k\in \Nm$ with $\tilde{\lambda}_\alpha$ instead of $\lambda_\alpha$. Thus, when no heterogeneity occurs in the data, the adaptive top-$k$ confidence bound, $\hat{V}^{\JER}_{\Rfam}$ with $\Rfam$ defined  in \autoref{adaptivetopk} is equal to the bound $\hat{V}^{\IP}_\varphi$. Hence, under this assumption, $\hat{V}^{\JER}_{\Rfam}\leq \hat{V}^{\SCo}_{\varphi}$. Nevertheless, under heterogeneity, the inequality $\hat{V}^{\SCo}_{\varphi}(S) <\hat{V}^{\JER}_{\Rfam}(S) $ can occur, as it is illustrated by the next example.

\exl{Let $\alpha > \exp\paren{-\frac{5}{3}}$ and $\frac{2 - \tilde{\lambda}_\alpha\sqrt{3}}{2}>\varepsilon>  0$. Define $a_1 = \frac{2 - \tilde{\lambda}_\alpha\sqrt{3}}{2}-\varepsilon_1$ and $a_2 = \frac{2 - \tilde{\lambda}_\alpha\sqrt{3}}{2}$. Assume that $p_1 = a_1$, $p_2=a_2$, $p_3 = a_2$ and $p_4 = 1$ are such that under their respective null hypotheses, \begin{itemize}
\item $\mu_{p_1} = a_1\delta_{a_1} + \paren{1-a_1}\delta_{1}$,
\item $\mu_{p_2} = a_2\delta_{a_2} + \paren{1-a_2}\delta_{1}= \mu_{p_3}$,
\item $\mu_{p_4} = \delta_{1}$.
\end{itemize}
Then, $\hat{V}^{\SCo}_{\varphi}\paren[1]{\set{1,2,4}} <\hat{V}^{\JER}_{\Rfam}\paren[1]{\set{1,2,4}} $ with $\Rfam$ defined  in \autoref{adaptivetopk}.}{SC1betterJER}

In \appref{proof_SC1betterJER}, we show the statement of \autoref{SC1betterJER}.
\rem{Since the validity of $\hat{V}^{\SCo}_{\varphi}$ and $\hat{V}^{\JER}_{\Rfam}$ 
rely on the same event that holds with high probability·, we define a confidence bound
with the following:
$\hat{V} = \hat{V}^{\SCo}_{\varphi}\wedge \hat{V}^{\JER}_{\Rfam}$. This confidence bound inherits
of the time complexity of the two bounds, but is valid and uniformly more powerful than each other. }

\propl{The adaptive top-$k$ procedure described in \autoref{adaptivetopk} is a uniform improvement of the inversion procedure applied to the homogenized Bretagnolle local test family: \[ \forall A \subseteq \Nm, \varphi_A^{hom}=\max\limits_{t\in \RR} \indev{i_A(t)> f_{|A|}(t)}, \] where, for all $t\in \RR$, and $n\in \Nm$, $f_{n}(t) = \sum\limits_{1 \leq i \leq n} \paren[1]{\mathbf{F}\paren{t}}_{[i]} + \tilde{\lambda}_\alpha\sqrt{n}.$}{jerbetterhom}

The proof of \autoref{jerbetterhom} is deferred to \appref{proof_jerbetterhom}

\subsection{Simultaneous \texorpdfstring{$k$}{k}-FWER} \label{kfwerbounds}

In this section, recall that the $k$-FWER setting assume to have a Simes-type inequality, that is, equivalently, some  family of local test of level $\alpha$ expressed as in \eqref{locform}. Then, the appropriate confidence envelope can be obtained as $\hat{V}^{\JER}_{\Rfam}$, where $\Rfam = \paren[1]{\set{i \in \Nm : p_i \leq \ell_{k:n}}, k-1}_{1\leq k\leq m}$ with $n$ possibly data dependant, such that $\Pro{\exists k \in \llbracket 1 , m_0 \rrbracket : p_{\paren{k:\cH_0}}\leq \ell_{k:n}}\leq \alpha$.
\subsubsection{Non adaptive envelope}
Applying \autoref{corNmtoH0} and \autoref{thresholds_ddr}, we easily derive a confidence envelope: 
\thm{Let $\alpha \in [0,1]$ and assume \autoref{ass_indep} and \autoref{ass_hetero}. Define a threshold family $\tau$  by \eqref{eq_tau_AHSU} and $\Rfam = \paren[1]{\set{i \in \Nm : p_i \leq \tau_k}, k-1}_{1\leq k\leq m}$.\\ Then, $\hat{V}^{\JER}_\Rfam$ is a confidence bound of level $\alpha$.}
This confidence envelope gives a convenient over-estimator of $m_0$ to get an adaptive confidence envelope. Indeed, \begin{equation}\label{m0intersimes}
\hat{m}_0^{\JER} := \hat{V}^{\JER}_\Rfam\paren[1]{\Nm} = \min\limits_{1\leq k\leq m}\set[2]{\abs[1]{\set{i\in \Nm : p_i> \tau_k}}+k-1}
\end{equation}
verifies the assumption $\set{\forall i \in \cH_0 \text{, } p_{(i : \cH_0)} > \tilde{\tau}_i}\subseteq \set{m_0 \leq \hat{m}_0}$ from \autoref{adap_thresholds}, because $\set{\forall i \in \cH_0 \text{, } p_{(i : \cH_0)} > \tilde{\tau}_i}\subseteq\set{\forall i \in \cH_0 \text{, } p_{(i : \cH_0)} > \tau_i}$. Thus it will allow us to define a family of adaptive thresholds thanks to \autoref{adap_thresholds}.
\subsubsection{Adaptive envelopes} \label{adapsimes}
Recall here, for $\paren{s,t}\in \RR^2$ such that for all $i\in \Nm$, $F_i(s) <1$, $\mathbf{H}\paren{s,t} = \paren{\frac{F_i(t)}{1- F_i(s)}}_{i\in \Nm}$.
We provide here two adaptive envelopes derived from \secref{sec_hetero_simes}. The first one will be computed thanks to the expression of $\hat{m}_0^{\JER}$ defined by \eqref{m0intersimes}: 
\thml{Let $\alpha \in [0,1]$ and assume \autoref{ass_indep} and \autoref{ass_hetero}. Define a threshold family $\tau^{(1)}$  by  \eqref{eqadaptau}, with $\hat{m}_0^{\JER}$ expressed by \eqref{m0intersimes}, $\lambda = \tau_{\hat{m}_0^{\JER}} $ and \[\Rfam = \paren[2]{\set{i \in \Nm : p_i \leq \tau_k^{(1)}}, k-1}_{1\leq k\leq \hat{m}_0^{\JER}+1}.\] 
Then, $\hat{V}^{\JER}_\Rfam$ is a confidence bound of level $\alpha$.}{adapenv1}

\begin{proof}
Define $\tilde{\tau}$ by \autoref{thresholds_ddr2} with $\lambda = \tau_{m_0}$ and $U= \cH_0$. Assume that the event $\set{\forall i \in \Nmz,  p_{(i : \cH_0)} > \tilde{\tau}_i }$ occurs. Then $m_0 \leq \hat{m}_0^{\JER}$ therefore $\tau_{\hat{m}_0^{\JER}} \geq \tau_{m_0}$, thus $\tau^{(1)}_k\leq \tilde{\tau}_k$ for all $k\leq \hat{m}_{0}^{\JER}$, and \autoref{adap_hetero_simes} concludes.
\end{proof}

Observe that 
we can
 re-enact this step of adaptivity to improve our bound again. Thus, a strategy would be to repeat this plug-in steps until $\hat{m}_0^{\JER}$ stops decreasing. This step-down strategy is classical in multiple testing literature and can be traced back to \citet{Holm1979}. It has also been used for confidence envelopes in \citet{blanchard2020post}.

The second envelope requires a more stringent assumption on $\lambda$.
\thml{Let $\alpha \in [0,1]$ and assume \autoref{ass_indep} and \autoref{ass_hetero}. Define a threshold family $\tau^{(2)}$  by \eqref{eqadaptau} with $\lambda$ verifying \eqref{lambdaDDR}, that is
\begin{equation*}
\lambda \in \set[2]{t\in \cA: \forall i\in\Nm, F_i(t)<1, \text{ and } \paren[1]{\mathbf{H}(t,t)}_{[1]} \leq \alpha },
\end{equation*} and $\mzsc$ defined by \eqref{DDRm0} 
and 
\begin{equation} \label{rfamhsimes}
\Rfam = \paren[2]{\set{i \in \Nm : p_i \leq \tau_k^{(2)}}, k-1}_{1\leq k\leq \mzsc+1}. 
\end{equation}

Then, $\hat{V}^{\JER}_\Rfam$ is a confidence bound of level $\alpha$.
}{adapenv2}

\begin{proof}
Apply \autoref{adap_thresholds} with $\hat{m}_0=\mzsc$.

\end{proof}
This assumption on $\lambda$ will generally be benign with respect to the thresholds range; indeed, in the case of the vanilla Simes inequality thresholds, the larger threshold is equal to $\alpha$,  and here, under the assumption that $F_i(t) = t$ for all $t\in \mathcal{A}_i$, $\lambda = \frac{\alpha}{1+\alpha}$ satisfies \eqref{lambdaDDR}.

\propl{Define for $n \leq m$ and $\lambda$ which verifies \eqref{lambdaDDR}, and define for all $i \leq n$, \[ \ell_{i:n} =  \mathrm{max}\set[2]{t\in \cA, t\leq \lambda : \paren[1]{\mathbf{H}(\lambda,t)}_{[1]} + \dots + \paren[1]{\mathbf{H}(\lambda,t)}_{[n -i+1] }\leq i\alpha}.\]
The family $\paren{\ell_{i:n}}$ yields a homogeneous Simes-like local test family at level $\alpha$: 
\begin{align*}
\forall A \subseteq \Nm, \varphi_A^{hom} &= \max_{t\in[0,\lambda]} \indev{\displaystyle\sum ^{|A|-i_A(t)+1}_{\ell=1}\paren[1]{\mathbf{H}\paren[0]{\lambda,t}}_{[\ell]} \leq i_A(t)\alpha}\\
& =\max_{1\leq i \leq |A|} \indev{p_{(i:A)}\leq \ell_{i:|A|}}.
\end{align*}
Then, $\hat{V}^{\JER}_\Rfam \leq \hat{V}^{\SCo}_{\varphi^{hom}}$ with $\Rfam$ defined by \eqref{rfamhsimes}.
 }{jerbettersimes}

The proof of the previous proposition is given in \appref{proof_jerbettersimes}.

\reml{The adaptive envelopes from \autoref{adapenv1} and \autoref{adapenv2} differ from the shortcut for the inversion procedure described by \autoref{SCgoemann} applied to $\varphi^{hom}$ by the over estimators of $m_0$ considered. The one provided by \autoref{SCgoemann} is \[\hat{m}_0^{hom} = \inf\limits_{t\in (0,\lambda]} \paren{\left\lceil\frac{1}{\alpha}\sum^{m-i(t)+1}_{k=1}\paren[1]{\mathbf{H}\paren[0]{\lambda,t}}_{[k]}\right\rceil + m-\paren{i(t)+1}} \wedge \paren[1]{m- i(0)},\] and is lower bounded by $\hat{m}_0^{\SCo}$ expressed by  \eqref{DDRm0}, hence the adaptive envelope from \autoref{adapenv2} is uniformly better than the one provided by the inversion procedure applied to the homogenized local tests.}{remhomsimes}

\subsection{Deterministic setting}
This section focuses on applying the Bretagnolle $m_0$ estimation on the framework from \citet{durand2020post}, that we recall here: assume that the family $\paren{R_k}_{k\in \cK}$ is deterministic, that is $\cK$ is deterministic and, for all $k\in \cK$, $R_k$ is deterministic. 

\thm{Let $\alpha \in (0,1]$. Assume \autoref{ass_hetero} and that for all $k\in \cK$, $R_k\cap \cH_0$ is a family of independent variables. For all $k\in \cK$, define: \[\zeta_k = \max\set{n \in \llbracket 0,\abs{R_k} \rrbracket : \forall i \in R_k,  \sum\limits_{1\leq j\leq n}\paren[1]{\mathbb{G}(p_i)}_{\paren{j:R_k}}\leq \tilde{\lambda }_{\frac{\alpha}{|\cK|}}\sqrt{n}} .\] 
Then, the reference family $\Rfam = \paren{R_k,\zeta_k}_{k\in \cK}$ has its $\JER$ bounded by $\alpha$.
}

\begin{proof}
For all $k\in \cK$, by \autoref{BretSC1} applied with $R_k$ instead of $\Nm$ and $\frac{\alpha}{|\cK|}$ instead of $\alpha$, $\Pro[1]{\abs{R_k\cap\cH_0} > \zeta_k}\leq \frac{\alpha}{|\cK|}.$ It follows, by union bound, that \[\JER(\Rfam) = \Pro[1]{\exists k \in \cK : \abs{R_k\cap\cH_0} > \zeta_k}\leq |\cK| \frac{\alpha}{|\cK|} = \alpha.\]
\end{proof}
\section{Numerical experiments}\label{sec_num}
The purpose of this section is to empirically visualize the gain of power on confidence envelopes by considering the heterogeneity of the data. We will compare our new procedures with the ones defined for the homogeneous setting, which we recall is the classical setting that was investigated in previous works.

\subsection{Simulation setting}
We follow the simulation setting from \citet{dohler2018new} to generate our data. We recall here this setting:
we simulate two groups of $N = 50$ subjects, and for each subject, we observe $m$ independent binary responses (\textit{e.g.} treatment effects, gene variants, ...). We denote by $\beta_{1i}$ and $\beta_{2i}$ the success probability, respectively for the first group and the second group. Then, we simultaneously test the null hypotheses $H_i : ``\beta_{1i}=\beta_{2i}"$ by the mean of two-sided Fisher's exact tests producing for each $i \in \Nm$ a $p$-value $p_i$ with its \textit{c.d.f.} $F_i$. We will then have three distinct sets of $p$-values. The first one has a cardinality of $m_0' < m$ and the paired $p$-values is generated by underlying data such that $\beta_{1i}=\beta_{2i} = 0.01$. We denote by $\pi_0' = m_0'/m_0$ the proportion of null hypotheses lying in this set of null $p$-values, and it will be set to $0.5$ for our experiments. The second set of $p$-values has a cardinality of $m_0''< m-m_0'$ and will be generated by underlying data such that $\beta_{1i}=\beta_{2i} = 0.1$. That is, $m_0=m_0'+m_0''$. Finally, the third set of $p$-values has a cardinality $m_1 = m - \paren{m_0' + m_0''}$, determined by the proportion of effects $1-\pi_0$, where $\pi_0=m_0/m$, and generated by underlying data such that $\beta_{1i}= 0.1$ and $\beta_{2i} = q$, where $q$ represents the strength of the signal. For all of our simulations, $\pi_0$ and $q$ will be specified and we will set $ m = 2000$, $\pi_0' =0.5$. 
Furthermore, for each couple $(\pi_0,q)$ we draw $m$ contingency tables, compute the associated $p$-values
and apply our methods to this unique set of $p$-values.

\subsection{Analyses for top-\texorpdfstring{$k$}{k} setting}
This section will focus on our results in the top-$k$ setting. We will compare our Bretagnolle bounds for heterogeneous data with the DKW bounds provided in the homogeneous case. Note that in the top-$k$ setting, $\zeta_k = \zeta\paren[0]{p_{(k)}}$ derives from a more general bound $\zeta(t)$ on $\abs{R(t) \cap \cH_0}$  which is deterministic when it derives from non-adaptive bounds. In this kind of simulation, an upper bound $\zeta\paren{\cdot}$ on the number of false discoveries performs better if it has small values. We recall here the expression of $\zeta\paren{\cdot}$ for the Bretagnolle and DKW inequalities \autoref{thm_families_topk_homogeneous}:
\[\zeta^{DKW}(t) = \floor{mt + \sqrt{m} \lambda_\alpha}, \zeta^{Bret}(t) = \floor{m\bar{F}(t) + \sqrt{m} \tilde{\lambda}_\alpha}, \]
and their adaptive counterparts:
\[\zeta^{DKW}_{adap}(t) = \floor{\hat{m}_{0,1} t + \sqrt{\hat{m}_{0,1}} \lambda_\alpha}, \zeta^{Bret}_{adap}(t) = \floor{ \sum _{1 \le i \le \hat{m}_{0,2}}\paren[1]{\mathbf{F}(t)}_{[i]} + \sqrt{\hat{m}_{0,2}} \tilde{\lambda}_\alpha}, \]
where $\hat{m}_{0,1}$ is defined by \eqref{eq:DKWhatm0nonlocal} and $\hat{m}_{0,2}$ is defined by \eqref{bretm0}.
\begin{figure}[ht]
\begin{center}
\includegraphics[scale=0.55]{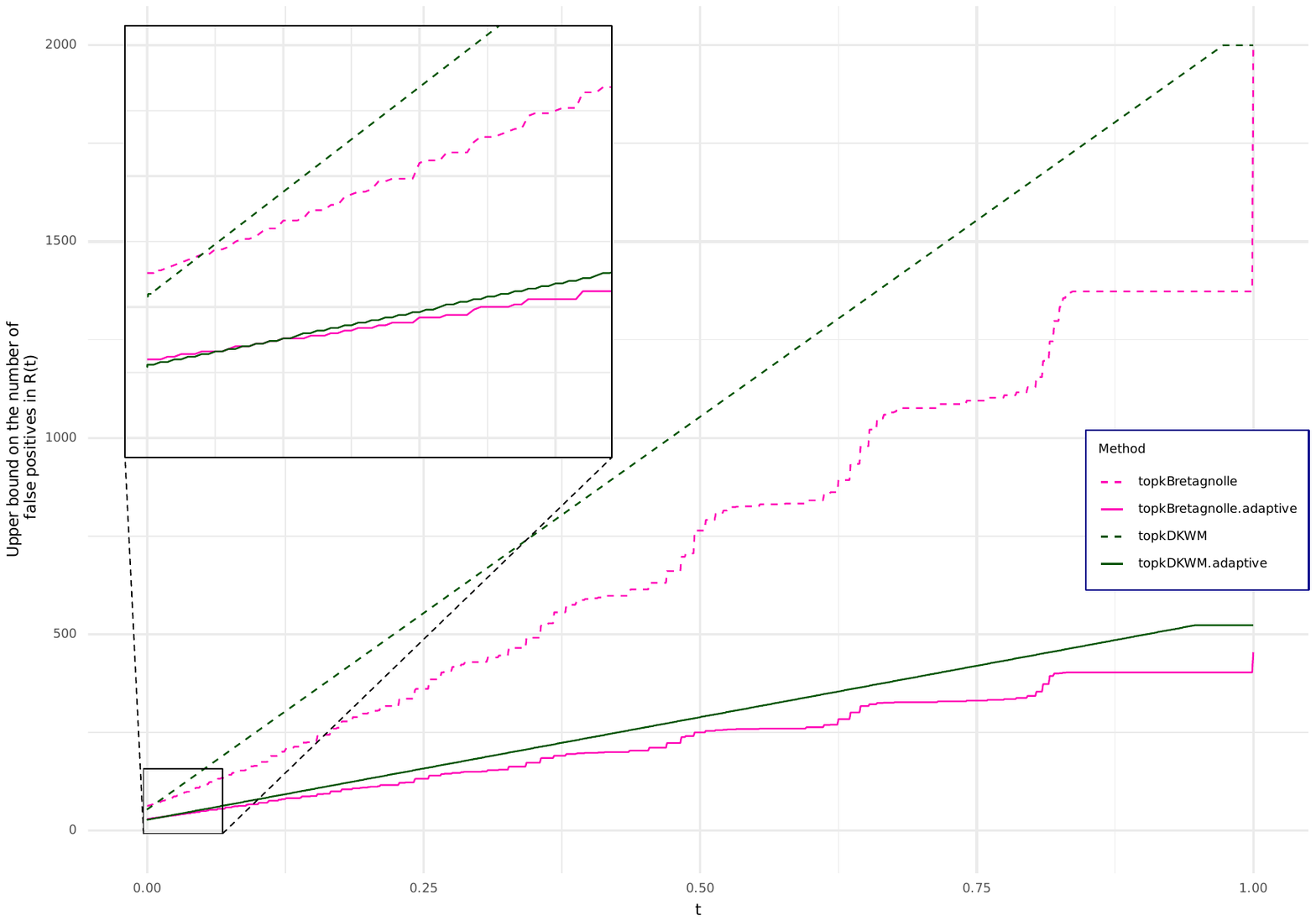}
\end{center}
\caption{Upper bounds $\zeta(\cdot)$ provided by the DKW and Bretagnolle inequalities, with their respective adaptive bounds. We set $\pi_0 = 0.2$ and $q = 0.4$. Lower is better.}
\label{envelopesbret}
\end{figure}

Here, we expect the DKW bounds to be slightly better than the Bretagnolle bounds on smaller $p$-values --- due to the difference between $\lambda_{\alpha}$ and $\tilde{\lambda}_{\alpha}$, see \autoref{lambdabret} --- if the number of $p$-values in the area near $0$, where $\zeta^{DKW}(\cdot)< \zeta^{Bret}(\cdot)$
, oversteps the $\sqrt{m}\lambda_{\alpha}$ or $\sqrt{\hat{m}_0}\lambda_{\alpha}$ term. In \autoref{envelopesbret}, we plot the four functions $\zeta^{DKW}, \zeta^{Bret}, \zeta^{DKW}_{adap}, \zeta^{Bret}_{adap}$, for one realization of $(\hat{m}_{0,1}, \hat{m}_{0,2})$. The lowest a function is at a given point $t$, the better the estimation of $\abs{R(t) \cap \cH_0}$. We see that, indeed, the functions based on the DKW inequality are better than their Bretagnolle counterparts for small values of $t$. We also remark that there is a huge gap between the adaptive and the vanilla methods, in favor of the adaptive ones, because the signal is strong and large (we used $\pi_0 = 0.2$ and $q = 0.4$), which is a favourable case for adaptive methods.

\begin{figure}[ht]
\begin{center}
\includegraphics[scale=0.55]{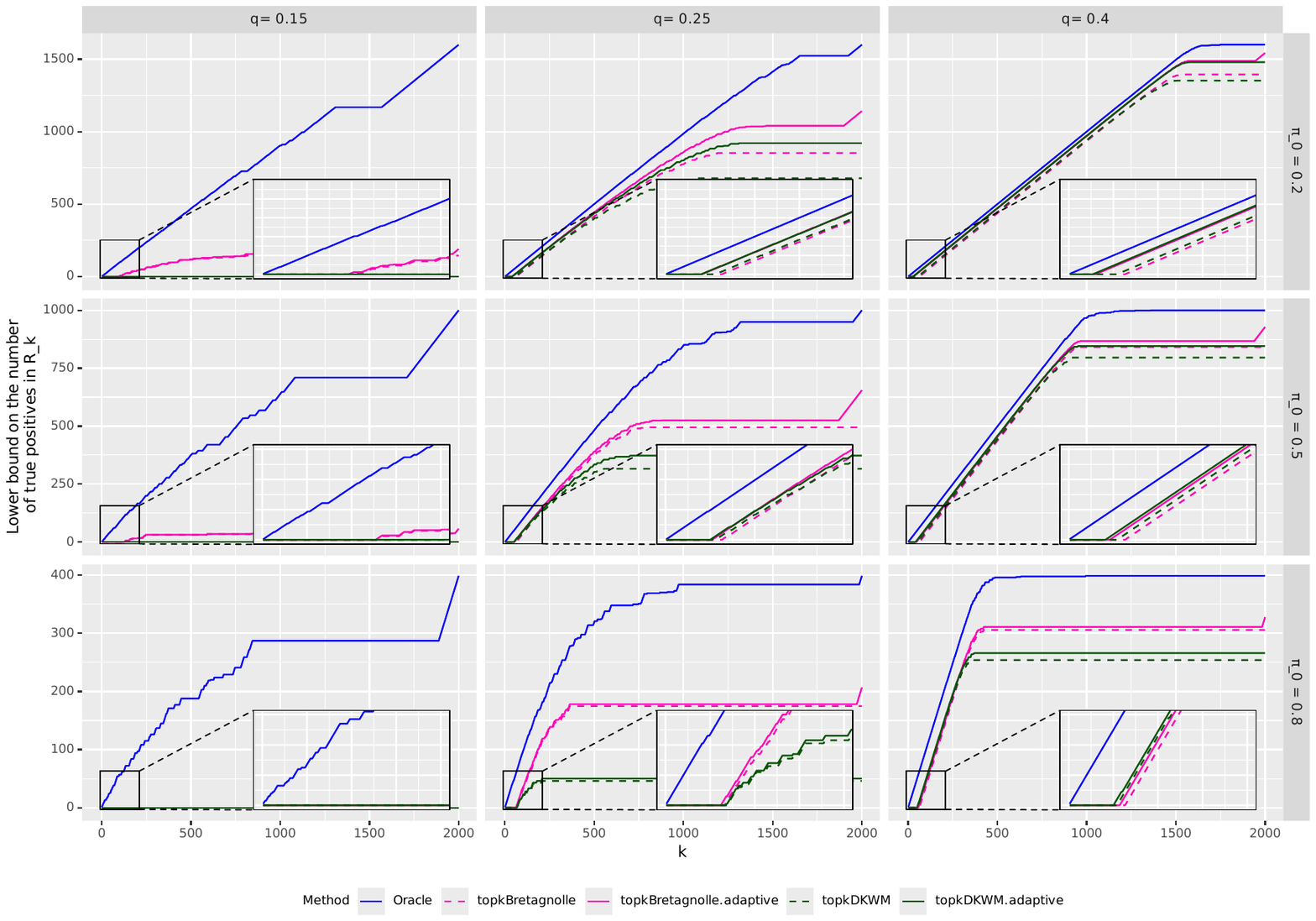}
\end{center}
\caption{Lower bounds on number of true positives in $R_k$ provided by DKW and Bretagnolle inequalities, with their respective adaptive bounds. Higher is better.}
\label{LBbret}
\end{figure}
\autoref{LBbret} illustrates the previous analysis of 
for different values of $q$ and $\pi_0$ by plugging the $p$-values to the bounds $\zeta\paren{\cdot}$.  We compare here the lower bound provided by the top-$k$ methods, derived from the sequence of inequalities $\abs{\cH_1 \cap S} = \abs{S}- \abs{\cH_0 \cap S} \geq \abs{S}- \hat{V}(S) $. Note that here, the Oracle curve is the number of elements from $R_k$ in $\cH_1$. As we expected, in settings with strong signal or with moderately strong and moderate to high signal proportion, the Bretagnolle bound is more conservative than the DKW bound for $R_k$ such that $p_{(k)}$
is small. This last phenomena is due to the corrective term from the Bretagnolle inequality to consider heterogeneity. However when the signal is weak, the Bretagnolle bound detects signal when the DKW inequality does not.

\subsection{Analyses for \texorpdfstring{$k$}{k}-FWER setting}

\begin{figure}[ht]
\begin{center}
\includegraphics[scale=0.55]{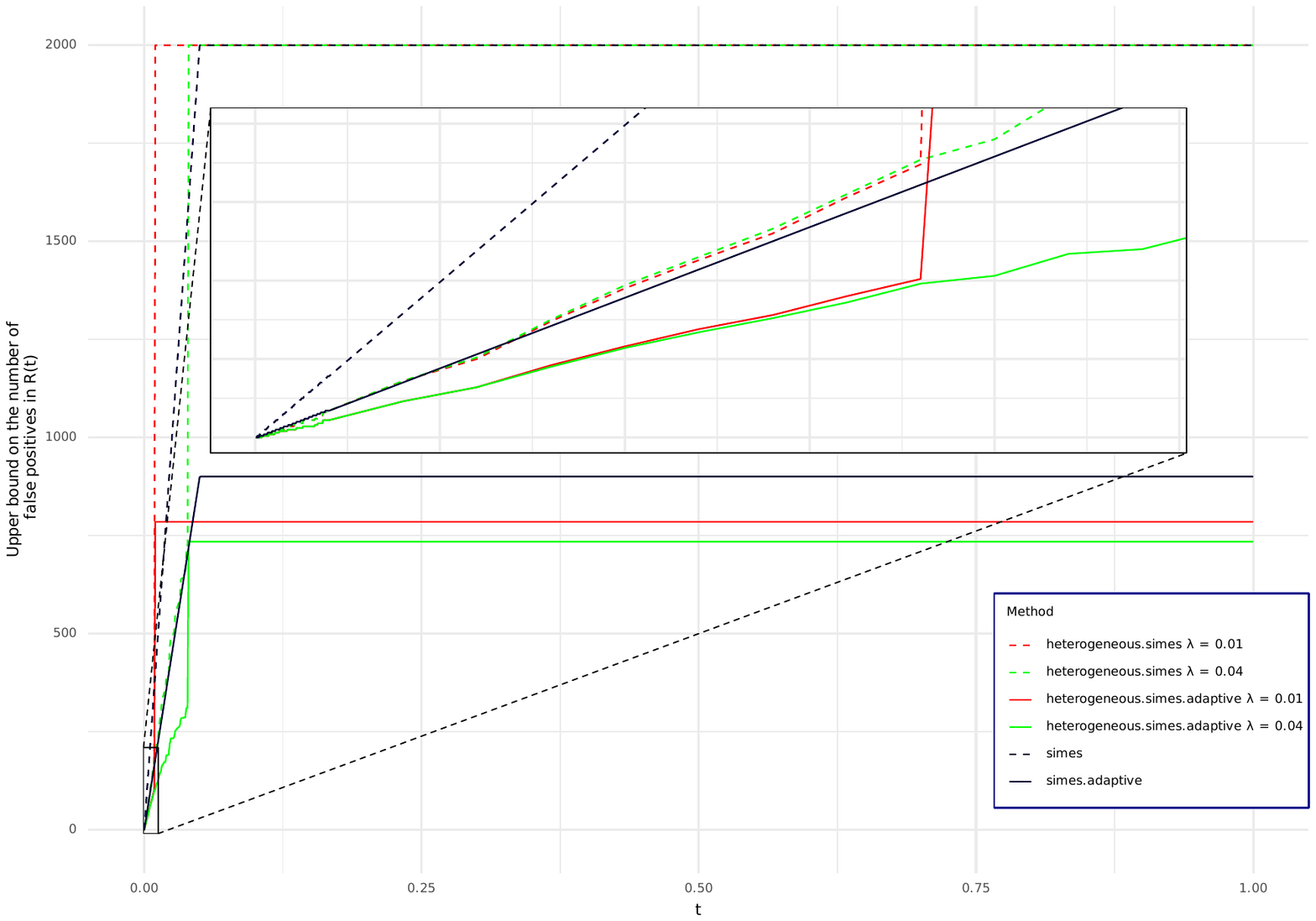}
\end{center}
\caption{Upper bounds $\zeta(\cdot)$ provided by Simes and heterogeneous Simes inequalities, with their respective adaptive bounds. We set $\pi_0 = 0.2$, $\pi_0' = 0.5$ and $q = 0.4$. Lower is better.}
\label{envelopessimes}
\end{figure}

Likewise to the top-$k$ case, one can construct an upper bound $\zeta(t)$ on $\abs{R(t) \cap \cH_0}$ in the $k$-FWER setting. To obtain this bound we define, for a thresholds family $\tau$ and all $t\in [0,1]$, $\zeta(t) = \abs{\set{i\in \Nm : \tau_i < t}}$. Note that it can be checked that this construction is exactly the same as the one displayed in \secref{sec_kFWER_to_topk}. We can observe two phenomena in \autoref{envelopessimes} ; when we reduce the parameter $\lambda$ in the heterogeneous Simes-like procedures, we get bounds slightly better in the non-adaptive fashion. This phenomenon was expected, since we enlarge 
the denominator of the coordinates of $\mathbf{H}\paren{\cdot,\cdot}$. However, it
deteriorates the bound in the adaptive fashion. 
It is a consequence of limiting the computation of $\hat{m}_0$ to $p$-values smaller than $\lambda$.

\begin{figure}[ht] \label{LBsimes}
\begin{center} 
\includegraphics[scale=0.55]{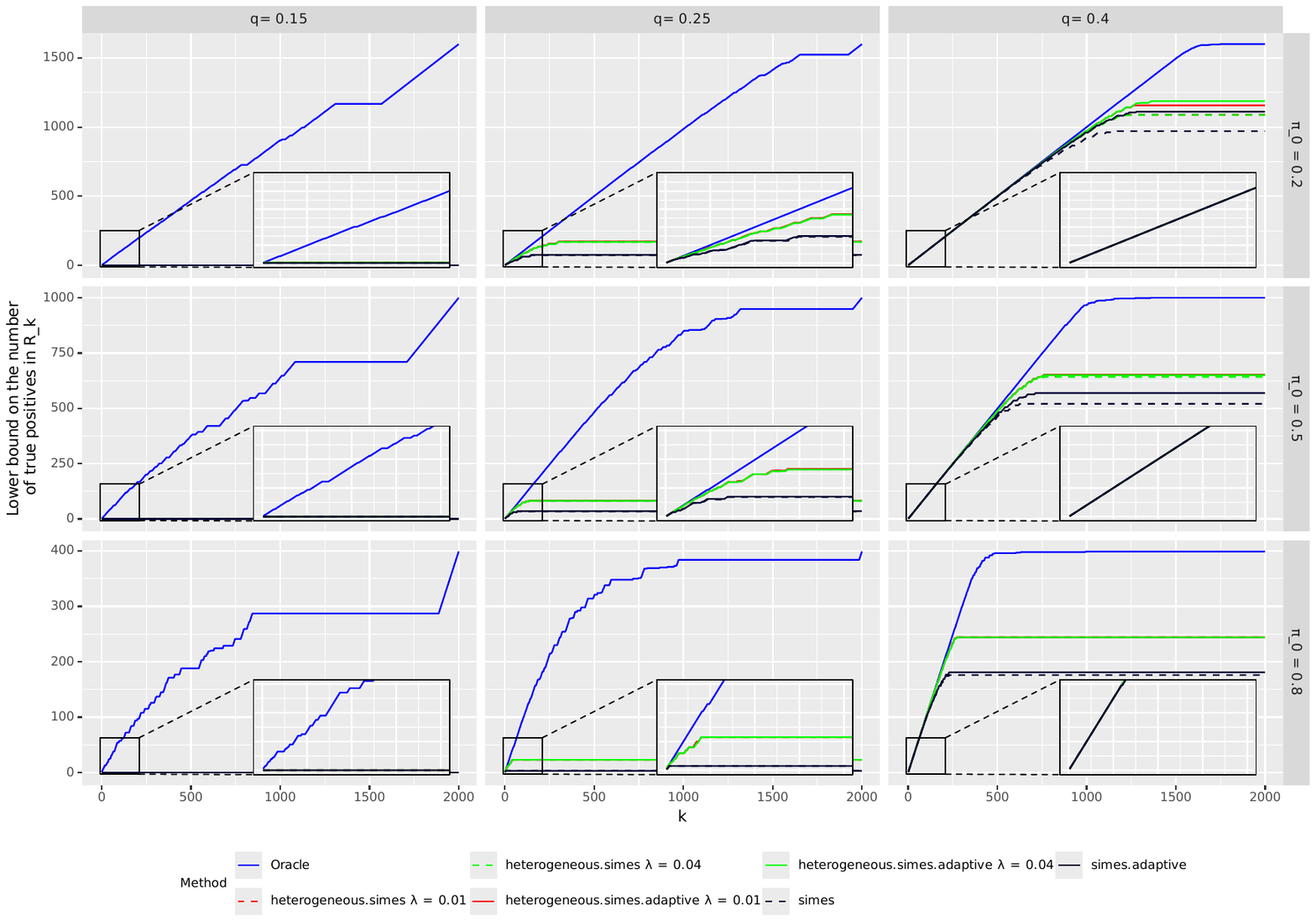}
\end{center}
\caption{Lower bounds on number of true positives in $R_k$ provided by Simes and heterogeneous Simes inequalities, with their respective adaptive bounds. Higher is better.}
\end{figure}

\autoref{LBsimes} shows efficiency of Simes-like confidence envelopes via lower bounds on the number of true positives in $R_k$. Likewise to the vanilla Simes method and accordingly to power studies from \citet{dohler2018new}, this class of $k$-FWER methods does not detect the signal if it is weak. However, with moderate and strong signal, the heterogeneous Simes methods will detect the signal, and is less conservative than the vanilla Simes method. The improvement using adaptive methods is substantial when the signal is strong and wide, and the slight difference between the adaptive methods mentioned earlier in the analysis of \autoref{envelopessimes} is visible for the couple of parameters $\paren{q,\pi_0} = \paren{0.4,0.2}$.
However, these improvements are slight or inexistent in the other settings hence these heterogeneous bounds inherits of the vanilla Simes inequality. 

\subsection{Comparison with other homogeneous bounds}
We compare our new envelopes for heterogeneous data with the bounds from \citet{MBR2024} and \citet{Katsevichramdas2020} provided in the homogeneous case. We only show here the adaptive versions of the bounds. As it is illustrated by \autoref{envelopesadap},
there is not an uniformly better upper bound over $[0,1]$. While the heterogeneous Simes-like bounds are the least conservative for the smallest $t$, this trend reverses as $t$ becomes less extreme --- but still inferior to approximatively $0.05$. For these values of $t$, the KR and Wellner confidence bounds are less conservative than the heterogeneous envelopes developed in the present article. However, for larger $t$, the less conservative bound is the Bretagnolle bound.

\begin{figure}[ht]\label{envelopesadap}
\begin{center}
\includegraphics[scale=0.55]{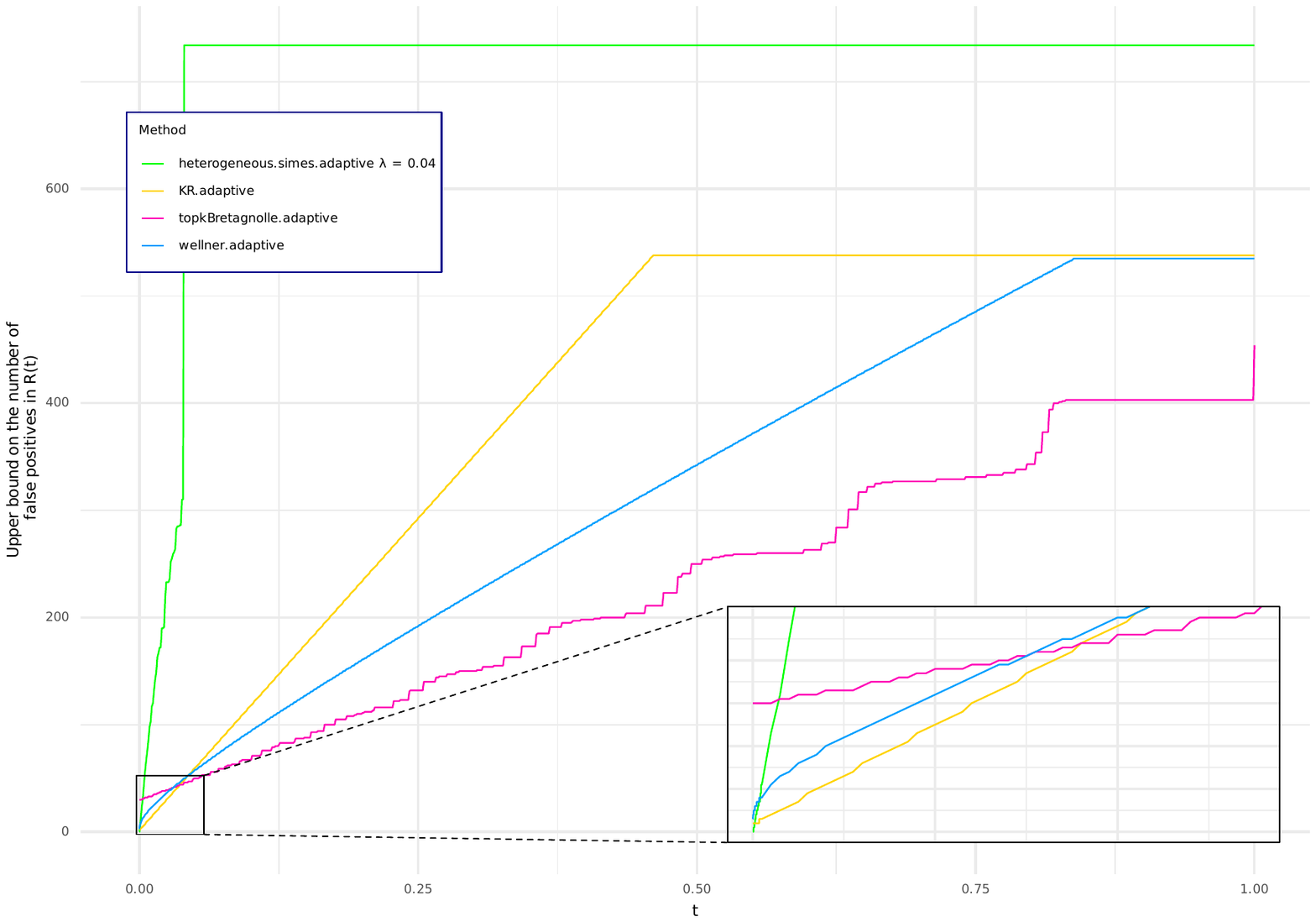}
\end{center}
\caption{Upper bounds $\zeta(\cdot)$ provided by adaptive bounds from \citet{MBR2024}, \citet{Katsevichramdas2020} and our adaptive heterogeneous bounds. We set $\pi_0 = 0.2$, $\pi_0' = 0.5$ and $q = 0.4$. Lower is better.}

\end{figure}
The previous analysis is corroborated by \autoref{LBadap} where we plugged the $p$-values in these adaptive bounds.

\begin{figure}[ht]\label{LBadap}
\begin{center}
\includegraphics[scale=0.55]{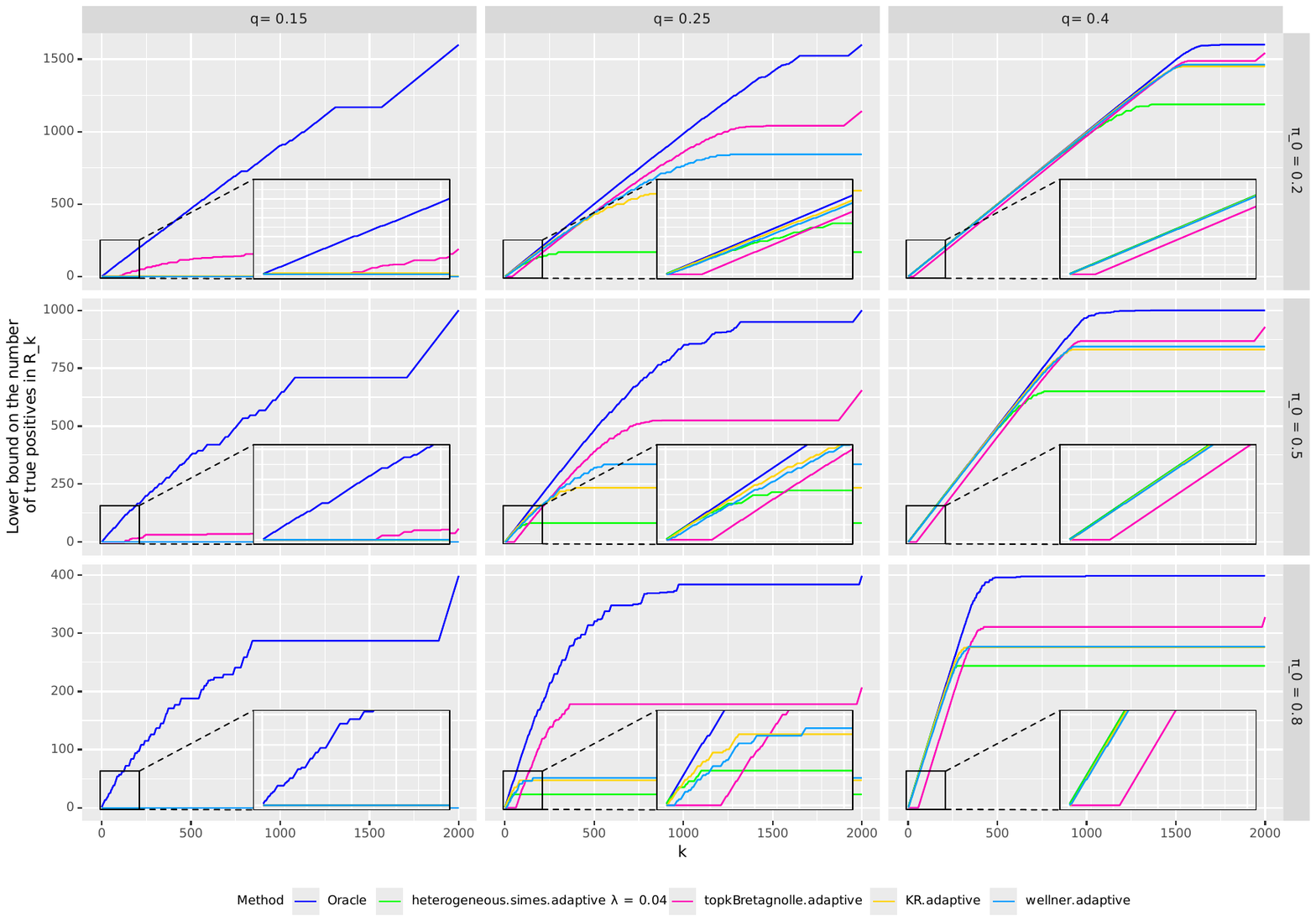}
\end{center}
\caption{Lower bounds on number of true positives in $R_k$ adaptive bounds from \citet{MBR2024}, \citet{Katsevichramdas2020} and our adaptive heterogeneous bounds. Higher is better.}

\end{figure}

\section{Conclusion and perspectives}\label{sec_ccl}

In this paper, we introduced new confidence envelopes for false discoveries by incorporating heterogeneous upper bounds on the cumulative distributive functions of null distribution. Our approach bridges key frameworks in the existing literature on confidence envelopes --- specifically the JER approach \citep{blanchard2020post} and the inversion procedure using
local tests \citep{Genovese2006,goeman2011multiple}.

We established generic novel shortcuts for the inversion procedure, providing significant new insights even in the standard homogeneous case. For instance, applied to the complete set of hypotheses,
the new shortcuts are exact (they coincide with the
theoretical inversion procedure) and thus provide
optimal (over-)estimators for $m_0$. Furthermore, 
using these over-estimators as plug-ins into the
JER framework yields a bound which coincides for all
rejection sets with
the inversion procedure in the homogeneous case, formally validating the optimality of several previous constructions from \citet{MBR2024}.

The methodologies introduced here extend naturally to the heterogeneous setting. 
By combination with heterogeneous ecdf inequalities 
(some of which are newly introduced in this work and
of independent interest), our novel shortcuts provide new confidence envelopes for heterogeneous tests and yield over-estimators for the number of null hypotheses.
These bounds are less conservative
than those obtained via a naive homogenization of the local test family. Consequently, when plugged in into the JER methodology, these estimators provide sharper, more efficient shortcuts.
Additionally, our new shortcuts for the inversion procedure also produce alternative confidence envelopes that can be less conservative than their JER counterparts in specific cases.

We validated our methods with numerical experiments highlighting the power gain of using envelopes tailored to the heterogeneity of the data compared to previous approaches designed for the homogeneous case.

A promising direction for future research involves deriving heterogeneous versions of the inequalities proposed by \citet{Katsevichramdas2020} or the uniform Wellner inequality from \citet{MBR2024}. Since these bounds outperform heterogeneous ones in specific $p$-value regions without requiring prior distributional information, getting heterogeneous versions could further enhance the power of these envelopes in heterogeneous context.

\section*{Acknowledgements}

This work has been supported by the French research grants ANR-19-CHIA-0021 (BISCOTTE), ANR-20-IDEES-0002 (PIA), ANR-21-CE23-0035 (ASCAI) and ANR-23-CE40-0018 (BACKUP).

\printbibliography[] 

\appendix

\section{Van Zuijlen inequality and confidence envelope}\label{appendix_VZ}
In this appendix, we focus on using the concentration inequality provided by \citet{shorack2009empirical} as a heterogeneous form of \autoref{prop_concen_topk_hmogeneous}. Unfortunately, this inequality turned out to not be sharp enough in our simulations, as it is outperformed by the heterogeneous Simes method..\\
Let us define the function $$h : x \mapsto xe^{x},$$ denote $W_{0}$ its inverse on $[-1, +\infty)$ (called principal real branch of the Lambert function) and $W_{-1}$ its inverse on $(-\infty, -1]$.

\begin{Theorem}[Van Zuijlen inequality]
Let $X_1, \dots, X_n$ be $n$ independent variables. Denote by $F_i$ their respective cdf and $\bar{F}$ the mean of these cdf. Then, for $\delta \in (0,1)$, with probability at least $1-\delta$, $$\forall t \in \mathbb{R}\text{, } \frac{1}{n} \sum _{1\leq i\leq n}\mathds{1}_{X_i \leq t} <  \frac{\bar{F}(t)}{-W_{0}\left(-\frac{\delta}{e(1+\delta)}\right)}.$$
\end{Theorem}
\begin{proof} 
Let $\delta \in (0,1)$. By Van Zuijlen inequality (see \citet{shorack2009empirical} chapter 25, from \citet{van1978properties}), for all $\lambda > 1$, 
$$\forall t \in \mathbb{R}\text{, } \frac{1}{n} \sum _{1\leq i\leq n}\mathds{1}_{X_i \leq t} \geq \lambda \bar{F}(t) $$ with probability at most $\frac{\exp(1-1/\lambda)(1/\lambda)}{1 - (1/\lambda)\exp(1-1/\lambda)}$. Solving $\frac{\exp(1-1/\lambda)(1/\lambda)}{1 - (1/\lambda)\exp(1-1/\lambda)} = \delta$ yields $\lambda = \frac{1}{-W_{-1}\left(-\frac{\delta}{e(1+\delta)}\right)}$ or $\lambda = \frac{1}{-W_{0}\left(-\frac{\delta}{e(1+\delta)}\right)}$. The fact that $\lambda >1$ then imposes that $\lambda = \frac{1}{-W_{0}\left(-\frac{\delta}{e(1+\delta)}\right)}$, as $W_{-1}$ takes its values in the set $(-\infty,-1]$. 
\end{proof}
\rem{There is no restriction on $\delta$ because $-\frac{1}{2}\mathrm{e}^{-1/2} \leq -\frac{1}{2}\mathrm{e}^{-1} \leq - \frac{\delta}{1+\delta} e^{-1}$. It follows that, by increasingness of $W_{0}$, for all $\delta \in (0,1)$, $-1/2  = W_{0}\left(-\frac{1}{2}\mathrm{e}^{-1/2}\right)\leq W_{0}\left(-\frac{\delta}{e(1+\delta)}\right)$, thus $\frac{1}{-W_{0}\left(-\frac{\delta}{e(1+\delta)}\right)} \geq 2 > 1$. }

\begin{Theorem}
Let $\alpha\in ]0,1[$. Assume that $(p_i)_{i\in \cH_0}$ is an independent family of $p$-values. Define
\[\zeta_k=k \wedge \frac{m}{-W_{-1}\paren{-\frac{\delta}{e(1+\delta)}}}\bar{F}\paren{p_{(k)}}.\] 
Then, the reference family $\Rfam = (R_k,\zeta_k)_{1 \leq k \leq m}$ has its $\JER$ bounded by $\alpha$.
\end{Theorem}
\begin{proof}
Use Van Zuijlen's inequality with $(p_i)_{i\in \cH_0}$, $n=m_0$, and upper bound $\sum\limits_{i\in \cH_0}\bar{F}\paren{p_{(k)}}$  by $\sum\limits_{1\leq i \leq m}\bar{F}\paren{p_{(k)}}$.
\end{proof}

\section{Bridging some gaps in the literature}\label{appendix_gaps}

In \citet{durand2020post}, in the proof of Proposition 1, the authors used the fact that $N_t(S)$ stochastically dominates $c=\sum_{i=1}^\nu\indev[a]{U_i>t}$, where the $U_i$ are independent uniform variables (we use the notations of that paper) to conclude. But that domination was only true at a fixed $t\in[0,1)$: for $t\in[0,1)$, for all $x\in\RR$, $\Pro[a]{N_t(S)\leq x}\leq \Pro[a]{c\leq x}$. Actually, we needed a much stronger result to be able to conclude, like the following: almost surely, $\tilde N_t(S)\geq c$ for all $t\in[0,1)$, for some process $\paren[a]{\tilde N_t(S)}_{t\in[0,1)}$ with the same distribution as $\paren[a]{N_t(S)}_{t\in[0,1)}$. This kind of result can be obtained by coupling but, given our \autoref{DKWin} (which is already proved by coupling), it is actually unnecessary given that now the $U_i$ can be considered super-uniform. In other words, by using the version of the DKW inequality that we gave here in \autoref{DKWin}, the gap in the proof is closed.

Also note that the proof of Proposition 1 in \citet{durand2020post} required $\lambda<1/2$, and so $\alpha/K<1/2$, because the authors used the version of the DKW inequality given by \citet{massart1990tight}. Thanks to \citet{reeve2024short}, \autoref{DKWin} does not have a similar constraint for its validity domain and so there is no more constraint on $\alpha$, as stated in \autoref{prop_dkwm_deter}.

Finally, in \citet{dohler2018new}, there was a slight oversight regarding the definition of $\cA$: 0 should be included in $\cA$, that is, we should have $\cA=\{0\}\cup\bigcup_{i=1}^m\cA_i$ instead of just $\cA=\bigcup_{i=1}^m\cA_i$. Otherwise, the sets that appear in the definitions of $\tau_m$ and $\tau_k$ for the [HSU], [HSD], [AHSU], [AHSD] and [HBR-$\lambda$] procedures may be empty, and so $\tau_m$ and $\tau_k$ may not be defined. This oversight has been corrected in the implementations of these procedures in the \texttt{R} package \texttt{DiscreteFDR} \citep{R-base, DiscreteFDR}.

\section{Complements on the top-\texorpdfstring{$k$}{k} setting}\label{appendix_ties}

\subsection{Top-\texorpdfstring{$k$}{k} with ties}\label{sub_appendix_ties}

Recall that in the top-$k$ setting, we defined $R_k$ with \equaref{eq_def_Rk_top_k}, while past works like \citet{MBR2024} used $R_k^\dag=R(p_{(k)})=\set{i\in\Nm: p_i\leq p_{(k)}}$. Also recall \autoref{rem_topk_nondecreasing_2}: we assume that $\zeta_k$ is of the form $f\paren[a]{p_{(k)}}\wedge k$ for some measurable function $f$. The next proposition shows that using $R_k$ or $R_k^\dag$ in that context is inconsequential, in the sense that it gives raise to the same confidence envelope.
\propl{Let $\zeta_k=f\paren[a]{p_{(k)}}\wedge k$ for all $k\in\Nm$, where $f$ is any measurable function with real values.

Let the following reference families: $\Rfam=\paren[a]{R_k, \zeta_k}$, $\Rfam^\sim=\paren[2]{R_k, f\paren[a]{p_{(k)}}}$, and  $\Rfam^\dag=\paren[2]{R_k^\dag, f\paren[a]{p_{(k)}}}$.

Then, for all $S\subseteq\Nm$,
\begin{align}
    \hat V^{\JER}_\Rfam(S)&=V^*_{\Rfam^\sim}(S)\label{eq_ties_1}\\
    &=V^*_{\Rfam^\dag}(S).\label{eq_ties_2}
\end{align}
}{prop_ties}
Note that the statement of \autoref{prop_ties} is purely deterministic. Also note that the statement of \autoref{prop_ties} implies that the choice of the permutation $\sigma$ that orders the $p$-values is irrelevant, given that $\Rfam^\dag$ does not depend on $\sigma$.
\begin{proof}
First, note that for any $A\subseteq \Nm$, $\abs[a]{A\cap R_k}\leq\zeta_k=f\paren[a]{p_{(k)}}\wedge k$ if and only if $\abs[a]{A\cap R_k}\leq f\paren[a]{p_{(k)}}$, because we always have $\abs[a]{A\cap R_k}\leq |R_k|=k$. This implies that $\Afam(\Rfam)=\Afam(\Rfam^\sim)$ (recall Equation \eqref{eq_a}). And so $V^*_{\Rfam}(S)= V^*_{\Rfam^\sim}(S)$ for all $S\subseteq\Nm$, by Equation \eqref{eq_vstar}. Hence \eqref{eq_ties_1} holds.

Note that $R_k\subseteq R_k^\dag$ and so if $\abs[a]{A\cap R_k^\dag}\leq f\paren[a]{p_{(k)}}$, then $\abs[a]{A\cap R_k}\leq f\paren[a]{p_{(k)}}$ also. Hence $\Afam(\Rfam^\dag)\subseteq \Afam(\Rfam^\sim)$.

Denote $g(k)=\max\set{j:p_{(j)}=p_{(k)}}$ and recall that $p_{(j)}=p_{\sigma(j)}$ for all $j\in\Nm$. First note that $f\paren[a]{p_{(g(k))}}=f\paren[a]{p_{(k)}}$ and $R_{g(k)}^\dag=R_k^\dag$ for all $k\in\Nm$. Second, note that if $p_i\leq p_{(k)}$ and $i\not\in R_{g(k)}$, then $i=\sigma(\ell)$ with $\ell>g(k)$, so $p_i=p_{\sigma(\ell)}=p_{(\ell)}>p_{(g(k))}=p_{(k)}$ by definition of $g(k)$, which is a contradiction, and so $R_{g(k)}=R_k^\dag$. Let $A\in \Afam(\Rfam^\sim)$. For all $k\in\Nm$,
\begin{align*}
    \abs[a]{A\cap R_k^\dag}&=\abs[a]{A\cap R_{g(k)}}\\
    &\leq f\paren[a]{p_{(g(k))}}\\
    &= f\paren[a]{p_{(k)}},
\end{align*}
so $A\in \Afam(\Rfam^\dag)$ and $\Afam(\Rfam^\sim)\subseteq \Afam(\Rfam^\dag)$. 

To conclude, $\Afam(\Rfam^\sim)= \Afam(\Rfam^\dag)$, and so $V^*_{\Rfam^\dag}(S)= V^*_{\Rfam^\sim}(S)$ for all $S\subseteq\Nm$. Hence \eqref{eq_ties_2} holds.
\end{proof}

\subsection{Links between the top-\texorpdfstring{$k$}{k} and \texorpdfstring{$k$}{k}-FWER settings}\label{sub_appendix_topk_kFWER}

\subsubsection{From \texorpdfstring{$k$}{k}-FWER to top-\texorpdfstring{$k$}{k}}\label{sec_kFWER_to_topk}

In this Section, we show that from every reference family in the $k$-FWER setting, a top-$k$ family can be built that leads to the same confidence envelope. 

Let $\Rfam=\paren{R_k, k-1}_{1\leq k\leq K}$ a reference family in the $k$-FWER setting, which means that for all $k\leq K$, $R_k=\set{i:p_i\leq t_k}$ for some nondecreasing sequence $(t_k)_{1\leq k\leq K}$. 

We introduce the following notation: for two integers $a\leq b$, $\llparenthesis a,b\rrbracket=\llbracket a,b\rrbracket\setminus\set{a}$. In particular, note that $\llparenthesis a,b\rrbracket=\varnothing$ if $a=b$.

The construction of the top-$k$ family is the following. Let $\sigma$ a (random) permutation that orders the $p$-values, and $R_\ell^\sigma=\set{\sigma(1),\dotsc,\sigma(\ell)}$. We let $\ell_0=0$, $\ell_{K+1}=m$ and, for any $k\in\llbracket1,K\rrbracket$, 
\begin{equation*}
    \ell_k=\max\set{\tilde \ell\in\Nm: p_{(\tilde\ell)}\leq t_k}.
\end{equation*}
For any $k\in\llbracket1,K\rrbracket$ and any $\ell\in \llparenthesis \ell_{k-1},\ell_k\rrbracket$, we let $\zeta^\sigma_\ell=k-1$. For any $\ell\in \llparenthesis \ell_{K},\ell_{K+1}\rrbracket$, we let $\zeta^\sigma_\ell=m$. Finally, we let $\Rfam^\sigma=\paren{R_\ell^\sigma,\zeta_\ell^\sigma}_{1\leq \ell\leq m}$, which is indeed built according to the top-$k$ setting.

An example of such construction, that highlights some formal difficulties, is given below (\autoref{ex_kFWER_to_topk}).
\propl{
For all $S\subseteq\Nm$, $\hat V^{\JER}_\Rfam(S)=V^*_{\Rfam^\sigma}(S)$.
}{prop_kFWER_to_topk}
Like \autoref{prop_ties}, this Proposition is purely deterministic. The proof uses arguments similar to the proof of \autoref{prop_ties}, that is, the use of Equations \eqref{eq_a} and \eqref{eq_vstar}.
\begin{proof}
First note that, for all $k\in\llbracket1,K\rrbracket$, $R_k=R^\sigma_{\ell_k}$. Indeed, if $p_i\leq t_k$, then $i=\sigma(\ell)$ with $\ell\leq \ell_k$ by definition of $\ell_k$ as a maximum, and conversely, for all $\ell\leq \ell_k$, $p_{\sigma(\ell)}\leq p_{\sigma(\ell_k)}\leq t_k$.

Second, note that for all $k\in\llbracket1,K\rrbracket$, $\zeta^\sigma_{\ell_k}\leq k-1$. Indeed, let $k\in\llbracket1,K\rrbracket$. Given that
\[
\bigcup_{k'=1}^k\llparenthesis \ell_{k'-1}, \ell_{k'}\rrbracket = \llbracket 1, \ell_k\rrbracket,
\]
there exists $k'$, $1\leq k'\leq k$, such that $\ell_k\in \llparenthesis \ell_{k'-1}, \ell_{k'}\rrbracket$. Then, by definition, $\zeta^\sigma_{\ell_k}=k'-1\leq k-1$.

Now, let $A\in \Afam\paren{\Rfam}$. Let any $\ell\in\Nm$, there exists a unique $k\in\llbracket1,K+1\rrbracket$ such that $\ell\in\llparenthesis \ell_{k-1}, \ell_{k}\rrbracket$. If $k=K+1$, then $\zeta^\sigma_\ell=m$ and so obviously $\abs{A\cap R^\sigma_\ell}\leq m = \zeta^\sigma_\ell$. If $k\leq K$, then $\zeta^\sigma_\ell=k-1$. By the first remark, $R_k=R^\sigma_{\ell_k}$, and given that $R^\sigma_\ell\subseteq R^\sigma_{\ell_k}$, we have:
\begin{align*}
    \abs{A\cap R^\sigma_\ell}&\leq \abs{A\cap R^\sigma_{\ell_k}}\\
    &=\abs{A\cap R_{k}}\\
    &\leq k-1 \text{ because }A\in \Afam\paren{\Rfam}\\
    &=\zeta^\sigma_\ell.
\end{align*}
This proves that $\Afam\paren{\Rfam}\subseteq\Afam\paren{\Rfam^\sigma}$.

Let $A\in \Afam\paren{\Rfam^\sigma}$. Let any $k\in\llbracket 1, K\rrbracket$. We have:
\begin{align*}
    \abs{A\cap R_k}&= \abs{A\cap R^\sigma_{\ell_k}} \text{ by the first remark}\\
    &\leq \zeta^\sigma_{\ell_k} \text{ because }A\in \Afam\paren{\Rfam^\sigma}\\
    &\leq k-1 \text{ by the second remark}.
\end{align*}
This proves that $\Afam\paren{\Rfam^\sigma}\subseteq\Afam\paren{\Rfam}$.

In the end, $\Afam\paren{\Rfam}=\Afam\paren{\Rfam^\sigma}$ and so $\hat V^{\JER}_\Rfam(\cdot)=V^*_{\Rfam^\sigma}(\cdot)$.
\end{proof}
\rem{Given that each $R_k$ is a particular $R^\sigma_{\ell_k}$, the previous construction just comes back to adding the missing $R^\sigma_\ell$'s to the family and to associate them a conservative upper bound for $V\paren[a]{R^\sigma_\ell}$.}
\exl{
Assume that $m=4$ and that the threshold family is the Simes family as described in \autoref{thm_simes_kFWER}: $t_1=\frac{\alpha}{4}, t_2=\frac{2\alpha}{4}, t_3=\frac{3\alpha}{4}, t_4={\alpha}{}$. Assume that $p_1=\frac{\alpha}{4000}, p_2=\frac{1.5\alpha}{4}, p_3=\frac{1.9\alpha}{4}, p_4=\frac{3.5\alpha}{4}$. Then $\ell_1=1$, $\ell_2=\ell_3=3$, and $\ell_4=4$. Notably, in the proof of \autoref{prop_kFWER_to_topk}, the $k'$ such that $\ell_3\in \llparenthesis \ell_{k'-1}, \ell_{k'}\rrbracket$ is $k'=2$, because $\ell_2=\ell_3$ and so $\llparenthesis \ell_{2}, \ell_{3}\rrbracket=\varnothing$. 
}{ex_kFWER_to_topk}

\subsubsection{From top-\texorpdfstring{$k$}{k} to \texorpdfstring{$k$}{k}-FWER}\label{sec_topk_to_kFWER}

Conversely, in this Section, we show that from every reference family in the top-$k$ setting, a $k$-FWER family can be built that leads to the same confidence envelope. 

Let $\Rfam=\paren{R_k, \zeta_k}_{1\leq k\leq K}$ a reference family in the top-$k$ setting, which means that $R_k=\set{\sigma(1),\dotsc,\sigma(k)}$ for some $\sigma$ a (random) permutation that orders the $p$-values. Recall \autoref{rem_topk_nondecreasing_2}: we assume that $\zeta_k$ is of the form $f\paren[a]{p_{(k)}}\wedge k$ for some measurable, nondecreasing function $f$ that only takes integer values, which is a stronger assumption than in \autoref{prop_ties}.

The construction of the $k$-FWER family is the following. By convention, we let $p_{(0)}=-1$ and $f(-1)=0$. We then let, for each $s\in\Nm$,
\begin{equation*}
    \nu(s)=\max\set[a]{k\in\llbracket0,m\rrbracket: f\paren[a]{p_{(k)}}\leq s-1}.
\end{equation*}
For any $s\in\Nm$, we then let $t_s=p_{\paren[a]{\nu(s)}}$ and $R^\nu_s=\set[a]{i: p_i\leq t_s}$. Note that $\nu(\cdot)$ is nondecreasing because $f\paren[a]{p_{(\nu(s))}}\leq s-1\leq (s+1) - 1$, and so the sequence $\paren{t_s}_{s\in\Nm}$ is nondecreasing. Finally, we let $\Rfam^\nu=\paren{R_s^\nu,s-1}_{1\leq s\leq m}$, which is indeed built according to the $k$-FWER setting.
\propl{Let $\zeta_k=f\paren[a]{p_{(k)}}\wedge k$ for all $k\in\Nm$, where $f$ is any measurable, nondecreasing function with integer values.

Then, for all $S\subseteq\Nm$, $\hat V^{\JER}_\Rfam(S)=V^*_{\Rfam^\nu}(S)$.
}{prop_topk_to_kFWER}
Like \autoref{prop_ties}, this Proposition is purely deterministic. The proof uses arguments similar to the proof of \autoref{prop_ties}, that is, the use of Equations \eqref{eq_a} and \eqref{eq_vstar}.
\begin{proof}
First note that for any $s\in\Nm$, if $\nu(s)\neq0$, $R^\nu_s=R_{\nu(s)}$. Indeed, if $i\in R^\nu_s$, then $p_i\leq p_{\paren[a]{\nu(s)}}$ and $i=\sigma(\ell)$ for some $\ell\in\Nm$. Hence $p_{\paren[a]{\ell}}\leq p_{\paren[a]{\nu(s)}}$, $f\paren[a]{p_{\paren[a]{\ell}}}\leq f\paren[a]{p_{\paren[a]{\nu(s)}}}\leq s-1$ because $f$ is nondecreasing, and $\ell\leq\nu(s)$ by definition of $\nu(s)$, hence $i\in R_{\nu(s)}$. Conversely, if $i\in R_{\nu(s)}$, this means that $i=\sigma(\ell)$ for some $\ell\leq \nu(s)$, and so $p_i=p_{\paren[a]{\ell}}\leq p_{\paren[a]{\nu(s)}}$ and $i\in R^\nu_s$.

Let $A\in\Afam(\Rfam)$, and $s\in\Nm$. If $\nu(s)=0$, $\abs[a]{A\cap R^\nu_s}=\abs[a]{A\cap \varnothing}=0\leq s-1$. If $\nu(s)\neq0$, by the previous remark,
\begin{align*}
    \abs[a]{A\cap R^\nu_s}&=\abs[a]{A\cap R_{\nu(s)}}\\
    &\leq \zeta_{\nu(s)} \text{ because }A\in\Afam(\Rfam)\\
    &\leq f\paren[a]{p_{\paren[a]{\nu(s)}}}\\
    &\leq s-1,
\end{align*}
so, in any case, $A\in\Afam(\Rfam^\nu)$.

Let $A\in\Afam(\Rfam^\nu)$, and let $k\in\Nm$. Note that $\abs[a]{A\cap R_k}\leq \abs[a]{R_k}=k$, so if $f\paren[a]{p_{(k)}}\geq k$, $\zeta_k=k$ and $\abs[a]{A\cap R_k}\leq \zeta_k$. Now assume that $f\paren[a]{p_{(k)}}\leq k-1$, which implies that $k\leq\nu(k)$. Let
\[
s^*=\min\set{s\in\Nm: k\leq\nu(s)},
\]
which is well-defined, because the set $\set{s\in\Nm: k\leq\nu(s)}$ is non-empty because it contains $k$, as we just emphasized. We distinguish cases, between $s^*=1$ and $s^*>1$. If $s^*=1$, $k\leq\nu(1)$ and so $0\leq f\paren[a]{p_{(k)}}\leq f\paren[a]{p_{(\nu(1))}}\leq0$ and so $f\paren[a]{p_{(k)}}=0$. If $s^*>1$, we have $\nu\paren[a]{s^*-1}< k\leq  \nu\paren[a]{s^*}$ and so $s^*-2<f\paren[a]{p_{(k)}}\leq s^*-1$. Given that $f$ only takes integer values, in both cases, $f\paren[a]{p_{(k)}}= s^*-1$. Also note that $k\leq  \nu\paren[a]{s^*}$ implies that both $R_k\subseteq R_{\nu\paren[a]{s^*}}$ and $\nu\paren[a]{s^*}\neq0$, so $R_{\nu\paren[a]{s^*}}=R^\nu_{s^*}$ by the previous remark. Finally,
\begin{align*}
    \abs[a]{A\cap R_k}&\leq \abs[a]{A\cap R^\nu_{s^*}}\\
    &\leq s^*-1 \text{ because }A\in\Afam(\Rfam^\nu)\\
    &=f\paren[a]{p_{(k)}},
\end{align*}
so $\abs[a]{A\cap R_k}\leq\zeta_k$ and $A\in\Afam(\Rfam)$.

In the end, $\Afam\paren{\Rfam}=\Afam\paren{\Rfam^\nu}$ and so $\hat V^{\JER}_\Rfam(\cdot)=V^*_{\Rfam^\nu}(\cdot)$.
\end{proof}
\rem{The previous construction just comes back to identifying where the sequence $\paren[a]{f\paren[a]{p_{(k)}}}_{k\in\Nm}$ makes its jumps.}

\subsubsection{The two Simes envelopes are the same}\label{sec_simes_equal_simes}

In this Section, we prove that the two reference families built using the Simes family that we encountered in \secref{sec_homo_JER} yield the exact same confidence envelope, as a corollary of \secref{sec_kFWER_to_topk}. 

Namely, let $\Rfam$ the reference family in the $k$-FWER setting described in \autoref{thm_simes_kFWER}: $\Rfam=\paren[a]{\set[a]{i:p_i\leq\frac{\alpha k}m}, k-1}_{k\in\Nm}$. Let $\Rfam^\top=\paren[a]{R^\sigma_\ell, \zeta_\ell}_{\ell\in\Nm}$ the reference family in the top-$k$ setting described in \autoref{thm_families_topk_homogeneous}, item \ref{item_topk_simes}, which means that $R_\ell^\sigma=\set{\sigma(1),\dotsc,\sigma(\ell)}$ for $\sigma$ a (random) permutation that orders the $p$-values, and $\zeta_\ell=\indev[a]{p_{(\ell)}>0}\floor[a]{\frac{mp_{(\ell)}}{\alpha}^-}\wedge \ell$.

\corl{
For all $S\subseteq\Nm$, $\hat V^{\JER}_\Rfam(S)=V^*_{\Rfam^\top}(S)$.
}{cor_simes_equal_simes}
\begin{proof}
First note that, just like we proved \eqref{eq_ties_1}, $V^*_{\Rfam^\top}(\cdot)=V^*_{\Rfam^{\ddag}}(\cdot)$, where 
\[
\Rfam^{\ddag}=\paren[a]{R^\sigma_\ell, \indev[a]{p_{(\ell)}>0}\floor[a]{\frac{mp_{(\ell)}}{\alpha}^-}\wedge m}_{\ell\in\Nm},
\]
that is, we just replaced the ``$\wedge\ell$'' by ``$\wedge m$''. The remainder of the proof consists in showing that $\Rfam^{\ddag}$ is exactly the family $\Rfam^\sigma$ constructed in \secref{sec_kFWER_to_topk} from $\Rfam$ (hence $K=m$ here), that is, that $\zeta^\sigma_\ell=\indev[a]{p_{(\ell)}>0}\floor[a]{\frac{mp_{(\ell)}}{\alpha}^-}\wedge m$ for all $\ell\in\Nm$. The result will then ensue by applying \autoref{prop_kFWER_to_topk}.

Let $\ell\in\Nm$. Let $k\in\llbracket1,m+1\rrbracket$ the unique integer such that $\ell\in \llparenthesis \ell_{k-1},\ell_k\rrbracket$. We distinguish the three cases $p_{(\ell)}=0$, $k\leq m$ with $p_{(\ell)}>0$, and $k=m+1$.

First, if $p_{(\ell)}=0$, $\indev[a]{p_{(\ell)}>0}\floor[a]{\frac{mp_{(\ell)}}{\alpha}^-}\wedge m=0$. Given that $p_{(\ell)}=0\leq \frac{\alpha \times 1}m$, $\ell\leq \ell_1$, $k=1$, and so $\zeta^\sigma_\ell=1-1=0$ by definition. So, indeed, $\zeta^\sigma_\ell=\indev[a]{p_{(\ell)}>0}\floor[a]{\frac{mp_{(\ell)}}{\alpha}^-}\wedge m$.

Second, assume that $k\leq m$ and $p_{(\ell)}>0$. Then $\zeta^\sigma_\ell=k-1$ by definition. Given that $\ell_{k-1}<\ell\leq\ell_k$, by definition of $\ell_k$, we have $p_{(\ell)}\leq p_{(\ell_k)}\leq \frac{\alpha k}m$. If $k\geq 2$, we also have, by definition of $\ell_{k-1}$, that $p_{(\ell)}>\frac{\alpha (k-1)}m$, and that is also true if $k=1$ because we assumed $p_{(\ell)}>0$. Then $k\geq \frac{mp_{(\ell)}}{\alpha}$ and $k-1<\frac{mp_{(\ell)}}{\alpha}$, which means that $k=\ceil[a]{\frac{mp_{(\ell)}}{\alpha}}$ and, by \autoref{rem_right_floor}, $\zeta^\sigma_\ell=k-1=\floor[a]{\frac{mp_{(\ell)}}{\alpha}^-}=\indev[a]{p_{(\ell)}>0}\floor[a]{\frac{mp_{(\ell)}}{\alpha}^-}$. Furthermore, given that $k\leq m$, $\zeta^\sigma_\ell=k-1< m$ and so, indeed, $\zeta^\sigma_\ell=\indev[a]{p_{(\ell)}>0}\floor[a]{\frac{mp_{(\ell)}}{\alpha}^-}\wedge m$.

Third, if $k= m+1$, then $\zeta^\sigma_\ell=m$ by definition. Furthermore, by definition of $\ell_m$, $p_{(\ell)}>\frac{\alpha m}{m}=\alpha$. In particular, $\indev[a]{p_{(\ell)}>0}=1$. We also have that $\frac{mp_{(\ell)}}{\alpha}>m$ and so $\floor[a]{\frac{mp_{(\ell)}}{\alpha}^-}\geq m$. So, indeed, $\zeta^\sigma_\ell=\indev[a]{p_{(\ell)}>0}\floor[a]{\frac{mp_{(\ell)}}{\alpha}^-}\wedge m$.
\end{proof}

The following is a counterexample showing that the equality can fail if we forget to use the left-limit of the floor function instead of the floor function in the top-$k$ Simes envelope.
\exl{
Let $m=2$, $p_1=\frac\alpha2$, $p_2=\alpha$, which can happen with large probability under the alternative because we do not assume that the distributions of the $p$-values are continuous. Then $\zeta_1=0$ and $\zeta_2=1$ and, for example, $V^*_{\Rfam^\top}(\set{1})=0$. But, without the left-limit, we would get $\zeta_1=1$, $\zeta_2=2$ and $V^*_{\Rfam^\top}(\set{1})=1$.
}{counterex_simes_equal_simes}

\section{Uniformization of super-uniform \texorpdfstring{$p$}{p}-values}\label{appendix_random_pvalues}

We present the following \autoref{lem_unif}, that explains how to transform any random variable into a uniform variable, which is, additionally in the case of super-uniformity, smaller than the first one.
This construction encompasses the classical ``randomized $p$-value'' construction, as stated by \autoref{rem_random_pvalues}. 

Note that this Lemma also offers an alternative way to extend the validity domain of various concentration inequalities from uniform variables to super-uniform variables (see Propositions \ref{DKWin} and \ref{Wellnerin}), by providing a different coupling construction than in \autoref{lem_unifcoupling}.

\leml[Uniformization of any random variable]{
Let $Y$ any real-valued random variable, let $F$ its cdf and $F^-$ the left-limit function of $F$. Assume that the probability space where $Y$ is defined is rich enough so that there exists $U$ another random variable independent from $Y$, with uniform distribution: $U\sim \cU([0,1])$. Define $\widetilde Y$ as
\[ 
\widetilde Y = F^-(Y) + U(F(Y) - F^-(Y) ).
\]
Then $\widetilde Y\sim \cU([0,1])$.

Furthermore, if $Y$ is super-uniform: $Y \succeq \cU([0,1])$, then $\widetilde Y \leq Y$ almost surely.
}{lem_unif}
\begin{proof}
We prove that the cdf $\widetilde F$ of $\widetilde Y$ is $x\mapsto 0 \vee (x \wedge 1)$.

First note that $\widetilde Y \in [0,1]$ a.s., so $\widetilde F(x)=1$ for all $x\in [1,\infty)$ and $\widetilde F(x)=0$ for all $x\in (-\infty, 0)$. So let $x\in [0,1)$ and let us show that $\widetilde F(x)=x$. We distinguish two cases: if $x\in F(\RR)$ or not.

In the first case, there exists $\beta\in\RR$ such that $x=F(\beta)$. We remark the following way to write $\widetilde F$:
\begin{align*}
    \widetilde F(x)&=\Pro[a]{F^-(Y) + U(F(Y) - F^-(Y) ) \leq x}\\
    &=\Esp[a]{\indev{F(Y)\leq x} + \indev{F^-(Y)\leq x<F(Y)}\indev{U(F(Y)-F^-(Y))\leq x-F^-(Y)} }
\end{align*}
Now, if $F^-(Y)\leq x<F(Y)$ is true, then $ F(\beta)<F(Y)$. Because $F$ is nondecreasing, necessarily $\beta<Y$, which in turn implies that $F(\beta)\leq F^-(Y)$: in the end, $x=F(\beta)=F^-(Y)$, and $F(Y)>F^-(Y)$, so $\indev{U(F(Y)-F^-(Y))\leq x-F^-(Y)}=\indev{U=0}=0$ a.s. So $\indev{F^-(Y)\leq x<F(Y)}\indev{U(F(Y)-F^-(Y))\leq x-F^-(Y)}=0$ a.s., and 
\begin{align*}
    \widetilde F(x)&=\Esp[a]{\indev{F(Y)\leq x}}\\
    &=\Pro[1]{F(Y)\leq x}\\
    &=x,
\end{align*}
where the last equality comes from Proposition 2 of Chapter 1 of \citet{shorack2009empirical}, thanks to $x$ being in the range of $F$.

Now we turn to the case where $x\not\in F(\RR)$. Let $q_x$ the $x$-level quantile of $F$: $q_x=\inf\set{t\in\RR: F(t)\geq x}$. By the usual properties of quantiles, $F^-(q_x)\leq x$ and $F(q_x)\geq x$, and we even have $F(q_x)> x$ because $x\not\in F(\RR)$. 

We now show that $\widetilde Y\leq x$ if and only if either $Y<q_x$, or $Y=q_x$ and $U(F(q_x)-F^-(q_x))\leq x-F^-(q_x)$. Assume that $\widetilde Y\leq x$. Note that $\widetilde Y\in\brac{F^-(Y),F(Y)}$ and so $F^-(Y)\leq x$. If we had $Y > q_x$, then we would have $F^-(Y)\geq F(q_x)$, and so $x\geq F^-(Y)\geq F(q_x)>x$ which would be a contradiction. So $Y\leq q_x$ and the implication follows. Conversely, assume the other proposition. Obviously, if $Y=q_x$ and $U(F(q_x)-F^-(q_x))\leq x-F^-(q_x)$ then $\widetilde Y\leq x$. If $Y<q_x$, then $\widetilde Y\leq F(Y)\leq F^-(q_x)\leq x$.

Thanks to the equivalence, we have 
\begin{align*}
    \widetilde F(x)&=\Pro[a]{Y<q_x}+\Pro[2]{\set{Y=q_x}\cap \set{ U(F(q_x)-F^-(q_x))\leq x-F^-(q_x) }}\\
    &=F^-(q_x)+\Pro[a]{Y=q_x}\Pro[a]{U(F(q_x)-F^-(q_x))\leq x-F^-(q_x)},
\end{align*}
by independence of $Y$ and $U$. Note that $F^-(q_x)\leq x < F(q_x)$, so $\frac{x-F^-(q_x)}{F(q_x)-F^-(q_x)}$ is well defined and in $[0,1]$. Then,
\begin{align*}
    \widetilde F(x)&=F^-(q_x)+\paren[1]{F(q_x)-F^-(q_x)}\Pro[a]{U\leq \frac{x-F^-(q_x)}{F(q_x)-F^-(q_x)}}\\
    &=F^-(q_x)+\paren[1]{F(q_x)-F^-(q_x)}\frac{x-F^-(q_x)}{F(q_x)-F^-(q_x)}\\
    &=x.
\end{align*}
This concludes that $\widetilde Y\sim \cU([0,1])$.

Now assume that $Y \succeq \cU([0,1])$, which means that $F(x)\leq 0 \vee (x \wedge 1)$ for all $x\in\RR$. In particular, $F(x)\leq x$ for all $x\geq0$, and $\Pro[a]{Y\leq 0}=0$, so $Y\geq 0$ a.s. Then
\begin{align*}
    \widetilde Y &= F^-(Y) + U(F(Y) - F^-(Y))\\
    &\leq F(Y) \text{ a.s.}\\
    &\leq Y \text{ a.s.,}
\end{align*}
which concludes.
\end{proof}

\reml{The construction of \autoref{lem_unif} covers the classical definition of ``random $p$-values''. Indeed, assume that we have at hand, for each $i\in\Nm$, a test statistic $T_i$ such that for all $\mu\in H_i, X\sim\mu$, the distribution of $T_i(X)$ does not depend on $\mu$ and is known. The classical unilateral $p$-value derived from $T_i$ is either $p_i^{(1)}(X)=\Proo[1]{T_i(X_0)\geq T_i(X)}{X_0\sim\mu_i}$, either $p_i^{(2)}(X)=\Proo[1]{T_i(X_0)\leq T_i(X)}{X_0\sim\mu_i}$, where the notation $\mathbb{P}_{X_0\sim\mu_i}$ denotes that the integration takes place with respect to a random variable $X_0$ drawn according to some $\mu_i\in H_i$. Then the corresponding classical random $p$-value is either $\tilde p_i^{(1)}(X)=\Proo[1]{T_i(X_0)> T_i(X)}{X_0\sim\mu_i}+U \Proo[1]{T_i(X_0)= T_i(X)}{X_0\sim\mu_i}$, either $\tilde p_i^{(2)}(X)=\Proo[1]{T_i(X_0)< T_i(X)}{X_0\sim\mu_i}+U \Proo[1]{T_i(X_0)= T_i(X)}{X_0\sim\mu_i}$, where $U\sim\cU([0,1])$ is independent of $X$. The key observation is that $\tilde p_i^{(1)}(X)$ and $\tilde p_i^{(2)}(X)$ follow exactly from the construction of \autoref{lem_unif}, applied respectively to the random variables $-T_i(X)$ and $T_i(X)$, and their cdf if $X\sim\mu_i$. 
}{rem_random_pvalues}

\section{Complements on \texorpdfstring{\citet{Katsevichramdas2020}}{Katsevich and Ramdas (2020a)} and \texorpdfstring{\citet{MBR2024}}{Meah et al. (2024)} confidence envelopes}\label{sec:KRcomplements}

We present in this section a slight refinement of the adaptive top-$k$ method from \citet{MBR2024} obtained by application of the 
formalism introduced in \secref{sec_general_shortcuts}.
Let us consider the confidence bound for the ecdf introduced by \citet{Katsevichramdas2020}; let $\delta \in (0,0.31)$, then, for $U_1,\dots,U_n \sim \mathcal{U}\paren[1]{[0,1]}$,
\[\Pro{\frac{1}{n}\sum^n_{i=1}\indev{U_i\leq t} \geq C_\delta\paren{1/n+t}}\leq \delta,\]
where $C_\delta = \frac{\log\paren{1/\delta}}{\log\paren{1+\log\paren{1/\delta}}}$.\\
Hence we derive for all $A\subseteq \Nm$ and $\alpha \in (0,0.31)$ a local test of level $\alpha$ : \[\varphi_A = \sup_{t\in [0,1]}\indev{i_A(t) \geq f_{|A|}(t)},\]
where for all $t\in [0,1]$ and $n\in \Nm$, $f_n(t) = C_\delta\paren[a]{1+tn}$. \\
For all $t\in [0,1]$ and $n\in \Nm$, $\varphi_{\Nm,n,t} = \indev{f_n(t)\leq n+i(t)-m}$. Therefore: 
\begin{itemize}
\item If $t\geq \frac{1}{C_\delta}$, $\varphi_{\Nm,n,t} = 0$,
\item If  $t< \frac{1}{C_\delta}$ and $\varphi_{\Nm,n,t} = 1$, $f_{n+1}(t) = C_\delta\paren[1]{1+\paren{n+1}t} = f_{n}(t) +C_\delta t \leq  n+1 +i(t)-m$. Hence $\varphi_{\Nm,n+1,t} = 1$.
\end{itemize}
Thus,  $\hat{V}^{\SCo}_{\varphi}\paren[1]{ \Nm} = \hat{V}^{\SCt}_{\varphi}\paren[1]{ \Nm} $ because $\varphi$ verifies the assumptions from \autoref{SC1eqSC2} for $S = \Nm$.\\
We now compute $b(t) = \max\set{n\in \NM: \varphi_{\Nm,n,t} = 0}$: \begin{itemize}
\item If $t \geq \frac{1}{C_\delta}$, the first point above implies $b(t)= m$.
\item If $t < \frac{1}{C_\delta}$, then: \begin{align*}
b(t) &= \max\set{n\in \NM: C_\delta\paren[a]{1+tn}> n+i(t)-m} \\
&= \max\set{n\in \NM: n > \frac{m-i(t) + C_\delta}{1 - tC_\delta}}\\
&= m\wedge \paren{\ceil{\frac{m-i(t) + C_\delta}{1 - tC_\delta}}-1}.
\end{align*} 
\end{itemize}
Hence, $\hat{V}^{\SCo}_{\varphi}\paren[1]{ \Nm} = m \wedge \paren{\min\limits_{t \in [0,1/C_\delta)}\ceil{\frac{m-i(t) + C_\delta}{1 - tC_\delta}}-1} $. \\
From \autoref{SCfuncversion}, we conclude that $\hat{V}^{\IP}_{\varphi}(S) = \min\limits_{0\leq k\leq |S|}\paren[2]{\ceil{f_{\hat{m}_0}\paren{p_{(k : S)}}} -1 + \abs{S\setminus R\paren{p_{(k:S)}}}}$ where $\hat{m}_0 = m \wedge \paren{ \min\limits_{t \in [0,1/C_\delta)}\ceil{\frac{m-i(t) + C_\delta}{1 - tC_\delta}}-1}$\footnote{The careful reader may remark that the local test from this proposition are expressed with a ``$>$'' instead of a ``$\geq$'' and that we replaced ``$\floor{\cdot}$'' by ``$\ceil{\cdot}-1$'' in the expression of $\hat{V}^{\IP}_{\varphi}(S)$. The proof can easily be adjusted, since the $f_n$ are bijective in the specific case of KR inequality. }.\\
\citet{MBR2024} already proposed an over-estimator on $m_0$, that we denote $\tilde{m}_0$, similar to the one we expose here. Precisely, the following holds: \[\hat{m}_0 = \ceil[1]{\tilde{m}_0\wedge\paren{m-i(0)+C_\delta}-1}.\]
It could be legitimate to question the usefulness of such refinements as the one described above or the one detailed by \autoref{SimesVSC1equalsVSC2} on vanilla Simes confidence bounds. However, Figure~20 from \citet{MBR2024} shows that it reduces conservativeness, as the closed testing envelopes --- that are equal to the inversion procedure envelopes --- appear to be less conservative. We have in turn run these simulation while integrating our refinements; the adapted interpolated envelopes and the close-testing envelopes happened to be the same, as expected from our theoretical developments.

\section{Proofs} \label{sec_proof}
\subsection{Proof of \autoref{lem_unifcoupling}}\label{proof_lem_unifcoupling}

Let $Q_1, \dotsc, Q_N$ the respective quantile functions of $U_1, \dotsc, U_N$. Let $\widetilde U_1, \dotsc, \widetilde U_N$ be $N$ independent random variables that are $\cU[0,1]$ (possibly defined on another probability space). By classical quantile properties, we know that $V_i=Q_i\paren[1]{\widetilde U_i}$ has the same distribution as $U_i$ for any $i\in\llbracket 1, N\rrbracket$, and so, by independence, $\paren[a]{V_i}_{i\in\llbracket 1, N\rrbracket}$ has the same distribution as $\paren{U_i}_{i\in\llbracket 1, N\rrbracket}$. Furthermore, by super-uniformity, $V_i = Q\paren[2]{\widetilde{U}_i} \geq \widetilde U_i$ almost surely for all $i\in\llbracket 1, N\rrbracket$.

\subsection{Proofs of \autoref{DKWin} and \autoref{Wellnerin} }\label{proof_DKWin}

We first prove \autoref{DKWin}. We know that the statement is true if the random variables are uniform, by \citet{reeve2024short}. Using the coupled variables $(\widetilde{U}_i,V_i)_{i \in \intrN}$ as constructed in \autoref{lem_unifcoupling}, it holds
\begin{align*}
    \Pro[4]{\sqrt{N}\sup\limits_{t\in [0,1]}\paren{\frac1N\sum_{i=1}^N\indev[a]{U_i\leq t} - t}> \lambda}&= \Pro[4]{\sqrt{N}\sup\limits_{t\in [0,1]}\paren{\frac1N\sum_{i=1}^N\indev[1]{V_i\leq t}- t}> \lambda}\\
    &\leq \Pro[4]{\sqrt{N}\sup\limits_{t\in [0,1]}\paren{\frac1N\sum_{i=1}^N\indev[a]{\widetilde U_i\leq t} - t}> \lambda}\\
    &\leq \exp\paren[a]{-2\lambda^2}.
\end{align*}

To prove \autoref{Wellnerin}, we know that the statement is true if $U_1, \dotsc, U_N$ are uniform, by \citet{MBR2024}. We then use the same trick as in the proof of \autoref{DKWin}.

\subsection{Proof of \autoref{Simesin_var}}\label{proof_Simesin_var}

We consider that the event $\set{\forall k\in \NN : U_{(k)}> \frac{k\delta}{N}}$ is realized, which happens with probability at least $1-\delta$ by \autoref{Simesin}. For any $t>\delta$, $\frac1N\sum_{i=1}^N\indev{ U_i\leq t}\leq 1 <\frac{t}{\delta}$. Now let $j\in\set{0,\dotsc, N-1}$ and $t\in\left(\frac{j\delta}{N},\frac{(j+1)\delta}{N}  \right]$. Then $\sum_{i=1}^N\indev{ U_i\leq t}$ is equal to $\mathrm{card}\set{i: U_i\leq t}$, and given that $U_{(j+1)}> \frac{(j+1)\delta}{N}\geq t$, $\mathrm{card}\set{i: U_i\leq t}\leq j < \frac{t N}{\delta}$ and finally $\frac1N\sum_{i=1}^N\indev{ U_i\leq t}<\frac{t}{\delta}$.

\subsection{Proof of \autoref{thm_simes_kFWER}}\label{proof_thm_simes_kFWER}
We recall here the short proof of \autoref{thm_simes_kFWER}:\\
Let $\mu\in\statfam$, $\alpha\in(0,1)$, and $X\sim\mu$.
\begin{align*}
    1-\JER(\Rfam)&=\Pro[1]{\exists k\leq m, |R_k\cap \cH_0|\geq k}\\
    &=\Pro[a]{\exists k\leq  m_0, p_{\paren[a]{k:\cH_0}}\leq  \frac{\alpha k}{m} }\\
    &\leq\Pro[a]{\exists k\leq m_0, p_{\paren[a]{k:\cH_0}}\leq  \frac{\alpha k}{m_0} }\\
    &\leq\alpha,
\end{align*}
where the last inequality is obtained by applying \autoref{Simesin} to the random variables $(p_i)_{i\in\cH_0}$.

\subsection{Proof of \autoref{prop_concen_topk_hmogeneous}}\label{proof_prop_concen_topk_hmogeneous}

We fix a $\delta \in (0,1]$, and a $\mu\in\statfam$. The proof is exactly the same as in \citet{MBR2024}: we apply Propositions \ref{DKWin}, \ref{Wellnerin} and \ref{Simesin_var} to $N=m_0(\mu)$ and $\set{U_1,\dotsc,U_N}=\set{p_i, i\in \cH_0(\mu)}$. We then use monotonicity properties and the fact that $m_0(\mu)\leq m$ to replace $m_0(\mu)$ by  $m$ in the obtained inequalities. Note that for Wellner's inequality, the desired monotonicity property is the last statement of Lemma 43 of \citet{MBR2024}.

\subsection{Proof of \autoref{prop_algo_topk_path}}\label{proof_prop_algo_topk_path}

First note that, for any $1\leq k\leq m$, by nestedness of the top-$k$ reference family:
\begin{align*}
    \hat V^{\JER}_\Rfam(R_k)&=\min_{1\leq i\leq m}\paren[a]{\zeta_i+\abs[a]{R_k\setminus R_i}}\wedge k, \text{ by \eqref{eq_vbar}};\\
    &=\min_{1\leq i\leq m}\paren[a]{\zeta_i+0\vee (k-i)}\wedge k, \text{ using \eqref{eq_def_Rk_top_k}}.
\end{align*}
By assumption on $f$, the sequence $\paren[a]{\zeta_k}_{1\leq k\leq m}$ is nondecreasing, so for all $i\geq k$,
\begin{align*}
    \zeta_i+0\vee (k-i)&=\zeta_i 
\geq \zeta_k 
= \zeta_k+0\vee (k-k),
\end{align*}
and so
\begin{equation}\label{eq_vstar_algo_1}
    \hat V^{\JER}_\Rfam(R_k)=\min_{1\leq i\leq k}\paren[a]{\zeta_i+ k-i}\wedge k.
\end{equation}
Second, we show that, for all $0\leq k\leq m$, 
\begin{equation}\label{eq_vstar_algo_2}
    \hat V^{\JER}_\Rfam(R_{k+1})=\min\paren[a]{\hat V^{\JER}_\Rfam(R_k)+1,\zeta_{k+1}},
\end{equation}
where $R_0=\varnothing,\hat V^{\JER}_\Rfam(R_0)=0$. Indeed, if $k=0$, 
\[\min\paren[a]{\hat V^{\JER}_\Rfam(R_0)+1,\zeta_{0+1}}=\min(0+1,\zeta_1),\] 
which coincides with \eqref{eq_vstar_algo_1} for $k=1$.
And if $k\geq1$, we have
\begin{align*}
    \hat V^{\JER}_\Rfam(R_{k+1})&=\min_{1\leq i\leq k+1}\paren[a]{\zeta_i+k+1-i}\wedge (k+1), \text{ by \eqref{eq_vstar_algo_1}},\\
    &=\paren[a]{\min_{1\leq i\leq k}\paren[a]{\zeta_i+k+1-i}}\wedge(k+1)\wedge\zeta_{k+1}\\
    &=\paren[a]{\paren[a]{\min_{1\leq i\leq k}\paren[a]{\zeta_i+ k-i}\wedge k}+1}\wedge\zeta_{k+1}\\
    &=\min\paren[a]{\hat V^{\JER}_\Rfam(R_k)+1,\zeta_{k+1}},\text{ again by \eqref{eq_vstar_algo_1}},
\end{align*}
which concludes the proof that \eqref{eq_vstar_algo_2} holds. 

With that, we can now prove that $\tilde\zeta_k=\hat V^{\JER}_\Rfam(R_k)$ for all $0\leq k\leq m$ by an induction scheme going over the integers visited by the \texttt{while} loop of \algoref{algo_topk_path}. First
define formally $\tilde\zeta_0 = \hat V^{\JER}_\Rfam(R_0)=0$. Now let $k$ be an integer visited by the \texttt{while} loop and assume that $\hat V^{\JER}_\Rfam(R_i)=\tilde\zeta_i$ for all $i\leq k-1$. We distinguish the two cases $\zeta_k>\tilde\zeta_{k-1}$ and $\zeta_k\leq \tilde\zeta_{k-1}$ as in the \texttt{if} condition of \algoref{algo_topk_path}.

First assume that $\zeta_k>\tilde\zeta_{k-1}$, let $j = \paren{\zeta_k - \tilde{\zeta}_{k-1}}\wedge \paren{m-k+1}$, so that $\zeta_k\geq\tilde{\zeta}_{k-1}+j$ and the next integer visited by the loop is $k+j$ (note that $j\geq1$). In that case, then, we aim to prove that $\hat V^{\JER}_\Rfam(R_i)=\tilde\zeta_i$ for all $i\leq k+j-1$. By definition, $\tilde{\zeta}_{k+i-1}=\tilde{\zeta}_{k-1}+i$ for all $i\in\llbracket 1,j\rrbracket$. Note that, for each such $i$,
\begin{equation}\label{eq_ineq_zeta_algo}
\zeta_{k+i-1}\geq\zeta_k\geq  \tilde{\zeta}_{k-1}+j\geq\tilde{\zeta}_{k-1}+i=\tilde{\zeta}_{k+i-1}.\end{equation}
To conclude this first case, we show by recursion that $\hat V^{\JER}_\Rfam(R_{k+i-1})=\tilde\zeta_{k+i-1}$ for each $i\in\llbracket 0,j\rrbracket$. For $i=0$, this is already true by the induction assumption. Now, if $\hat V^{\JER}_\Rfam(R_{k+i-1})=\tilde\zeta_{k+i-1}$ for some $i\in\llbracket 0,j-1\rrbracket$, then
\begin{align*}
    \hat V^{\JER}_\Rfam(R_{k+i})&=\paren[a]{\hat V^{\JER}_\Rfam(R_{k+i-1})+1}\wedge\zeta_{k+i},\text{ by \eqref{eq_vstar_algo_2}},\\
    &=(\tilde{\zeta}_{k+i-1}+1)\wedge\zeta_{k+i}\\
    &=(\tilde{\zeta}_{k-1}+i+1)\wedge\zeta_{k+i}\\
    &=\tilde{\zeta}_{k+i}\wedge\zeta_{k+i}=\tilde{\zeta}_{k+i},\text{ by \eqref{eq_ineq_zeta_algo}.}
\end{align*}

In the second case, assume that $\zeta_k\leq\tilde\zeta_{k-1}$, hence $\tilde{\zeta}_{k}=\tilde{\zeta}_{k-1}$ by construction, and the next integer visited by the loop is $k+1$. In that case, then, we aim to prove that $\hat V^{\JER}_\Rfam(R_i)=\tilde\zeta_i$ for all $i\leq k$. First note that
\begin{align*}
    \tilde\zeta_{k-1}&=\hat V^{\JER}_\Rfam(R_{k-1})\\
    &=\max_{A\in\Afam(\Rfam)}\abs[a]{A\cap R_{k-1}}\\
    &\leq \zeta_{k-1}\text{ because }A\in\Afam(\Rfam)\text{ implies }\abs[a]{A\cap R_{k-1}}\leq \zeta_{k-1}\\
    &\leq \zeta_{k},
\end{align*}
so actually $\zeta_k=\tilde\zeta_{k-1}$. To conclude this second case, we just need to show that $\hat V^{\JER}_\Rfam(R_k)=\tilde{\zeta}_{k}$. Indeed,
\begin{align*}
    \hat V^{\JER}_\Rfam(R_k)&=\paren[a]{\hat V^{\JER}_\Rfam(R_{k-1})+1}\wedge\zeta_k, \text{ by \eqref{eq_vstar_algo_2}},\\
    &=(\hat V^{\JER}_\Rfam(R_{k-1})+1)\wedge\tilde{\zeta}_{k-1}\\
    &=(\tilde{\zeta}_{k-1}+1)\wedge\tilde{\zeta}_{k-1},\text{ by the induction assumption},\\
    &=\tilde{\zeta}_{k-1}\\
    &=\tilde{\zeta}_{k}.
\end{align*}

\subsection{Proof of \autoref{goemannloct}}\label{proof_goemannloct}

Suppose that $\max\limits_{i \in \llbracket 1, |A|\rrbracket}\mathds{1}{\set{p_{(i:A)}\leq \ell_{i:|A|}}} = 1$. Then there is some $i$ such that $p_{(i:A)}\leq \ell_{i:|A|}$. Denote $j = i_A\paren{p_{(i:A)}}\geq i$, we have $ p_{(i:A)} \leq \ell_{i:|A|} \leq \ell_{j:|A|}$, thus there is some $t$ such that $\indev{t\leq  \ell _{i_A(t):|A|}}=1$.\\
Suppose now that $\max\limits_{t \in [0,1] } \indev{t\leq  \ell _{i_A(t):|A|}} =1$. Then there is $ t\in [0,1]$ such that $t\leq \ell _{i_A(t):|A|}$, thus $p_{\paren{i_A(t):A}}\leq t\leq \ell _{i_A(t):|A|}$.

\subsection{Proof of \autoref{SChetero}}\label{proof_SChetero}

For $n \in \llbracket 0 , |S|\rrbracket$, denote $\cN_S^n =  \set[1]{A\in \cN_\varphi \text{ : } |A\cap S| = n  } $, where $\cN_\varphi = \set{A \subseteq \Nm : \varphi_A = 0}$. Then, $\hat{V}^{\IP}_{\varphi} (S) = \max\limits_{A\in \cN} |A \cap S|= \max\set{n\in \NS\text{ : }\cN_S^n \not = \varnothing} $.\\
Let $\varphi_{S,n} = \min\limits_{\substack{A \subseteq \Nm\\ |A\cap S| = n}}\varphi_{A}  $ be the test statistic testing whether $\cN_S^n$ is empty or not; indeed, $\varphi_{S,n} = 0$, if and only if there is some $A \subseteq \Nm$ such that $|A\cap S| = n$ and $\varphi_A = 0$.\\
For all $t\in [0,1]$ and $A \subseteq \Nm,  \varphi_{A,t} \leq \varphi_{A}$, thus, take consecutively the minimum over $A \subseteq \Nm$ such that  $|A\cap S| = n$ and the maximum over $t\in [0,1]$ in this inequality to obtain: 
\[\varphi_{S,n} \geq \max\limits_{t \in [0,1]} \varphi_{S,n,t}.\] 
Thus, $\varphi_{S,n}=0$ implies $\max\limits_{t \in [0,1]} \varphi_{S,n,t}  = 0$, thus \[ \set{n\in \NS:\varphi_{S,n} =0} \subseteq \set{n\in \NS:\max\limits_{t \in [0,1]} \varphi_{S,n,t} =0},\] which implies the first inequality.
For all $t\in [0,1]$, by definition, \[\max\limits_{t \in [0,1]} \varphi_{S,n,t}\geq \varphi_{S,n,t},\] thus $ \set{n\in \NS:\max\limits_{t \in [0,1]} \varphi_{S,n,t} =0} \subseteq \set{n\in \NS:\varphi_{S,n,t} =0}$, implying \[\max\set{n\in \NS:\max\limits_{t \in [0,1]} \varphi_{S,n,t} =0} \leq \max \set{n\in \NS:\varphi_{S,n,t} =0}.\] Take the minimum over $t$ in this inequality to get the second inequality.

\subsection{Proof of \autoref{continuoustodiscrete}}\label{proof_continuoustodiscrete}

\begin{itemize}
\item The first equality is straightforward ; if $\varphi_A = 1$ then there is some $t\in [0,1]$ such that $\varphi_{A,t} = 1$. It follows by assumption that $\varphi_{A,p_{\paren{i_A(t):A}}} = 1$.
\item To show the second equality, suppose that $\max\limits_{t \in [0,1]} \varphi_{S,n,t} = 1$. Then there is some $t\in  [0,1]$ such that $\varphi_{S,n,t} = 1$. Thus, for all $A\subseteq \Nm$ such that $|A\cap S| = n$, $\varphi_{A,t} =1$, followed, by assumption, by $\varphi_{A,s} = 1$ for all $s\in \left[p_{\paren{i_A(t):A}},t\right]$. Hence, $\varphi_{A,p_{\paren{i(t)}}}=1$ with $s = p_{\paren{i(t)}}\in \left[p_{\paren{i_A(t):A}},t\right]$ for all $A\subseteq \Nm$.
\item Finally, to show the last equality, we get from the above item that if $n \in \NS$ is such that $ \min\limits_{\substack{A \subseteq \Nm\\ |A\cap S| = n}}\varphi_{A,p_{(i)}} =0$, then for all $t\in \left[p_{(i)},p_{(i+1)}\right)$, $\varphi_{S,n,t} =0$, hence, for all $t\in \left[p_{(i)},p_{(i+1)}\right)$, \[\max\set{n \in \NS : \min\limits_{\substack{A \subseteq \Nm\\ |A\cap S| = n}}\varphi_{A,p_{(i)}} = 0} \leq \max\set{n \in \NS : \varphi_{S,n,t} = 0}.\]
\end{itemize}

\subsection{Proof of \autoref{SC1eqSC2}}\label{proof_SC1eqSC2}

\begin{itemize}
\item $\hat{V}^{\SCo}_{\varphi} (S) \leq \hat{V}^{\SCt}_{\varphi} (S)$  by \autoref{SChetero}.
\item Let $k_0 =  \hat{V}^{\SCt}_{\varphi} (S)$. Then, for all $t \in [0,1]$, 
for all $n \leq k_0$,  $\varphi_{S,n,t}=0$ hence, for all $n \leq k_0$, $\max\limits_{t \in [0,1]} \varphi_{S,n,t}  = 0$. Finally, $\hat{V}^{\SCo}_{\varphi} (S) = \max\set{n\in \NS:\max\limits_{t \in [0,1]} \varphi_{S,n,t} =0} \geq k_0$.
\end{itemize}

\subsection{Proof of \autoref{Simeslikemin}}\label{proof_Simeslikemin}

Let $t\in [0,1]$, $n\in \NS$ and $A \subseteq \Nm$ such that $\abs{A \cap S} = n$. Let $k =i_{A\cap S^c}(t)$. Then $\indev{t \leq \ell_{i_A(t):|A|}} = \indev{t \leq \ell_{\paren{i_{A\cap S}(t)+k}:(n+k+ \abs{A \cap S^c \cap R(t)^c})}} $ is minimal if $S^c \cap R(t)^c \subseteq A$ and $ S \cap R(t)^c \subseteq A $ (if $n$ is large enough), thus \[\min\limits_{\substack{A \subseteq \Nm\\ |A\cap S| = n, i_{A\cap S^c}(t) = k}}\varphi_{A,t}  = \indev{t \leq \ell_{\paren{\paren{n + i_S(t)-|S|}\vee 0+k}:(n+k+ \abs{S^c \cap R(t)^c})}} .\]

\subsection{Proof of \autoref{SimesVSC1equalsVSC2}}\label{proof_SimesVSC1equalsVSC2}

Let $n \in \Nm$ and $i\leq n$. Then \begin{align*}
\ell_{i+1:n+1} - \ell_{i:n} & = \paren{\frac{1}{n+1} - \frac{1}{n}}i\alpha + \frac{1}{n+1}\alpha\\
& = \frac{-i}{n(n+1)}\alpha + \frac{1}{n+1}\alpha\\
& \geq 0.
\end{align*}
We then compute $ b(t) = \max\set{n \in \Nm : \varphi_{\Nm,n,t} = 0}$ for all $t\in (0,1]$. By \autoref{minmaxsimes}, for all $t\in(0,\alpha)$, $ b(t) = \max\set{n \in \Nm : t > \frac{n + i(t)-m}{n}\alpha } = \max\set{n \in \Nm : n < \frac{m-i(t)}{1-t/\alpha} }$, therefore $b(t) = \paren{\left\lceil\frac{m - i(t)}{1- t/\alpha} \right\rceil - 1}\wedge m$. For $t\in [\alpha,1]$ and $n\in \Nm$, $\varphi_{\Nm,n,t} =0$ thus $b(t) = m$, and for $t=0$, $b(0) =m- i(0)$.

\subsection{Proof of \autoref{eqJERIP}}\label{proof_eqJERIP}

We first show the control of the joint error rate: \begin{align*}
\JER(\Rfam) &= \Pro[1]{\exists k\in \llbracket 1, \mzsc + 1\rrbracket : V\paren{R_k} \geq k}\\
& = \Pro[2]{\paren[1]{\exists k\in \llbracket 1, m_0\wedge \mzsc\rrbracket : p_{\paren{k:\cH_0}} \leq \ell_{k:\mzsc}}\cup \paren{ m_0 >\mzsc}}\\
&\leq \Pro[2]{\paren[1]{\exists k\in \llbracket 1, m_0\rrbracket : p_{\paren{k:\cH_0}} \leq \ell_{k:m_0}}\cup \paren{ m_0 >\mzsc}}= \Pro{\varphi_{\cH_0} = 1} \\
 &\leq \alpha.
\end{align*} Then, recall that for all $S\subseteq \Nm$, $\hat{V}^{\JER}_\Rfam(S) = \max\limits_{A \in \cA(\Rfam)} \abs{A\cap S}$, where \[\cA(\Rfam) = \set{A\subseteq \Nm : \forall k \in \llbracket 1, \mzsc+1\rrbracket,  \abs{A\cap R_k}\leq k-1}.\]
\begin{itemize}
\item Let $A \in \cN_\varphi$, then $\varphi_A = 0$ thus for all $i \in \llbracket 1, |A| \rrbracket$, $p_{(i:A)} > \ell_{i:|A|}$ \textit{i.e.} $p_{[i:A]}>\ell_{(|A|-i+1):|A|}$ therefore, by definition of $\mzsc$, $|A|\leq \mzsc$ thus $|A\cap R_{\mzsc +1}|\leq \mzsc$.\\
Let $i \in \llbracket 1, \mzsc\rrbracket$, $|A\cap R_i|\leq i-1$ because $p_{(i:A)} > \ell_{i:|A|} \geq \ell_{i:\mzsc}$ thus $\cN_\varphi \subseteq \cA(\Rfam)$, therefore \[\hat{V}^{\IP}_{\varphi} \leq \hat{V}^{\JER}_\Rfam.\]

\item Let $S \subseteq \Nm$. If  $\hat{V}^{\JER}_\Rfam(S)  = \min\limits_{1\leq u\leq \mzsc}\set{\abs[1]{\set{i\in S : p_i> \ell_{u:\mzsc}}}+u-1} \wedge \mzsc = \mzsc$, for all $u\in \llbracket 1, \mzsc \rrbracket$, $\abs[1]{\set{i\in S : p_i> \ell_{u:\mzsc}}} \geq \mzsc-u+1$, therefore we can recursively construct $i_{\mzsc}, \dots, i_{1}$ such that $p_{i_{\mzsc -u+1}}$ is maximal for $i_{\mzsc -u+1} \in \set{i\in S : p_i> \ell_{u:\mzsc}}\setminus\set{i_{\mzsc },\dots,i_{\mzsc -u+2}}$. Thus, $A = \set{i_1,\dots,i_{\mzsc}}$ is such that $A \in \cN_\varphi$ and $|A\cap S| = \mzsc$. 
  
Assume now that $\min\limits_{1\leq u\leq \mzsc}\set{\abs[1]{\set{i\in S : p_i> \ell_{u:\mzsc}}}+u-1} \wedge \mzsc < \mzsc$, and let $u_0$ be minimal such that $v_0 =\abs[1]{\set{i\in S : p_i> \ell_{u_0:\mzsc}}}+u_0-1$ is equal to \[\min\limits_{1\leq u\leq \mzsc}\set{\abs[1]{\set{i\in S : p_i> \ell_{u:\mzsc}}}+u-1}.\] The idea from this part of the proof will be to construct a set $A \in \cN_\varphi$ of cardinality $\mzsc$ such that $|A\cap S| = v_0$. From minimality property of $u_0$ and $v_0$, we get: 
\begin{itemize}

\item $\forall u \leq u_0 -1, u_0 -u < \abs{S \cap R\paren{\ell_{u,\mzsc}}^c \cap R\paren{\ell_{u_0,\mzsc}}}$,
\item $\forall u \geq u_0, u - u_0 \geq\abs{S \cap R\paren{\ell_{u,\mzsc}} \cap R\paren{\ell_{u_0,\mzsc}}^c}$ $(\star)$.
\end{itemize}
By the first item, there exists indices $i_1, i_2,\dots, i_{u_0-1}\in S$ such that $p_{i_1}\leq p_{i_2}\leq \dots\leq  p_{i_{u_0-1}}$ and $\ell_{u_0,\mzsc}\geq p_{i_u} >\ell_{u,\mzsc}$ for all $u \leq u_0-1$. Let $B$ be the set of indices matching with the $\mzsc - v_0$ greatest $p$-values from $\set{i\in S^c : p_i> \ell_{u_0:\mzsc}}$ and define $A = \set{i_1,\dots, i_{u_0-1}} \cup \set{i\in S : p_i> \ell_{u_0:\mzsc}} \cup B$. Then, for $j \in \llbracket 1 ,\mzsc \rrbracket$: 
\begin{itemize}
\item If $j \leq u_0-1,$ $p_{(j:A)} = p_{i_j} > \ell_{j:\mzsc}$.
\item If $j \geq u_0$ is such that $R\paren{p_{(j:A)}}^c \subseteq A$ 
, then let $k=\abs{R\paren{p_{(j:A)}}^c}\leq \mzsc -j$, and by definition of $\mzsc $, $p_{(j:A)}  = p_{\brac{k+1}}> \ell_{\mzsc - k: \mzsc}\geq \ell_{j: \mzsc}$.
\item Else $\set{i\in A : \ell_{u_0:\mzsc} < p_i \leq p_{(j:A)}} \subseteq S$, because if $i \in \set{i\in A : \ell_{u_0:\mzsc} < p_i \leq p_{(j:A)}}$ verifies $i\in B$ although $R\paren{p_{(j:A)}}^c \not\subseteq A$, there exists $i' \in A^c$ such that $p_i< p_{i'}$ which contradicts $B$'s definition. Thus $j - u_0 \geq \abs[2]{\set{i\in A : \ell_{u_0:\mzsc} < p_i \leq \ell_{j:\mzsc} }}$ by $(\star)$. Furthermore, $\abs[2]{\set{i\in A : \ell_{u_0:\mzsc} < p_i \leq p_{(j:A)}}}\geq j - u_0 +1$ because 
$\abs[2]{\set{i\in A : p_i \leq \ell_{u_0:\mzsc} }} = u_0-1$, therefore $p_{(j:A)} > \ell_{j : \mzsc}$.
\end{itemize}
Thus, $A \in \mathcal{N}^{\IP}_{\varphi}$ and, since $B\cap S = \varnothing$,  $\abs{S\cap A} = v_0= \min\limits_{1\leq u\leq \mzsc}\set{\abs[1]{\set{i\in S : p_i> \ell_{u:\mzsc}}}+u-1}$, hence \[\hat{V}^{\IP}_\varphi(S) \geq \hat{V}^{\JER}_\Rfam(S) .\]
\end{itemize}

\rem{One interesting by-product from this proof is that the inversion procedure bound is always attained for a set of cardinality $\mzsc$. }

\subsection{Proof of \autoref{SCfuncversion}}\label{proof_SCfuncversion}

If $f_n$ is nondecreasing, then the assumption from \autoref{continuoustodiscrete} is fulfilled, and the expression of $\varphi_A$ follows.\\
Assume now that $f_n$ is right continuous for all $n\in \Nm$ and $\paren{f_n}$ is nonincreasing. Then, define the generalized inverse of $f_n$ by: \[\forall t \in [0,1], f_n^{-1}(t) = \min\set{x\in \RR : f_n(x) \geq t},\] and note that $\varphi_A = \max\limits_{j \in \llbracket 1, |A| \rrbracket}\indev{ f_{n}^{-1} \paren{j}> p_{(j:A)}}$, and that $\paren[1]{f_{n}^{-1} \paren{j}}$ is nondecreasing with respect to $j$ and nonincreasing with respect to $n$. Thus, by letting $\ell_{j:n}=f_{n}^{-1} \paren{j}$, we can apply the version of \autoref{SCgoemann} suited to local tests defined by strict inequalities, and we get
\begin{align*}
\hat{V}^{\IP}_{\varphi}(S) &=\min\limits_{1\leq u\leq |S|}\set{\abs[1]{\set{i\in S : p_i \geq f^{-1}_{\hat{m}_0}\paren{u}}}+u-1}\\
&=\min\limits_{1\leq u\leq |S|}\set{\abs[1]{\set{i\in S : f_{\hat{m}_0}\paren{p_i} \geq u}}+u-1}\\
&= \min\limits_{0\leq k\leq |S|}\set{\abs[1]{\set{i\in S : f_{\hat{m}_0}\paren{p_i} \geq \floor{f_{\hat{m}_0}\paren{p_{(k : S)}}}+1}}+\floor{f_{\hat{m}_0}\paren{p_{(k : S)}}}}\\
&= \min\limits_{0\leq k\leq |S|} \psi(k)  ,\\
\end{align*} 
with 
$\hat{m}_0$ as defined in \equaref{m0perier} and $\psi(k) = \abs[1]{\set{i\in S : f_{\hat{m}_0}\paren{p_i} \geq \floor{f_{\hat{m}_0}\paren{p_{(k : S)}}}+1}}+\floor{f_{\hat{m}_0}\paren{p_{(k : S)}}}$ for all $0\leq k\leq \abs{S}$
, where the third equality holds because $\abs[1]{\set{i\in S : f_{\hat{m}_0}\paren{p_i} \geq u}}+u-1$ is nondecreasing with respect to $u$ on $\left\llbracket  \floor{f_{\hat{m}_0}\paren{p_{(k : S)}}} +1, \floor{f_{\hat{m}_0}\paren{p_{(k +1 : S)}}} \right\rrbracket$ when $ \floor{f_{\hat{m}_0}\paren{p_{(k : S)}}}\not =  \floor{f_{\hat{m}_0}\paren{p_{(k+1 : S)}}}$.

We now finally prove \equaref{eq_VIP_top_k_view}.
For all $k \in \llbracket 0, \abs{S}\rrbracket$, let 
\[k^*=\max\set{\ell \in \llbracket 0, \abs{S}\rrbracket: \floor{f_{\hat{m}_0}\paren{p_{(\ell : S)}}} = \floor{f_{\hat{m}_0}\paren{p_{(k : S)}}}}.\] 
Notably, $k^*\geq k$ and $\psi(k^*)=\psi(k)$. By construction and by $f_{\hat{m}_0}$ being nondecreasing, for all $\ell \in \llbracket 0, \abs{S}\rrbracket$,
\begin{align*}
    \ell > k^* &\Leftrightarrow  \floor{f_{\hat{m}_0}\paren{p_{(\ell : S)}}} \geq \floor{f_{\hat{m}_0}\paren{p_{(k : S)}}} +1\\
    &\Leftrightarrow f_{\hat{m}_0}\paren{p_{(\ell : S)}} \geq \floor{f_{\hat{m}_0}\paren{p_{(k : S)}}} +1.
\end{align*}
So, 
\begin{align*}
\set{i\in S:f_{\hat{m}_0}\paren{p_i} \geq \floor{f_{\hat{m}_0}\paren{p_{(k : S)}}}+1}&=\set{i\in S: p_i=p_{(\ell:S)} \text{ for some }\ell >k^*}\\
&=\set{i\in S: p_i>p_{(k^*:S)} }\\
&= S\setminus R\paren{p_{(k^*:S)}},
\end{align*}
where the second inequality holds because if $p_i>p_{(k^*:S)}$, let $\ell$ such that $p_i=p_{(\ell:S)}$, necessarily $\ell > k^*$, and reciprocally if $p_i=p_{(\ell:S)} $ with $\ell > k^*$, then $\floor{f_{\hat{m}_0}\paren{p_{(\ell : S)}}} > \floor{f_{\hat{m}_0}\paren{p_{(k : S)}}} =\floor{f_{\hat{m}_0}\paren{p_{(k^* : S)}}} $ and so  $p_i>p_{(k^* : S)}$ again by monotonicity.\\
Thus,
\begin{align*}
\psi(k)&=\psi\paren{k^*}\\
&=\floor{f_{\hat{m}_0}\paren{p_{(k^* : S)}}} + \abs{S\setminus R\paren{p_{(k^*:S)}}}\\
&\leq \floor{f_{\hat{m}_0}\paren{p_{(k : S)}}} + \abs{S\setminus R\paren{p_{(k:S)}}},
\end{align*}
where the last inequality holds because $k^*\geq k$, hence $S\setminus R\paren{p_{(k^*:S)}}\subseteq S\setminus R\paren{p_{(k:S)}}$.

Finally, let $k_0\in \argmin\limits_{0\leq k\leq |S|} \psi(k)$, for all $k \in \llbracket 0, \abs{S}\rrbracket$ we have:
\begin{align*}
\floor{f_{\hat{m}_0}\paren{p_{(k_0^* : S)}}} + \abs{S\setminus R\paren{p_{(k_0^*:S)}}}&=\psi(k_0^*)=\psi(k_0) = \hat{V}^{\IP}_{\varphi}(S)\\
&\leq \psi(k)\leq \floor{f_{\hat{m}_0}\paren{p_{(k : S)}}} + \abs{S\setminus R\paren{p_{(k:S)}}},
\end{align*}
which entails \equaref{eq_VIP_top_k_view}.

\reml{The proposition and its proof prove that local tests of the form \eqref{eq:topklocaltests} can be rewritten in the form \eqref{locform} by letting $\ell_{i:n}=f^{-1}_{n}(i) $. Furthermore, the $\hat m_0$ given by \equaref{m0perier} is then exactly the $\hat m_0$ given by \equaref{m0goemann}.}{rem_topk_local_test_to_homo_simes}

\subsection{Proofs of \autoref{DKWestimator} and \autoref{DKWtopkopt}}\label{proof_DKWtopkopt}

For all $n \in \Nm$ $f_n$ is nondecreasing and $\paren{f_n}_{n\in \Nm}$ is nonincreasing.
Applying \autoref{SCfuncversion} and \autoref{scgoemanproof2}, for all $S\subseteq\Nm$, \[\hat{V}^{\IP}_\varphi\paren{S} = \min\limits_{0\leq k\leq |S|}\paren[a]{\floor{f_{\mzsc}\paren{p_{(k : S)}}} + \abs{S\setminus R\paren{p_{(k:S)}}}} ,\] with \[\mzsc  = \max \set[2]{n\in \llbracket0,m\rrbracket : \forall t \in [0,1], f_{n}\paren{t}\geq n+i(t)-m}. \]
For all $n \in \NM $, all $t\in [0,1]$ such that $\varphi_{\Nm,n,t} = \indev{f_n(t) < n + i(t) -m }= 1$, \begin{align*}
 f_{n+1}\paren{t} &= \paren{n+1}t + \sqrt{n+1}\lambda_\alpha \\
 &= f_n(t) + t + \paren{\sqrt{n+1}-\sqrt{n}}\lambda_\alpha\\
 &< n + i(t)-m + 1 + \frac{i(t)}{n} - \frac{m}{n}-\paren{\sqrt{n}+\frac{1}{\sqrt{n}}-\sqrt{n+1}}\lambda_\alpha\\
 &< n+1+i(t)-m,
\end{align*}
thus $\varphi_{\Nm,n+1,t} =1$. Applying \autoref{SC1eqSC2}, \[\mzsc =\min_{t\in [0,1]} \max \set[2]{n\in \llbracket0,m\rrbracket : f_{n}\paren{t}\geq n +i(t) -m},\] and the expression \eqref{eq:DKWhatm0nonlocal} follows by studying the roots of a second-degree polynomial in $\sqrt{n}$.

\subsection{Proof of \autoref{bretin}}\label{proof_bretin}

Let $\bar{J}:x\mapsto\frac1N\sum_{i=1}^NJ_i(x)$. We first prove that, for all $\lambda \geq 0$, 
\begin{equation}\label{eq_bretagnolle}
\Pro{\sqrt{N}\sup\limits_{t\in\mathbb{R}}\paren{\frac1N\sum_{i=1}^N\indev{X_i\leq t}- \bar{J}(t)}> \lambda}\leq e \cdot \exp\left(-2\lambda^2\right).
\end{equation}
From \citet{shorack2009empirical}, we know that \eqref{eq_bretagnolle} holds, with $ce \cdot \exp\left(-2\lambda^2\right)$ as the right-hand side of the inequation, for any constant $c$ such that for all $\nu \geq 0$, 
\begin{equation}\label{eq_prerequisite_bretagnolle}
\Pro{\sqrt{N}\sup\limits_{t\in [0,1]}\paren{\frac1N\sum_{i=1}^N\indev{U_i\leq t} - t}> \nu}\leq c\exp\left(-2\nu^2\right),
\end{equation}
with $U_1,\dots,U_N$ independent, uniform, random variables. \citet{reeve2024short} established
that~\eqref{eq_prerequisite_bretagnolle} holds with $c=1$.
This gives \eqref{eq_bretagnolle}. Equation \eqref{eq_bretagnolle2} follows immediately from \eqref{eq_bretagnolle} by a simple inclusion of events.

Note that, compared to \eqref{eq_bretagnolle}, \eqref{eq_bretagnolle2} is suited to the general case of \autoref{ass_hetero}, where the exact cdf of the $p$-values under the null is not necessarily known. 

\rem{
\citet{massart1990tight} showed that \eqref{eq_prerequisite_bretagnolle} holds with $c=1$ but only for $\nu\geq \sqrt{\frac{\log 2}{2}}$, which does not allow the proof of \citet{shorack2009empirical} to use $c=1$ and get \eqref{eq_bretagnolle}, as that proof really requires \eqref{eq_prerequisite_bretagnolle} to be true for all $\nu\geq0$. Prior to the work of \citet{reeve2024short}, we could have used $c=2$, as \eqref{eq_prerequisite_bretagnolle} holds for $c=2$ and any $\nu\geq0$, by the two-sided version of the DKW inequality proven by \citet{massart1990tight}. Hence, we divide the upper bound on the probability by a factor 2 thanks to very recent works.
}

\subsection{Proof of \autoref{cor_local_bretagnolle}}\label{proof_cor_local_bretagnolle}

Let $\mu\in H_A$. Equivalently, $A\subseteq\cH_0$, so for each $i\in A$, the cdf of $p_i$ is bounded by $F_i$, by \autoref{ass_hetero}. Applying \eqref{eq_bretagnolle2} to $(p_i)_{i\in A}$ and $\tilde{\lambda}_\alpha$, it comes that:
\begin{align*}
\alpha=e \cdot \exp\left(-2\tilde{\lambda}_\alpha^2\right)&\geq \Proo{\sqrt{|A|}\sup\limits_{t\in\mathbb{R}}\paren{\frac1{|A|}\sum_{i\in A}\indev{p_i\leq t}- \frac1{|A|}\sum_{i\in A}F_i(t)}> \tilde{\lambda}_\alpha}{\mu}\\
&=\Proo{\sup\limits_{t\in\mathbb{R}}\paren{\sum_{i\in A}\indev{p_i\leq t}- \sum_{i\in A}F_i(t)}> \tilde{\lambda}_\alpha\sqrt{|A|}}{\mu}\\
&=\Proo{\exists{t\in\mathbb{R}}:{\sum\limits_{i\in A}\mathbb{G}_i(t)}> \tilde{\lambda}_\alpha\sqrt{|A|}}{\mu}\\
&=\Proo{\varphi_A=1}{\mu}.
\end{align*}
This proves that $\varphi_A$ is indeed a local test at level $\alpha$.

\subsection{Proof of \autoref{Bretmaxmin}}\label{proof_Bretmaxmin}

Let $t\in (0,1]$. Then to minimize $\varphi_{A,t}$ with respect to $A \subseteq \Nm$ such that $ |A\cap S| = n$, we let $|A| = n + j$ and then minimize  $\sum\limits_{k\in A}\mathbb{G}_{k}(t)= \sum\limits_{k\in A\cap S}\mathbb{G}_{k}(t)+\sum\limits_{k\in A\cap S^c}\mathbb{G}_{k}(t)$. The first sum is lower bounded by $\sum\limits_{1\leq k\leq n}\paren[1]{\mathbb{G}(t)}_{(k:S)}$, while the second one is lower bounded by $\sum\limits_{1\leq k\leq j}\paren[1]{\mathbb{G}(t)}_{(k:S^c)}$.

\subsection{Proof of \autoref{adaptivebretin}}\label{proof_adaptivebretin}

By assumption, \[\set[3]{\forall t \in [0,1],  m_0^{-1/2}  \sum_{i \in \cH_0 } \mathbb{G}_i(t) \leq \tilde{\lambda}_\alpha}  \subseteq \set[3]{m_0 \leq \hat{m}_0},\] 
thus, 
\begin{align*}
&\set[3]{\forall t \in [0,1],  m_0^{-1/2}  \sum_{i \in \cH_0 } \mathbb{G}_i(t) \leq \tilde{\lambda}_\alpha} \\
=& \set[3]{\forall t \in [0,1],   \sum_{i \in \cH_0 } \mathds{1}_{p_i \leq t} \leq \sum_{i \in \cH_0 }F_i(t) + \sqrt{m_0}\tilde{\lambda}_\alpha}\bigcap \set[3]{m_0 \leq \hat{m}_0}\\
\subseteq & \set[3]{\forall t \in [0,1],   \sum_{i \in \cH_0 } \mathds{1}_{p_i \leq t} \leq  \sum_{1\le i \le \hat{m}_0 }\paren[1]{\mathbf{F}(t)}_{[i]} + \sqrt{\hat{m}_0}\tilde{\lambda}_\alpha}.
\end{align*}
Then, as $\tilde{\lambda}_\alpha$ verifies the equality $e\cdot\exp\paren{-2\tilde\lambda_{ \alpha}^2} = \alpha$, it follows from \autoref{bretin} that 
\begin{align*}
\Pro{\forall t \in [0,1],   \sum_{i \in \cH_0 } \mathds{1}_{p_i \leq t} \leq  \sum_{1\le i \le \hat{m}_0 }\paren[1]{\mathbf{F}(t)}_{[i]} + \sqrt{\hat{m}_0}\tilde{\lambda}_\alpha} &\geq \Pro{\forall t \in [0,1],  m_0^{-1/2}  \sum_{i \in \cH_0 } \mathbb{G}_i(t) \leq \tilde{\lambda}_\alpha}\\
& \geq 1 -\alpha.
\end{align*}

\subsection{Proof of \autoref{examplesbest}}\label{proof_examplesbest}

Let $U\subseteq V$. The threshold family $\tau$ satisfies \eqref{C1} with respect to $V$, thus, for all $k\in \llbracket 1, |U| \rrbracket $ and $A\subseteq V$ such that $|A| = |V| -k +1$, $\sum\limits_{i \in A} \mathbf{H}\paren{\tau_{|V|},\tau_k} \leq k \alpha$, therefore 
\begin{align*}
\sum_{1\leq i\leq |U|-k+1}\paren[2]{\mathbf{H}\paren{\tau_{|U|},\tau_k}}_{\brac{i:U}} &\leq  \sum_{1\leq i\leq |U|-k+1}\paren[2]{\mathbf{H}\paren{\tau_{|V|},\tau_k}}_{\brac{i:U}} \\
& \leq \sum_{1\leq i\leq |V|-k+1}\paren[2]{\mathbf{H}\paren{\tau_{|V|},\tau_k}}_{\brac{i:V}} \\
& \leq k \alpha,
\end{align*} 
where the first inequality comes because $\tau $ is a nondecreasing family of thresholds and $\mathbf{H}$ is coordinate-wise non decreasing on its first parameter. The second inequality comes because all the terms of the smaller sum are terms of the greater sum. Thus $\tau_k \in \set[2]{t\in \cA, t\leq \lambda : \paren[1]{\mathbf{H}(\lambda,t)}_{\brac{1:U}} + \dots + \paren[1]{\mathbf{H}(\lambda,t)}_{\brac{|U|-k+1:U} }\leq k\alpha}$
therefore $\tau_k \leq \tilde{\tau}_k$.\\
Note that the previous sequence of inequalities with $k = |U|$ grants that $\tau_{|U|}$ verifies \eqref{eq_lambda_U}.
Thus $\tilde{\tau}$ satisfies \eqref{C1} at level $\alpha$ with respect to $U$, and for all $k\in \Nm$, $\tau_k\leq \tilde{\tau}_k$, therefore $\tau$ satisfies \eqref{C1} at level $\alpha$ with respect to $U$.

\subsection{Proof of \autoref{DDRlocaltest}}\label{proof_DDRlocaltest}

Let $\tilde{\tau} $ be defined by \autoref{thresholds_ddr2} with $U =A$. Then, $\tilde{\tau}$ satisfies \eqref{C1} at level $\alpha$ with respect to $A$, thus
by \autoref{thm_hetero_simes}, $\mathbb{P}\paren{ \exists i \in \llbracket 1,|A| \rrbracket \text{, } p_{(i : A)} \leq \tilde{\tau}_i} \leq \alpha$, and \eqref{localDDRorderstat} arises from the following equality:
\[\set{p_{(i:A)}\leq \tilde{\tau}_i} = \set{\sum_{\ell = 1}^{ k-i+1} \paren[2]{\mathbf{H}\paren[1]{\lambda,p_{(i:A)}}}_{[\ell:A]} \leq i\alpha}\bigcap\set{p_{(i:A)}\leq \lambda}, \]
thus $\varphi_A$ is indeed a local test of level $\alpha$.\\
We now prove that expressions \eqref{localDDRorderstat} and \eqref{localDDRt} are equal. Suppose that expression \eqref{localDDRorderstat} is equal to 1. Then, there exists $i\in \Nk$ such that $\sum\limits_{\ell = 1}^{ k-i+1} \paren[2]{\mathbf{H}\paren[1]{\lambda,p_{(i:A)}}}_{[\ell:A]} \leq i\alpha$ and $p_{(i:A)}\leq \lambda$.\\
Let $j = i\paren{p_{(i:A)}}.$ Then $p_{(j:A)} = p_{(i:A)}\leq \lambda$, $j\geq i$, thus $\sum\limits^{k-j+1}_{\ell=1}\paren[2]{\mathbf{H}\paren[1]{\lambda,p_{(j:A)}}}_{[\ell:A]} \leq j\alpha.$ \\
Thus, expression \eqref{localDDRt} is equal to 1, with $t=p_{(j:A)}$.\\
Assume now that expression \eqref{localDDRt} is equal to 1, meaning there exists $t\leq \lambda$ such that $\sum\limits^{k-i_A(t)+1}_{\ell=1}\paren[1]{\mathbf{H}\paren[0]{\lambda,t}}_{[\ell:A]} \leq i_A(t)\alpha$. Then, because $\mathbf{H}\paren[1]{\lambda,\cdot}$ is coordinate-wise nondecreasing and $p_{\paren{i_A(t):A}}\leq t\leq \lambda$, $\sum\limits^{k-i_A(t)+1}_{\ell=1}\paren[2]{\mathbf{H}\paren[1]{\lambda,p_{\paren{i_A(t):A}}}}_{[\ell:A]} \leq i_A(t)\alpha$, thus expression \eqref{localDDRorderstat} is equal to 1.

\subsection{Proof of \autoref{DDRminphi}}\label{proof_DDRminphi}

We show that if $A\subseteq \Nm$ is such that $|A\cap S| = n$ and $\varphi_{A,t} = 0$, then $B = A \cup \paren{R(t)^c\cap S^c}$ is such that $\varphi_{B,t} = 0$.
Assume that $A^c\cap R(t)^c\cap S^c \not = \varnothing$ and let $j\in A^c\cap R(t)^c\cap S^c$ and $\tilde{A} = A\cup\{j\}$. We have $\varphi_{A,t} = 0$, thus \[\displaystyle\sum ^{|A|-i_A(t)+1}_{k=1}\paren[1]{\mathbf{H}\paren[a]{\lambda,t}}_{[k:A]} > i_A(t)\alpha,\] therefore  \[\displaystyle\sum ^{|\tilde{A}|-i_{\tilde{A}}(t)+1}_{k=1}\paren[1]{\mathbf{H}\paren[a]{\lambda,t}}_{[k:\tilde{A}]} > i_{\tilde{A}}(t)\alpha\]  because $i_{\tilde{A}}(t) = i_{A}(t)$ and $|\tilde{A}| = |A| +1$. \\
Thus $\varphi_{S,n,t}$ is achieve on a set $A$ such that $R(t)^c\cap S^c\subseteq A$. \\
We now assume that $A\cap R(t)\cap S^c \not = \varnothing$ and let $j\in A\cap R(t)\cap S^c$ and $\tilde{A} = A\setminus\{j\}$. 
We have $\varphi_{A,t} = 0$, thus \[\displaystyle\sum ^{|A|-i_A(t)+1}_{k=1}\paren[1]{\mathbf{H}\paren[a]{\lambda,t}}_{[k:A]} > i_A(t)\alpha,\] therefore,  \[\displaystyle\sum ^{|\tilde{A}|-i_{\tilde{A}}(t)+1}_{k=1}\paren[1]{\mathbf{H}\paren[a]{\lambda,t}}_{[k:\tilde{A}]} > i_{\tilde{A}}(t)\alpha,\] because $i_{\tilde{A}}(t) = i_{A}(t) - 1$ and all the two sums are on the same subset of $A$ except possibly for $j$. In this case, the distance between the last term of the second sum and $\mathbf{H}_j\paren[a]{\lambda,t}$ is at most $\alpha$ by definition of $\lambda$. Thus $\varphi_{S,n,t}$ is achieve on a set $A$ such that $ A \cap R(t)\cap S^c = \varnothing$. \\
Note that by similar arguments considering $\tilde{A} = A\cup \set {j_2}\setminus\set{j_1}$ with $j_1 \in A\cap S\cap R(t)$ and $j_2\in A^c \cap S \cap R(t)^c$, 
$\varphi_{S,n,t}$ is achieve on a set $A$ such that 
$R(t)^c\cap S\subseteq A$. \\
To summarize, $\varphi_{S,n,t}$ is attained on a set $A$ such that $R(t)^c\subseteq A$ and $ A \cap R(t)\cap S^c = \varnothing$. \\
Thus, defining $k_S(t) = \paren{ n - \abs{R(t)^c \cap S}}\vee 0$ and $S(t) = R(t)\cap S$, \[\varphi_{S,n,t} = \indev{\displaystyle\sum ^{m - i(t)   +1}_{k=1}\paren[2]{\mathbf{H}_{R(t)^c}\paren[a]{\lambda,t},\paren[1]{\mathbf{H}\paren[a]{\lambda,t}}_{[1:S(t)]}, \dots,\paren[1]{\mathbf{H}\paren[a]{\lambda,t}}_{\brac{k_S(t):S(t) }} }_{\brac{k}} \leq k_S(t)\alpha}.\]

\subsection{Proof of \autoref{SC1betterJER}}\label{proof_SC1betterJER}

The conditions on $\alpha$ and $\varepsilon$ are merely assumptions implying positive $a_1<1$ and $a_2<1$. \\
We first show that $\hat{m}_0 = \max\set{n \in \llbracket 0,4 \rrbracket : \forall i \in \llbracket 1,4\rrbracket,  \sum\limits_{1\leq j\leq n}\paren[1]{\mathbb{G}(p_i)}_{(j)}\leq\tilde{\lambda}_\alpha\sqrt{n}}  = 3$.

\begin{align}
3 - \sum\limits_{1\leq j\leq 4}\mathbf{F}_j(p_3) &= 3 - a_1- 2a_2 = 3 - \paren{\frac{2 - \tilde{\lambda}_\alpha\sqrt{3}}{2}-\varepsilon} - \paren{2 - \tilde{\lambda}_\alpha\sqrt{3}} \nonumber\\
&= \tilde{\lambda}_\alpha \paren{\frac{3\sqrt{3}}{2}- \sqrt{4}} + \sqrt{4}\tilde{\lambda}_\alpha + \varepsilon \nonumber\\
&> \sqrt{4} \tilde{\lambda}_\alpha.\label{com_hatm0}
\end{align}
The inequality $\sum\limits_{1\leq j\leq 3}\paren[1]{\mathbb{G}(p_i)}_{(j)}\leq\tilde{\lambda}_\alpha\sqrt{3}$ is straightforward for $i \in \set{1,4}$, because the left-hand side of this inequality is non-positive.
For $i \in \set{2,3}$, $\sum\limits_{1\leq j\leq 3}\paren[1]{\mathbb{G}(p_i)}_{(j)} = 2 - 2a_2 = \tilde{\lambda}_\alpha\sqrt{3} $.
Thus, with $\zeta_{1} = 1$, $\zeta_{2} = 2 \wedge \left\lfloor a_1+2a_2 + \sqrt{3}\tilde{\lambda}_\alpha\right\rfloor = 2  $, $\zeta_3 \geq 2$ and $\zeta_4 = 3$ defined by \eqref{zetabretadaptive}, we get $\hat{V}^{\JER}_{\Rfam}\paren[1]{\set{1,2,4}} = 3$.\\
We now show that $\hat{V}^{\SCo}_{\varphi}\paren[1]{\set{1,2,4}} \leq 2$. Recall that, from \autoref{BretSC1}, \[\scalebox{0.87}{$\hat{V}^{\SCo}_\varphi\paren{\set{1,2,4}} = \max\set{n \in \llbracket 0, 3 \rrbracket : \forall i \in \llbracket 1, 4 \rrbracket,\exists j\leq 1, \sum\limits_{1\leq k\leq n}\paren[1]{\mathbb{G}(p_i)}_{(k:\set{1,2,4})}+ \sum\limits_{1\leq k\leq j}\mathbb{G}_3(p_i)\leq\tilde{\lambda}_\alpha\sqrt{n+j}}.$}\]
Thus, if we test the condition for $n = 3$ in $p_3$,
\begin{align*}
\sum\limits_{k \in \set{1,2,4}}\mathbb{G}_k(p_3) &= 2 - F_1\paren{p_3} - F_2\paren{p_3} - F_4\paren{p_3} \\
&= 2 - a_1 - a_2\\
&= \sqrt{3}\tilde{\lambda}_\alpha + \varepsilon > \sqrt{3}\tilde{\lambda}_\alpha,
\end{align*}
and, 
\begin{align*}
\sum\limits_{k \in\set{1,2,4}}\mathbb{G}_k(p_3) + \mathbb{G}_3(p_3) &= 3 - \sum\limits_{1\leq j\leq 4}\mathbf{F}_j(p_3) \\
&> \sqrt{4} \tilde{\lambda}_\alpha,
\end{align*}
from \eqref{com_hatm0}. Therefore, $\hat{V}^{\SCo}_\varphi\paren[1]{\set{1,2,4}} \leq 2$.

\subsection{Proof of \autoref{jerbetterhom}}\label{proof_jerbetterhom}

The expression of  $\varphi_A^{hom}$ as a local test of level $\alpha$ arises from upper bounding $\sum\limits_{ i \in A} \mathbf{F}_i\paren{t}$ by $\sum\limits_{1 \leq i \leq |A|}\paren[1]{\mathbf{F}\paren{t}}_{[i]}$.\\
Furthermore, $\varphi_A^{hom}$ has the expression \eqref{eq:topklocaltests} and verifies that each individual $f_n$ is nondecreasing for all $n\in \Nm$ and $\paren{f_n}_{n\in \Nm}$ is nondecreasing as a sequence of functions, thus, by \autoref{SCfuncversion},
\begin{align*}
\hat{V}^{\IP}_{\varphi^{hom}}(S) &= \min\limits_{0\leq k\leq |S|}\paren[2]{\floor{f_{\hat{m}_0^{hom}}\paren{p_{(k : S)}}} + \abs{S\setminus R\paren{p_{(k:S)}}}} 
\end{align*}
where  \begin{equation*}
\hat{m}_0^{hom} = \max \set[2]{s\in \llbracket0,m\rrbracket : \forall i \in \llbracket 1, s\rrbracket, f_{s}\paren{p_{\brac{i}}}\geq s-i+1}.
\end{equation*}
However, if $\mzsc$ is defined by \eqref{bretm0}, we have $\mzsc \leq \hat{m}_0^{hom}$ because $\varphi_{\Nm, n,t}^{hom} \leq \varphi_{\Nm, n,t} $. Therefore, for all $S\subseteq \Nm$, $\hat{V}^{\JER}_\Rfam (S)\leq \hat{V}^{\IP}_{\varphi^{hom}}(S)$ because \[\hat{V}^{\JER}_\Rfam (S) = \min\limits_{0\leq k\leq |S|}\paren[2]{\floor{f_{\mzsc}\paren{p_{(k : S)}}} + \abs{S\setminus R\paren{p_{(k:S)}}}}\wedge \mzsc .\]

\subsection{Proof of \autoref{jerbettersimes} and \autoref{remhomsimes}}\label{proof_jerbettersimes}

Note that $\hat{m}_0^{hom}$ is more conservative than the over estimator $\mzsc$ defined by \eqref{DDRm0}, because $\hat{m}_0^{hom} = \hat{V}_{\varphi^{hom}}^{\SCo}\paren[1]{\Nm}$ and for all $A \subseteq \Nm$ and $t \in [0,1]$, $\varphi_{A,t} \geq \varphi^{hom}_{A,t}$ hence $\varphi_{\Nm,n,t} \geq \varphi^{hom}_{\Nm,n,t}$ for all $n\in \NM$. Therefore, $\hat{V}^{\JER}_\Rfam \leq \hat{V}^{\SCo}_{\varphi^{hom}}$.

We then compute the expression of $\hat{m}_0^{hom}$. Note that $\ell_{i+1:n+1} =  \mathrm{max}\set[1]{t\in \cA, t\leq \lambda : \paren[1]{\mathbf{H}(\lambda,t)}_{[1]} + \dots + \paren[1]{\mathbf{H}(\lambda,t)}_{[n -i+1] }\leq \paren{i+1}\alpha}\geq \ell_{i:n}$, thus $\hat{m}_0^{hom} = \hat{V}^{\SCt}_{\varphi^{hom}} \paren[1]{\Nm}$ by \autoref{scgoemanproof2}. \\
Let $t\in (0,1]$. Since $\varphi^{hom}_{\Nm,n,t} = \indev{t\leq \ell_{\paren{n + i(t) - m}:n}}$ for all $n\in \Nm$ by \autoref{Simeslikemin}, then $b(t) = \max\set{n \in \NM :\varphi^{hom}_{\Nm,n,t} =0} = \max\set{n \in \NM : t > \ell_{\paren{n + i(t) - m}:n} }$. Thus, by definition of $\paren{\ell_{i:n}}$, $b(t)= \max\set{n \in \NM :\sum\limits^{m-i(t)+1}_{k=1}\paren[1]{\mathbf{H}\paren[0]{\lambda,t}}_{[k]} > \paren{ n + i(t) - m}\alpha }$, and the expression of $\hat{m}_0^{hom}$ follows. \\

\end{document}